%% file: main.tex
\documentclass[11pt,leqno]{article}
\input{preamble}

\title{A motivic filtration on the topological cyclic homology of commutative ring spectra}
 
\author{Jeremy Hahn, Arpon Raksit, and Dylan Wilson}

\date{}

\bibliography{Bibliography.bib}

\begin{document}

\maketitle

\begin{abstract}
  For a prime number $p$ and a $p$-quasisyntomic commutative ring $R$, Bhatt--Morrow--Scholze defined motivic filtrations on the $p$-completions of $\THH(R), \TC^{-}(R), \TP(R),$ and $\TC(R)$, with the associated graded objects for $\TP(R)$ and $\TC(R)$ recovering the prismatic and syntomic cohomology of $R$, respectively. We give an alternate construction of these filtrations that applies also when $R$ is a well-behaved commutative ring spectrum; for example, we can take $R$ to be $\mathbb{S}$, $\mathrm{MU}$, $\mathrm{ku}$, $\mathrm{ko}$, or $\mathrm{tmf}$.  We compute the mod $(p,v_1)$ syntomic cohomology of the Adams summand $\ell$ and observe that, when $p \ge 3$, the motivic spectral sequence for $V(1)_*\mathrm{TC}(\ell)$ collapses at the $\mathrm{E}_2$-page.
\end{abstract}

{\small
\setcounter{tocdepth}{1}
\tableofcontents
\vspace{3.0ex}
}

\section{Introduction}
\input{PossibleIntro.tex}

\section{The even filtration} \label{SecEvenFiltration}
\input{EvenFiltration.tex}

\section{Wilson spaces} \label{WilsonAppendix}
\input{Wilson.tex}

\input{MotivicFiltration.tex} \label{SecMotivicFiltration}

\input{ComparisonTheorems.tex} \label{SecComparison}

\section{The mod $(p,v_1)$ syntomic cohomology of the Adams summand} \label{SecAdamsSummand}
\input{AdamsSummand.tex}

\begin{appendix}
  \section{The even filtration for modules}
  \label{mod}
  \input{ModuleAppendix}
\end{appendix}

\printbibliography

\end{document}

%% file: preamble.tex
\usepackage[%
  tmargin=1in,bmargin=1in,%
  lmargin=1.4in,rmargin=1.4in,%
  marginparsep=0.2in, marginparwidth=1.0in%
]{geometry}
\AtBeginDocument{%
  \setlength{\abovedisplayskip}{1.5ex plus 0.3ex minus 0.3ex}%
  \setlength{\abovedisplayshortskip}{1.0ex plus 0.3ex minus 0.3ex}%
  \setlength{\belowdisplayskip}{1.5ex plus 0.3ex minus 0.3ex}%
  \setlength{\belowdisplayshortskip}{1.0ex plus 0.3ex minus 0.3ex}%
}

\usepackage{datetime2}
\usepackage{fancyhdr}
\usepackage[hyphens]{url}
\usepackage[style=alphabetic,maxbibnames=99,backend=bibtex8]{biblatex}
\usepackage{titlesec}
\usepackage[dotinlabels]{titletoc}
\usepackage{appendix}
\usepackage[letterspace=50]{microtype}
\usepackage{stmaryrd}
\usepackage[bottom]{footmisc}
\usepackage[inline]{enumitem}
\usepackage{xspace}
\usepackage{calc}
\usepackage{etoolbox}
\usepackage{ifthen}
\usepackage{tikz}
\usepackage{tikz-cd}
\usepackage{textcase}
\usepackage{spectralsequences}

\definecolor{cite}{rgb}{0.2,0.5,0.5}
\definecolor{link}{rgb}{0.2,0.5,0.5}
\definecolor{url}{rgb}{0.4,0.4,0.4}
\usepackage[%
  hyperfootnotes=false,
  colorlinks,%
  linkcolor=link,%
  citecolor=cite,%
  urlcolor=url,%
]{hyperref}
\DeclareCiteCommand{\citebare}
  {\usebibmacro{prenote}}
  {\usebibmacro{citeindex}%
   \usebibmacro{cite}}
  {\multicitedelim}
  {\usebibmacro{postnote}}

\renewbibmacro{in:}{}
\DeclareFieldFormat*{title}{\mkbibitalic{#1}}
\DeclareFieldFormat*{journaltitle}{#1\isdot}
\DeclareFieldFormat*[online,misc]{date}{(#1)}
\DeclareFieldInputHandler{pages}{}

\newbibmacro*{addendum+pubstate+date}{%
  \printfield{addendum}%
  \newunit\newblock
  \printfield{pubstate}
  \usebibmacro{date}}

\newbibmacro*{doi+eprint+url+date}{%
  \iftoggle{bbx:doi}
  {\printfield{doi}}
  {}%
  \newunit\newblock
  \iftoggle{bbx:eprint}
  {\usebibmacro{eprint}}
  {}%
  \iftoggle{bbx:url}
  {\usebibmacro{url+urldate}}
  {}
  \usebibmacro{date}}

\DeclareBibliographyDriver{online}{%
  \usebibmacro{bibindex}%
  \usebibmacro{begentry}%
  \usebibmacro{author/editor+others/translator+others}%
  \setunit{\printdelim{nametitledelim}}\newblock
  \usebibmacro{title}%
  \newunit
  \printlist{language}%
  \newunit\newblock
  \usebibmacro{byauthor}%
  \newunit\newblock
  \usebibmacro{byeditor+others}%
  \newunit\newblock
  \printfield{version}%
  \newunit
  \printfield{note}%
  \newunit\newblock
  \printlist{organization}%
  \newunit\newblock
  \usebibmacro{doi+eprint+url+date}%
  \newunit\newblock
  \usebibmacro{addendum+pubstate}%
  \setunit{\bibpagerefpunct}\newblock
  \usebibmacro{pageref}%
  \newunit\newblock
  \iftoggle{bbx:related}
    {\usebibmacro{related:init}%
     \usebibmacro{related}}
    {}%
  \usebibmacro{finentry}}

\DeclareBibliographyDriver{misc}{%
  \usebibmacro{bibindex}%
  \usebibmacro{begentry}%
  \usebibmacro{author/editor+others/translator+others}%
  \setunit{\printdelim{nametitledelim}}\newblock
  \usebibmacro{title}%
  \newunit
  \printlist{language}%
  \newunit\newblock
  \usebibmacro{byauthor}%
  \newunit\newblock
  \usebibmacro{byeditor+others}%
  \newunit\newblock
  \printfield{howpublished}%
  \newunit\newblock
  \printfield{type}%
  \newunit
  \printfield{version}%
  \newunit
  \printfield{note}%
  \newunit\newblock
  \usebibmacro{doi+eprint+url}%
  \newunit\newblock
  \usebibmacro{addendum+pubstate+date}%
  \setunit{\bibpagerefpunct}\newblock
  \usebibmacro{pageref}%
  \newunit\newblock
  \iftoggle{bbx:related}
  {\usebibmacro{related:init}%
    \usebibmacro{related}}
  {}%
  \usebibmacro{finentry}}

\DeclareFieldFormat{eprint:ktheory}{%
  K-theory preprint archives: \href{https://faculty.math.illinois.edu/K-theory/#1}{ #1}%
}

\usepackage{amsmath,amsthm}

\usepackage[T1]{fontenc}
\usepackage[largesc]{newtxtext}
\usepackage{newtxmath}
\usepackage{MnSymbol}

\usepackage[%
  cal=cm,       calscaled=0.95,%
  bb=boondox,   bbscaled=1.0,%
  frak=euler,   frakscaled=1.05,%
]{mathalfa}  

\makeatletter
\g@addto@macro\bfseries{\boldmath}
\makeatother

\usepackage[textsize=footnotesize,backgroundcolor=orange!50]{todonotes}
\usepackage[compress,capitalize]{cleveref}

\let\theoldbibliography\thebibliography
\renewcommand{\thebibliography}[1]{%
  \theoldbibliography{#1}%
  \setlength{\parskip}{0ex}
  \setlength{\itemsep}{0.5ex plus 0.2ex minus 0.2ex}
  \small
}
\apptocmd{\thebibliography}{\raggedright}{}{}

\renewcommand{\title}[1]{\newcommand{\thetitle}{#1}}
\renewcommand{\author}[1]{\newcommand{\theauthor}{#1}}

\renewcommand{\maketitle}{%
  \begin{center}
    {\linespread{1.15}%
      \bfseries\MakeTextUppercase%
      \thetitle\par} \vspace{4.0ex}
    \footnotesize
    {\MakeUppercase \theauthor}
  \end{center}
  \vspace{3.0ex}
  \thispagestyle{fancy}
}

\renewenvironment{abstract}{\noindent\begin{center}\begin{minipage}{0.8\linewidth}\small{\scshape Abstract.}}{\end{minipage}\end{center}}

\newlength{\tagsep}
\makeatletter
\def\fullwidthdisplay{\displayindent\z@ \displaywidth\columnwidth}
\edef\@tempa{\noexpand\fullwidthdisplay\the\everydisplay}
\everydisplay\expandafter{\@tempa}
\makeatother

\titleformat{\section}{\centering}{\textsection\thesection.}{1.5\tagsep}{\scshape}
\titleformat{\subsection}[runin]{}{\fontseries{b}\selectfont\textsection\bfseries\thesubsection.}{1.5\tagsep}{\bfseries}[.]

\titlespacing*{\section}{0pt}{4ex}{\medskipamount}
\titlespacing*{\subsection}{0pt}{\bigskipamount}{0.5em}

\dottedcontents{section}[1.3em]{\vspace{0.25ex}}{1.3em}{0.7em}
\dottedcontents{subsection}[2.0em]{\vspace{0.25ex}}{2.0em}{0.7em}

\newcommand{\crefeqfmt}[1]{
  \crefformat{#1}{(##2##1##3)}
  \Crefformat{#1}{(##2##1##3)}
  \crefrangeformat{#1}{(##3##1##4--##5##2##6)}
  \Crefrangeformat{#1}{(##3##1##4--##5##2##6)}
  \crefmultiformat{#1}{(##2##1##3}{, ##2##1##3)}{, ##2##1##3}{, ##2##1##3)}
  \Crefmultiformat{#1}{(##2##1##3}{, ##2##1##3)}{, ##2##1##3}{, ##2##1##3)}
  \crefrangemultiformat{#1}{(##3##1##4--##5##2##6}{, ##3##1##4--##5##2##6)}{, ##3##1##4--##5##2##6}{, ##3##1##4--##5##2##6)}
  \Crefrangemultiformat{#1}{(##3##1##4--##5##2##6}{, ##3##1##4--##5##2##6)}{, ##3##1##4--##5##2##6}{, ##3##1##4--##5##2##6)}
}
\newcommand{\crefsecfmt}[1]{%
  \crefformat{#1}{\textsection##2##1##3}
  \Crefformat{#1}{\textsection##2##1##3}
  \crefrangeformat{#1}{\textsection\textsection##3##1##4--##5##2##6}
  \Crefrangeformat{#1}{\textsection\textsection##3##1##4--##5##2##6}
  \crefmultiformat{#1}{\textsection\textsection##2##1##3}{--##2##1##3}{, ##2##1##3}{ and~##2##1##3}
  \Crefmultiformat{#1}{\textsection\textsection##2##1##3}{--##2##1##3}{, ##2##1##3}{ and~##2##1##3}
  \crefrangemultiformat{#1}{\textsection\textsection##3##1##4--##5##2##6}{ and~##3##1##4--##5##2##6}{, ##3##1##4--##5##2##6}{ and~##3##1##4--##5##2##6}
  \Crefrangemultiformat{#1}{\textsection\textsectionXS##3##1##4--##5##2##6}{ and~##3##1##4--##5##2##6}{, ##3##1##4--##5##2##6}{ and~##3##1##4--##5##2##6}
}

\numberwithin{equation}{block}

\crefeqfmt{equation}
\crefeqfmt{enumi}
\crefeqfmt{enumii}
\crefsecfmt{section}
\crefsecfmt{subsection}
\crefsecfmt{appendix}
\crefname{part}{Part}{Parts}
\crefname{chapter}{Chapter}{Chapters}
\crefname{figure}{Figure}{Figures}

\makeatletter

\newcommand{\blocknumfont}{\bfseries}
\newcommand{\blockheadfont}{\bfseries}
\newcommand{\blocknotefont}{\normalfont}
\newcommand{\blockspecialfont}{\itshape}
\newcommand{\blockhorizspace}{0.4em}
\newcommand{\blocknotespace}{0.4em}
\newcommand{\blocknumsep}{}
\newcommand{\blocksep}{.}
\newcommand{\blockvertspace}{\medskipamount}

\newtheoremstyle{block}%
  {\blockvertspace}
  {\blockvertspace}
  {}
  {}
  {} 
  {}
  {0em}
  {\thmname{{\blockheadfont#1}}%
    \@ifnotempty{#1}{\@ifnotempty{#2}{ }}%
    \thmnumber{{\blocknumfont #2\blocknumsep}}%
    \@ifnotempty{#1#2}{{\blockheadfont\blocksep}}%
    \thmnote{\hspace{\blocknotespace}[{\blocknotefont#3}]}%
    \hspace{\blockhorizspace}%
  }

\newtheoremstyle{blockspecial}%
  {\blockvertspace}
  {\blockvertspace}
  {\blockspecialfont}
  {}
  {} 
  {}
  {0em}
  {\thmname{{\blockheadfont#1}}%
    \@ifnotempty{#1}{\@ifnotempty{#2}{ }}%
    \thmnumber{{\blocknumfont #2\blocknumsep}}%
    \@ifnotempty{#1#2}{{\blockheadfont\blocksep}}%
    \thmnote{\hspace{\blocknotespace}[{\blocknotefont#3}]}%
    \hspace{\blockhorizspace}%
  }

\newtheoremstyle{blocknamed}%
  {\blockvertspace}
  {\blockvertspace}
  {}
  {}
  {} 
  {}
  {0em}
  {\@ifempty{#3}{\thmname{\blockheadfont#1}}{\thmname{\blockheadfont#3}}%
    \@ifnotempty{#1#3}{\blockheadfont\blocksep}\hspace{\blockhorizspace}%
  }
  
\theoremstyle{blocknamed}
\newtheorem*{pf}{Proof}

\renewenvironment{proof}[1][]{%
  \pushQED{\qed}%
  \begin{pf}[#1]
}{%
  \popQED\endtrivlist\@endpefalse%
  \end{pf}
}
  
\newcommand{\defblock}[3]{%
  \theoremstyle{block}
  \newtheorem{#1}[block]{#2}%
  \Crefname{#1}{#2}{#3}
  \newtheorem*{#1*}{#2}%
}

\newcommand{\defblockspecial}[3]{%
  \theoremstyle{blockspecial}
  \newtheorem{#1}[block]{#2}%
  \Crefname{#1}{#2}{#3}
  \newtheorem*{#1*}{#2}%
}

\makeatother

\setlist{%
  parsep=0ex, listparindent=\parindent,%
  itemsep=0.75ex, topsep=0.75ex,%
  leftmargin=2.5em,%
}

\setlist[enumerate, 1]{%
  label={\upshape(\arabic*)},%
  ref={\upshape\arabic*},%
}
\setlist[enumerate, 2]{%
  leftmargin=*,%
  label={\upshape(\theenumi-\roman*)},%
  ref=\theenumi-\roman*,%
  widest=0,
}
\setlist[itemize, 1]{%
  label={--},%
}
\setlist[itemize, 2]{%
  label=--,%
}

\usetikzlibrary{calc,decorations.pathmorphing,shapes,arrows}
\tikzcdset{
  arrow style=tikz,
  diagrams={>={stealth}},
}

\newcommand{\lblto}[1]{\xrightarrow{#1}}
\newcommand{\isoto}{\lblto{\raisebox{-0.3ex}[0pt]{$\scriptstyle \sim$}}}
\newcommand{\inj}{\hookrightarrow}
\newcommand{\surj}{\twoheadrightarrow}

\newcommand{\iso}{\simeq}

\defblock{construction}{Construction}{Constructions}
\defblockspecial{corollary}{Corollary}{Corollaries}
\defblock{definition}{Definition}{Definitions}
\defblock{example}{Example}{Examples}
\defblockspecial{lemma}{Lemma}{Lemmas}
\defblock{notation}{Notation}{Notations}
\defblock{recollection}{Recollection}{Recollections}
\defblock{nothing}{}{}
\defblockspecial{proposal}{Proposal}{Proposals}
\defblockspecial{proposition}{Proposition}{Propositions}
\defblock{question}{Question}{Questions}
\defblock{remark}{Remark}{Remarks}
\defblockspecial{theorem}{Theorem}{Theorems}
\defblock{variant}{Variant}{Variants}
\defblock{warning}{Warning}{Warnings}

\numberwithin{block}{subsection}

\newcommand{\num}[1]{\mathbb{#1}}

\newcommand{\Ab}{\mathrm{Ab}}
\newcommand{\Alg}{\mathrm{Alg}}
\newcommand{\Beil}{\mathrm{B}}
\newcommand{\MBL}{\mathrm{MBL}}
\newcommand{\BMS}{\mathrm{BMS}}
\newcommand{\BP}{\mathrm{BP}}
\newcommand{\Bcir}{{\mathrm{BS}^1}}

\newcommand{\BiFilCAlg}{\mathrm{BiFilCAlg}}
\newcommand{\BiFilSpt}{\mathrm{BiFilSpt}}
\newcommand{\BiGrSpt}{\mathrm{BiGrSpt}}
\newcommand{\CAlg}{\mathrm{CAlg}}
\newcommand{\Cp}{\mathrm{C}_p}
\newcommand{\Mod}{\mathrm{Mod}}
\newcommand{\CycCAlg}{\mathrm{CycCAlg}}
\newcommand{\CycMod}{\mathrm{CycMod}}
\newcommand{\CycSpt}{\mathrm{CycSpt}}
\newcommand{\D}{\mathrm{D}}
\newcommand{\E}{\mathbb{E}}
\newcommand{\F}{\num{F}}

\newcommand{\FilCAlg}{\mathrm{FilCAlg}}
\newcommand{\FilMod}{\mathrm{FilMod}}
\newcommand{\FilSpt}{\mathrm{FilSpt}}

\newcommand{\Gr}{\mathrm{Gr}}
\newcommand{\GrAb}{\mathrm{GrAb}}

\newcommand{\GrSpt}{\mathrm{GrSpt}}
\renewcommand{\H}{\mathrm{H}}
\newcommand{\HC}{\mathrm{HC}}
\newcommand{\HH}{\mathrm{HH}}
\newcommand{\HP}{\mathrm{HP}}
\newcommand{\HKR}{\mathrm{HKR}}
\renewcommand{\L}{\mathrm{L}}
\newcommand{\LSym}{\operatorname{LSym}}
\newcommand{\Map}{\mathrm{Map}}
\newcommand{\MU}{\mathrm{MU}}
\newcommand{\MW}{\mathrm{MW}}
\newcommand{\Q}{\num{Q}}

\renewcommand{\S}{\num{S}}

\newcommand{\Spc}{\mathrm{Spc}}
\newcommand{\Spec}{\mathrm{Spec}}
\newcommand{\Spf}{\mathrm{Spf}}
\newcommand{\Spt}{\mathrm{Spt}}
\newcommand{\TC}{\mathrm{TC}}

\newcommand{\THH}{\mathrm{THH}}
\newcommand{\TP}{\mathrm{TP}}

\newcommand{\Tor}{\operatorname{Tor}}

\newcommand{\Z}{\num{Z}}

\newcommand{\alg}{\mathrm{alg}}

\newcommand{\can}{\mathrm{can}}

\newcommand{\cir}{{\mathrm{S}^1}}

\newcommand{\cofib}{\operatorname*{cofib}}
\newcommand{\colim}{\operatorname*{colim}}
\newcommand{\cpl}[1]{#1^\wedge}

\newcommand{\ev}{\mathrm{ev}}
\newcommand{\gr}{\operatorname{gr}}
\newcommand{\fib}{\operatorname*{fib}}
\newcommand{\fil}{\operatorname{fil}}

\newcommand{\h}{\mathrm{h}}

\newcommand{\mot}{\mathrm{mot}}
\newcommand{\op}{\mathrm{op}}

\newcommand{\tate}{\mathrm{t}}

\renewcommand{\llbracket}{\lsem}
\renewcommand{\rrbracket}{\rsem}

%% file: PossibleIntro.tex
Topological Hochschild homology is an invariant that to an associative ring spectrum $R$ assigns a cyclotomic spectrum $\THH(R)$, closely related to the algebraic K-theory of $R$. In \cite{BMS}, Bhatt--Morrow--Scholze 
fix a prime number $p$ and study $p$-completed topological Hochschild homology $\cpl{\THH(R)}_p$ for a class of 
discrete commutative rings $R$ that they call \emph{quasisyntomic}.
In particular, they construct a natural \emph{motivic
filtration} on $\cpl{\THH(R)}_p$ for such $R$.
After accounting for the cyclotomic structure on $\cpl{\THH(R)}_{p}$, the BMS motivic filtration allows one to define both the prismatic cohomology\footnote{Throughout this paper, the term ``prismatic cohomology'' will be used to refer to what is more properly called ``Nygaard-completed absolute prismatic cohomology.''} and syntomic cohomology of $R$.
The construction of prismatic and syntomic cohomology has in turn led to an explosion of new results in algebraic $K$-theory and mixed characteristic algebraic geometry \cite{BhattICM}.

In the ICM address \cite[Conjecture 5.1]{RognesICM}, Rognes explains his long-standing conjecture that a similar motivic filtration exists on the mod $(p,v_1)$ algebraic $\mathrm{K}$-theory of the $p$-completed Adams summand $\cpl{\ell}_p$, which is a commutative but not discrete ring spectrum.  Specifically, the calculation by Ausoni and Rognes of $\mathrm{K}^{\alg}(\cpl{\ell}_p)/(p,v_1)$ strongly suggests that there is a motivic filtration on $\THH(\ell)$ that is appropriately compatible with cyclotomic structure \cite{AusoniRognes}.  In \cite[\textsection 6.1]{HahnWilson}, the first and third authors studied a filtration on $\THH(\ell)$, and more generally on the $\THH$ of truncated Brown--Peterson spectra, that behaves much like the conjectured motivic filtration.  Our construction of this filtration was very much ad-hoc, and it was unclear in what generality it could be defined. Furthermore, while the filtration was obviously $\cir$-equivariant, it did not obviously interact well with the cyclotomic Frobenius. In other words, we were able to define the Nygaard-filtered prismatic cohomology of $\ell$, but not the syntomic cohomology of $\ell$.

Here, we offer and attempt to justify the following proposal:

\begin{proposal}
  \label{in--main-proposal}
  There is a functorial filtration defined on all $\E_\infty$-rings that we call the \emph{even filtration}.  For a class of $\E_\infty$-rings $R$ that we call \emph{chromatically quasisyntomic}, the even filtration on $\THH(R)$ is the motivic filtration on $\THH(R)$, and it is appropriately compatible with all cyclotomic structure.
\end{proposal}

The even filtration has a straightforward definition, and can be applied to $\E_\infty$ rings that do not arise from $\THH$ constructions.  In addition to recovering the BMS motivic filtration, we prove that it can be used to recover the Adams--Novikov filtration on the sphere, the ``synthetic analog'' functor $\nu$ into Morel--Voevodsky's $\mathcal{SH}(\mathbb{C})$, the Hochschild–Kostant–Rosenberg filtration, and the global motivic filtration of Morin and Bhatt--Lurie.

For chromatically quasisyntomic $\E_\infty$-rings $R$, such as $R=\mathbb{S}$, $\MU$, $\ell$, $\mathrm{ku}$, $\mathrm{ko}$, or $\mathrm{tmf}$, the even/motivic filtration on $\THH(R)$ can be calculated using explicit presentations: we find a map of cyclotomic $\E_\infty$-rings $\THH(R) \to B$ such that $\fil^n_{\mot}\THH(R) \simeq \lim_{\Delta} (\tau_{\ge 2n} (B^{\otimes_{\THH(R)} \bullet +1}))$. Readers familiar with \cite{BMS} may appreciate the analogy where $R$ is a discrete $p$-quasisyntomic ring and $B$ is $\THH$ of a quasiregular semiperfectoid cover (and everything is $p$-completed). Such presentations in particular allow us to compute $V(1)_* \TC(\ell)$ at the prime $p=3$, which was not previously known.

\subsection{The even filtration and comparison theorems}
\label{in--ev}

For motivation, let us recall how Bhatt--Morrow--Scholze defined their motivic filtration, which we denote here by $\fil^\star_\BMS \cpl{\THH(R)}_p$. First, they focus on a certain class of discrete commutative rings, which they call quasisyntomic and which from here forwards we will call $p$-complete $p$-quasisyntomic. Given a $p$-complete $p$-quasisyntomic ring $R$, they prove that there
exist another such ring $S$ and a map $R \to S$ satisfying certain conditions that in particular guarantee the following:
\begin{enumerate}
\item the canonical map $\cpl{\THH(R)}_p \to \lim_{\Delta}(\cpl{\THH(S^{\otimes_R \bullet+1})}_p)$ is an equivalence;
\item for each $n \ge 0$, the spectrum $\cpl{\THH(S^{\otimes_R n+1})}_p$ is \emph{even}, i.e. its homotopy groups are concentrated in even degrees.
\end{enumerate}
They then define $\fil^k_\BMS \cpl{\THH(R)}_p$ to be $\lim_{\Delta}(\tau_{\ge 2k}(\cpl{\THH(S^{\otimes_R \bullet+1})}_p))$, showing that this is independent of the choice of map $R \to S$ using the formalism of Grothendieck topologies.

More generally, the strategy of computing  $\THH(R)$ by realizing it as a totalization of even ring spectra has been broadly applied to great effect \cite{BMS, LiuWang,HahnWilson,NikolausKrause,Zpn,DavidLee}.  It inspires the following construction:

\begin{definition}
  \label{in--fil-ev}
  An $\E_\infty$-ring $B$ is \emph{even} if its homotopy groups $\pi_*B$ are concentrated in even degrees. For any $\E_\infty$-ring $A$, we define $\fil^n_{\ev} A$ to be the limit, over all maps of $\E_\infty$-rings $A \to B$ with $B$ even, of $\tau_{\ge 2n} B$. Together, the spectra $\fil^n_{\ev} A$ assemble to define a filtered $\E_\infty$-ring $\fil^\star_{\ev} A$.\footnote{For the definition of ``filtered $\E_\infty$-ring'', see \Cref{ConventionsSection}.} We refer to this construction as the \emph{even filtration}.
\end{definition}

\begin{remark}
  \label{in--fil-ev-rmk}
  We make the above definition precise by considering the category $\CAlg$ of $\E_\infty$-rings, the full subcategory $\CAlg^{\ev}$ of even $\E_\infty$-rings, and the category $\FilCAlg$ of filtered $\E_\infty$-rings. Then $\fil^\star_\ev(\--)$ is the right Kan extension, along the inclusion $\CAlg^{\ev} \inj \CAlg$, of the double-speed Postnikov filtration functor $\tau_{\ge 2*}:\CAlg^{\ev} \to \FilCAlg$.\footnote{The set-theoretic issues involved in this Kan extension are addressed in \cref{SecEvenFiltration}.}
\end{remark}

Suppose now that $R$ is a discrete commutative ring with bounded $p$-power torsion for all primes $p$ and such that the algebraic cotangent complex $\L^\alg_R$ has Tor-amplitude contained in $[0,1]$.  For example, $R$ might be a polynomial ring over $\mathbb{Z}$, or the quotient of a polynomial ring by a finite regular sequence.  For each prime number $p$, the $p$-completion $\cpl{R}_p$ is $p$-complete $p$-quasisyntomic, so we may speak of $\fil^\star_{\BMS}\cpl{\THH(R)}_p = \fil^\star_{\BMS}\cpl{\THH(\cpl{R}_p)}_p $.  By gluing together the BMS filtration at all primes $p$, as well as rational information, Morin \cite{Morin} and Bhatt--Lurie \cite[\textsection 6.4]{BL} constructed a \emph{global motivic filtration} $\fil^\star_{\mot} \THH(R)$.  Here, we prove that this motivic filtration is the even filtration:

\begin{theorem} \label{IntroBLcomparison}
  Let $R$ be a discrete commutative ring with bounded $p$-power torsion for all primes $p$ and such that the algebraic cotangent complex $\L^\alg_R$ has Tor-amplitude contained in $[0,1]$.  Then there is a canonical equivalence
  \[
    \fil^\star_{\mot} \THH(R) \simeq \fil^\star_{\ev} \THH(R),
  \]
  where $\fil^\star_{\mot}$ denotes the global motivic filtration of Morin and Bhatt--Lurie.
\end{theorem}

\begin{remark}
  Bhatt--Lurie further define the global motivic filtration on $\THH(R)$ for all animated commutative rings $R$, by left Kan extension from the case when $R$ is a polynomial $\mathbb{Z}$-algebra. Since polynomial $\mathbb{Z}$-algebras satisfy the conditions of \Cref{IntroBLcomparison}, one may use $\fil^\star_{\ev}$ and left Kan extension to recover $\fil_{\mot}^{\star} \THH(R)$ for any animated commutative ring $R$. By $p$-completion, one may then recover $\fil_{\BMS}^{\star} \cpl{\THH(R)}_p$ for any $p$-complete $p$-quasisyntomic $R$; this can be also recovered directly from a $p$-complete variant of the even filtration, as will be discussed further below.
\end{remark}

In light of the above theorem and remark, it is fair to say that the even filtration provides an alternate construction of the motivic filtration on the $\THH$ of animated commutative rings. Notably, the construction is inherently global, and avoids mention of perfectoid rings and the quasisyntomic site.

Even more notably, the even filtration is defined on \emph{any} $\E_\infty$-ring, not only $\E_\infty$-rings that arise as the $\THH$ of discrete commutative rings.  For example, we may take the even filtration of the sphere spectrum $\S$: it turns out that the result is the d\'ecalage of the Adams--Novikov filtration. More generally, we have the following result:

\begin{theorem}
For any $\E_\infty$-ring $A$, 
\[
  \fil^{\star}_{\ev} A \simeq \lim_{\Delta}(\fil^{\star}_{\ev}(A \otimes \mathrm{MU}^{\otimes \bullet+1})),
\]
where the limit is taken in the category of filtered $\E_\infty$-rings.
\end{theorem}

The Adams--Novikov filtration features heavily in Morel--Voevodsky's theory of $\mathbb{C}$-motivic stable homotopy theory \cite{MorelVoevodsky}, and hence the above theorem connects the even filtration to that theory. Precisely, the corollary below follows from the above theorem and the work of Gheorghe--Isaksen--Krause--Ricka \cite{Cmot} (cf. \cite{Pstragowski,SpecialFiber}):

\begin{corollary}
  Fix a prime number $p$. Then the $p$-completed cellular subcategory of the category $\mathcal{SH}(\mathbb{C})$ of $\mathbb{C}$-motivic spectra is equivalent to the category of $p$-complete filtered modules over $\fil^\star_\ev(\S)$. Under this identification, for every $\E_\infty$-ring $A$, the $p$-completion of $\fil^{\star}_\ev A$ is naturally an $\E_\infty$-algebra object in the category of $\mathbb{C}$-motivic spectra, and if $A$ is bounded below and $\MU_*A$ is even, then,
  \[
   \cpl{\left(\fil^{\star}_{\ev}A\right)}_p \simeq \cpl{\nu(A)}_p,
  \]
  where $\cpl{\nu(-)}_p$ is the synthetic analog functor from spectra to the $p$-completed cellular subcategory of $\mathcal{SH}(\mathbb{C})$.
\end{corollary}

\begin{remark}
  \label{in--nu}
  In \cref{mod}, we extend the notion of the even filtration to modules over $\mathbb{E}_{\infty}$-rings, defining a functor $(A,M) \mapsto \fil^\star_{\ev/A}M$, which recovers the above definition when $M=A$. With this notation, it follows from the results in \cref{mod} that, under the identification from \cite{Cmot} between $p$-complete filtered modules over $\fil^\star_\ev(\S)$ and $p$-complete cellular $\mathbb{C}$-motivic spectra,
  \[
    \cpl{\left(\fil^{\star}_{\ev/\mathbb{S}} M\right)}_p \simeq \cpl{\nu(M)}_p
  \]
  for any bounded below spectrum $M$ (or, more generally, for any $\MU$-complete spectrum $M$). It is then special to the scenario where $\MU_*A$ is even that we have the second equivalence in the string
  \[
    \fil^\star_\ev A \iso \fil^\star_{\ev/A}A \simeq \fil^\star_{\ev/\mathbb{S}}A.
  \]
  In general, filtered modules over $\fil^{\star}_\ev A$ can be viewed as a deformation of the category of $A$-modules, and the functor $\fil^\star_{\ev/A}:\Mod_A \to \FilMod_{\fil^\star_\ev A}$ associates a natural deformation to any $A$-module.
\end{remark}

The even filtration may also be used to  recover constructions in classical algebra. Namely, for $k \to R$ a map of discrete commutative rings, the Hochschild--Kostant--Rosenberg (HKR) filtration on the relative Hochschild homology $\mathrm{HH}(R/k)$ is a classical analog of the motivic filtration on $\THH(R)$. In the quasi-lci setting, the HKR filtration also turns out to be a special case of the even filtration:

\begin{theorem} \label{introHKRcomparison}
  Let $k \to R$ be a map of discrete commutative rings such that the algebraic cotangent complex $\L^\alg_{R/k}$ has Tor-amplitude contained in $[0,1]$. Then
  \[
    \fil^\star_{\ev} \HH(R/k) \simeq \fil^\star_{\HKR} \HH(R/k).
  \]
\end{theorem}

\begin{remark}
  \label{in--ev--proofs}
  Our proofs of the above comparison theorems rely on descent properties of $\fil^\star_{\ev}$, studied in \Cref{SecEvenFiltration}. Specifically, we identify certain maps of $\E_\infty$-rings $A \to B$, which we call \emph{evenly faithfully flat} (\emph{eff}), along which there is an identification $\fil_{\ev}^{\star}(A) \iso \lim_{\Delta}(\fil_{\ev}^{\star}(B^{\otimes_{A} \bullet+1}))$; when $B$ is moreover even, this formula simplifies to $\fil_{\ev}^{\star}(A) \iso \lim_{\Delta}(\tau_{\ge2\star}(B^{\otimes_{A} \bullet+1}))$. It is this technique of finding eff maps $A \to B$ with $B$ even that allows us to control the even filtration in the contexts above.
\end{remark}

\subsection{Variants of the even filtration}
\label{in--var}

Returning to topological Hochschild homology, we recall that the $\THH$ of an $\E_\infty$-ring is naturally a cyclotomic $\E_\infty$-ring \cite{NikolausScholze}.  In other words, it carries an $\cir$-action, allowing us to form $\TC^{-}(R)=\THH(R)^{\h\cir}$ and $\TP(R)=\THH(R)^{\tate\cir}$, together with a Frobenius $\varphi$ at each prime $p$, which Nikolaus--Scholze observed allows one to define $p$-completed $\TC(R)$ by the formula
\[
  \cpl{\TC(R)}_p \iso \fib\left(\cpl{\TC^{-}(R)}_p \lblto{\varphi^{\h\cir}-\mathrm{can}} \cpl{\TP(R)}_p\right),
\]
at least when $R$ is connective \cite{NikolausScholze}.
For $R$ a $p$-complete $p$-quasisyntomic discrete ring, Bhatt--Morrow--Scholze defined motivic filtrations on not just $\cpl{\THH(R)}_p$, but also $\cpl{\TC^{-}(R)}_p$, $\cpl{\TP(R)}_p$, and $\cpl{\TC(R)}_p$. Similarly, Morin and Bhatt--Lurie defined motivic filtrations not just on $\THH(R)$ but on $\TC^-(R)$ and $\TP(R)$.

In \Cref{SecComparison}, we prove that variants of the even filtration that account for $p$-completeness and/or equivariant structure can be used to recover all of these filtrations. Let us illustrate here by formulating a few of these variants and stating the comparison result for $p$-completed $\TC$.

\begin{definition}
  \label{in--var--p}
  We say that an $\E_\infty$-ring $B$ is of \emph{bounded $p$-power torsion} if for each $n \in \Z$ the homotopy group $\pi_n(B)$ has bounded $p$-power torsion. For $A$ a $p$-complete $\E_\infty$-ring, we define $\fil^\star_{\ev,p} A$ to be the limit, over all maps of $\E_\infty$-rings $A \to B$ with $B$ even, $p$-complete, and of bounded $p$-power torsion, of $\tau_{\ge2\star}(B)$.
\end{definition}

\begin{definition}
  \label{in--var--cir}
  We say that an $\cir$-equivariant $\E_\infty$-ring\footnote{Here ``$\cir$-equivariant $\E_\infty$-ring'' refers to a local system of $\E_\infty$-rings over $\mathrm{B}\cir$, not a more sophisticated notion of genuine equivariant homotopy theory.} $B$ is even if its underlying $\E_\infty$-ring is even. For $A$ an $\cir$-equivariant $\E_\infty$-ring, we define $\fil^\star_{\ev,\h\cir} A$ (resp. $\fil^\star_{\ev,\tate\cir} A$) to be the limit, over all maps of $\cir$-equivariant $\E_\infty$-rings $A \to B$ with $B$ even, of $\tau_{\ge2\star}(B^{\h\cir})$ (resp. $\tau_{\ge2\star}(B^{\tate\cir})$).
\end{definition}

The modifications introduced in \cref{in--var--p,in--var--cir} can be combined in an evident manner. Similarly, we have the following definition in the setting of $p$-complete cyclotomic $\E_\infty$-rings.

\begin{definition}
  \label{in--var--TC}
  We say that a cyclotomic $\E_\infty$-ring is \emph{even} or of \emph{bounded $p$-power torsion} if its underlying $\E_\infty$-ring is so. For any connective, $p$-complete cyclotomic $\E_\infty$-ring $A$, we define $\fil^{\star}_{\ev,p,\TC} A$ to be the limit, over all cyclotomic $\E_\infty$-ring maps $A \to B$ with $B$ even, connective, $p$-complete, and of bounded $p$-power torsion, of
  \[
    \fib\left(\tau_{\ge 2\star}\cpl{(B^{\h\cir})}_p \lblto{\varphi^{\h\cir}-\mathrm{can}} \tau_{\ge 2\star} \cpl{(B^{\tate\cir})}_p\right).
  \]
\end{definition}

\begin{theorem}
  \label{in--var--BMS}
  Let $R$ be a $p$-complete $p$-quasisyntomic commutative ring. Then
  \[
    \fil^{\star}_{\BMS}\cpl{\TC(R)}_p \simeq \fil^{\star}_{\ev,p,\TC} \THH(R).
  \]
\end{theorem}

\begin{remark}
  \label{in--var--proofs}
  The comments made in \cref{in--ev--proofs} apply also to the variants of the even filtration discussed in this subsection. That is, these variants also satisfy descent along eff maps, and we are able to control their behavior by finding eff maps into even objects. As usual, a slight variant notion of \emph{$p$-completely eff} is necessary in the $p$-complete contexts. On the other hand, a map of $\cir$-equivariant or cyclotomic $\E_\infty$-rings is ($p$-completely) eff if and only if its underlying map of $\E_\infty$-rings is so.
\end{remark}

\begin{remark}
  \label{in--var--other}
  There are many other variants of the even filtration that one could consider. As one example, one could define a version for the integral $\TC$ of connective cyclotomic $\E_\infty$-rings in a similar fashion to \cref{in--var--TC}, using the formula for integral $\TC$ given in \cite{NikolausScholze}; we choose not to pursue this here.
\end{remark}

\subsection{$\THH$ of chromatically quasisyntomic rings}

So far, we have discussed theorems showing that the even filtration recovers known filtrations. The ubiquity of these results suggests that there may be broader, unstudied contexts in which the even filtration is a useful tool. Here, we introduce the following class of $\E_\infty$-rings $R$ for which the even filtration on $\THH(R)$ can be controlled:

\begin{definition}
  A connective $\E_\infty$-ring $R$ is \emph{chromatically quasisyntomic} if $\MU_*R$ is even, has bounded $p$-power torsion for all primes $p$, and (when considered as an ungraded commutative ring) has algebraic cotangent complex $\L^{\alg}_{\MU_*R}$ with Tor-amplitude contained in $[0,1]$.
\end{definition}

\begin{example}
  If $R$ is a discrete commutative ring with bounded $p$-power torsion for each prime $p$ and such that the cotangent complex $\L^{\alg}_{R}$ has Tor-amplitude contained in $[0,1]$, then $R$ is chromatically quasisyntomic. This is because $\MU_*R$ is a polynomial algebra $R[b_1,b_2,\cdots]$, by complex orientation theory.
\end{example}

\begin{example}
  The $\E_\infty$-ring spectra $\mathbb{S}$, $\MU$, $\mathrm{ku}$, $\mathrm{ko}$, and $\mathrm{tmf}$ are all chromatically quasisyntomic. Indeed, if $R$ is any of these ring spectra, then $\MU_*R$ is a polynomial $\mathbb{Z}$-algebra concentrated in even degrees.  The Adams summand $\ell$ is also chromatically quasisyntomic, with $\MU_*\ell$ a polynomial $\mathbb{Z}_{(p)}$-algebra.
\end{example}

\begin{definition}
  \label{in--xq--filmot}
  For $R$ a chromatically quasisyntomic $\E_\infty$-ring, we define:
  \begin{itemize}
  \item $\fil_{\mot}^{\star} \THH(R)$ to be $\fil_{\ev}^{\star} \THH(R)$ (\cref{in--fil-ev})
  \item $\fil_{\mot}^{\star} \TC^{-}(R)$ to be $\fil^\star_{\ev,\h\cir} \THH(R)$ (\cref{in--var--cir}).

  \item $\fil_{\mot}^{\star} \TP(R)$ to be $\fil^\star_{\ev,\tate\cir} \THH(R)$ (\cref{in--var--cir}).
  \item $\fil^{\star}_{\mot} \cpl{\TC(R)}_p$ to be $\fil_{\ev,p,\TC}^{\star} \THH(R)$ (\cref{in--var--TC}).
  \end{itemize}
  These \emph{motivic filtrations} lead to \emph{motivic spectral sequences} for $\THH_*(R)$, $\TC^{-}_*(R)$, $\TP_*(R)$, and $\pi_*\left(\cpl{\TC(R)}_p\right)$, respectively.  In analogy with the discrete case, we call the associated graded of $\fil^{\star}_{\mot}\TP(R)$ the \emph{prismatic cohomology} of $R$,\footnote{To be more precise, we could define the \emph{Nygaard-completed absolute prismatic cohomology of $R$ with Breuil--Kisin twist $n$} to be $\gr^n_\mot \TP(R)[-2n]$. However, in this paper, we will simply write things in terms of $\gr^*_\mot \TP(R)$, and allow ourselves to use the term ``prismatic cohomology'' to refer to this graded object. Note that our choice not to incorporate the shift $[-2n]$ has an effect on what we refer to as the ``degree'' of a class in prismatic cohomology; see \cref{convention-adams-grading} for more detail.} and we call the associated graded of $\fil_{\mot}^{\star}\cpl{\TC(R)}_p$ the \emph{syntomic cohomology} of $R$.\footnote{The comments made about prismatic cohomology in the previous footnote apply analogously to syntomic cohomology. In addition, beware that, in our terminology, the syntomic cohomology of $R$ depends only on the $p$-completion of $R$ (for the implicit prime $p$). When $R$ is discrete, this recovers what is called the ``syntomic cohomology of $\Spf(R)$'' in \cite[\textsection 7.4]{BL}, as opposed to the ``syntomic cohomology of $\Spec(R)$'' defined in \cite[\textsection 8.4]{BL}.}
\end{definition}

In this context, we prove the following results:

\begin{theorem}
  \label{motConv}
  For $R$ a chromatically quasisyntomic $\E_\infty$-ring, the colimits of the filtered diagrams $\fil_{\mot}^{\star} \THH(R)$, $\fil_{\mot}^{\star} \TC^{-}(R)$, $\fil_{\mot}^{\star} \TP(R)$, and $\fil_{\mot}^{\star} \cpl{\TC(R)}_p$ are $\THH(R)$, $\TC^{-}(R)$, $\mathrm{TP}(R)$, and $\cpl{\TC(R)}_p$, respectively.
\end{theorem}

\begin{theorem}
  \label{filteredFrob}
  Let $R$ be a chromatically quasisyntomic $\E_\infty$-ring.  Then, for each prime number $p$, the Nikolaus--Scholze Frobenius
  \[
    \varphi:\mathrm{TC^{-}}(R)_p^{\wedge} \to \mathrm{TP}(R)_p^{\wedge}
  \]
  refines to a natural map
  \[
    \varphi:\fil_{\mot}^{\star} \TC^{-}(R)^{\wedge}_p\to \fil_{\mot}^{\star} \mathrm{TP}(R)_p^{\wedge}.
  \]
  The same is true of the canonical map between the same objects, and $\fil^{\star}_{\mot} \cpl{\TC(R)}_p$ can be computed as the equalizer of the filtered Frobenius and canonical maps.
\end{theorem}

\begin{remark}
  Our proof of \Cref{filteredFrob} is not formal, and in particular crucially uses the hypothesis that $R$ is chromatically quasisyntomic. The difficulty is that $\fil_{\mot}^{\star} \TC^{-}(R)$ and $\fil_{\mot}^{\star} \TP(R)$ are defined using only the $\cir$-action on $\THH(R)$, while $\fil^{\star}_{\mot} \cpl{\TC(R)}_p$ is defined using the full cyclotomic structure.  To deduce the theorem, we will need to show that there is a sufficient supply of cyclotomic $\mathbb{E}_{\infty}$-rings under $\THH(R)$. Specifically, we show that there exists an eff map of cyclotomic $\mathbb{E}_{\infty}$-rings
  \[
    \THH(R) \to B,
  \]
  where $B$ is connective, even, and of bounded $p$-power torsion for all primes $p$. The existence of such implies that the variants of the even filtration involving different structure are consistent with one other (see \cref{in--ev--proofs,in--var--proofs}).
	
  A key ingredient in establishing this existence is a result of Steve Wilson, which states, for each positive integer $i$ and even ring spectrum $A$, that $A_*\Omega^{\infty} \Sigma^{2i} \mathrm{MU}$ is an even, polynomial $A_*$-algebra.  In \Cref{WilsonAppendix}, we enhance Wilson's result by proving that, for at least some even $\E_\infty$-rings $A$, $\THH$ relative to $A \otimes \Sigma^{\infty}_+\Omega^{\infty} \Sigma^{2i} \mathrm{MU}$ carries a cyclotomic structure. In other words, there is a dashed arrow and commutative diagram of $\cir$-equivariant $\E_\infty$-rings
  \[
    \begin{tikzcd}
      \THH(A \otimes \Sigma^{\infty}_+\Omega^{\infty} \Sigma^{2i} \MU)  \arrow{r}{\varphi} \arrow{d}{\pi}  &  \THH(A \otimes \Sigma^{\infty}_+\Omega^{\infty} \Sigma^{2i} \MU) ^{\tate\Cp} \arrow{d}{\pi^{\tate\Cp}}\\
      A \otimes \Sigma^{\infty}_+\Omega^{\infty} \Sigma^{2i} \MU \arrow[dashed]{r}  & (A \otimes \Sigma^{\infty}_+\Omega^{\infty} \Sigma^{2i} \MU)^{\tate\Cp}.
    \end{tikzcd}
  \] 
\end{remark}

It would be interesting to know if any analog of \Cref{filteredFrob} holds for a broader class of $\E_\infty$-rings $R$.

\begin{remark}
  Concretely, given a chromatically quasisyntomic $\E_\infty$-ring $R$, one can compute the motivic spectral sequence for $\THH_*(R)$ by finding an $\E_\infty$-$\MU$-algebra map $S \to \MU \otimes R$ with the following properties:
  \begin{enumerate}
    \item $\pi_*S$ is a polynomial $\MU_*$-algebra, with polynomial generators in even degrees.
    \item The map $\pi_*S \to \MU_*R$ is surjective.
    \end{enumerate}
    The motivic filtration will then be given by descent along the composite 
    \[
      \THH(R) \to \THH(R) \otimes \MU \simeq \THH(R \otimes \MU/\MU) \to \THH(R \otimes \MU/S).
    \]
\end{remark}

\begin{remark}
  Suppose that $R$ is a chromatically quasisyntomic $\E_\infty$-ring spectrum such that $\pi_0R$ is also chromatically quasisyntomic.  Given any filtration on $\cpl{\mathrm{K}^{\mathrm{alg}}(\pi_0R)}_p$ compatible with the motivic filtration on $\cpl{\TC(\pi_0R)}_p$, one can define a filtration on $\cpl{\mathrm{K}^{\mathrm{alg}}(R)}_p$ by pullback \cite[Theorem 0.0.2]{DGM}.  We thus expect that, by mixing Voevodsky's filtration on the algebraic $\mathrm{K}$-theory of discrete rings with our motivic filtration on $\TC$ of $\E_\infty$-rings, one may obtain a motivic spectral sequence for the algebraic $\mathrm{K}$-theory of many chromatically quasisyntomic ring spectra.
\end{remark}

\begin{remark} \label{number-ring-example}
Let $\mathcal{O}_K$ be the ring of integers in a local number field, and fix a choice of uniformizer $\pi$. The works \cite{BMS,LiuWang,NikolausKrause,Zpn} study $\cpl{\THH(\mathcal{O}_K)}_p$ and $\cpl{\TC(\mathcal{O}_K)}_p$ by ($p$-completed) descent along the map 
\[\THH(\mathcal{O}_K) \to \THH(\mathcal{O}_K/\mathbb{S}[\pi]),\]
which is a map of cyclotomic $\E_\infty$-ring spectra. 
It follows from our work here (specifically, from the fact that $\THH(\Sigma^{\infty}_+ \mathbb{N}) \to \Sigma^{\infty}_+ \mathbb{N}$ is evenly free) that the descent filtration so obtained agrees with the even filtration on $\cpl{\TC(\mathcal{O}_K)}_p$, and hence with the BMS filtration on $\cpl{\TC(\mathcal{O}_K)}_p$.  A comparison of the descent and BMS filtrations was independently obtained by Antieau--Krause--Nikolaus, and will appear in their announced work \cite{Zpn}.
\end{remark}

The work of Liu--Wang referenced in \Cref{number-ring-example} shows that the motivic filtration allows for the practical computation of $\TC$ of number rings \cite{LiuWang}.  Additionally, works such as \cite{Zpn, Sulyma} and \cite[\textsection 10]{RecentTHH} show that the motivic filtration can be used to make precise calculations of the algebraic $K$-theories of commutative rings with lci singularities.

Similarly, we expect the motivic filtration to be a useful computational tool in higher chromatic contexts. In \cite{DavidLee} David Jongwon Lee completely computes $\THH_*(\mathrm{ku})$ by use of the motivic spectral sequence.  No complete computation of $\THH_*(\mathrm{ku})$ previously existed in the literature, though tremendous progress was achieved by Ausoni \cite{AusoniTHH} and Angeltveit--Hill--Lawson \cite{THHlko}, building on work of McClure--Staffeldt \cite{McClureStaffeldt} and Angeltveit--Rognes \cite{AngeltveitRognes}.  In the final section of the paper, we explain another computational application of the motivic spectral sequence.

\subsection{The motivic spectral sequence for $V(1)_\star \TC(\ell)$}

In \Cref{SecAdamsSummand} we will study the connective Adams summand $\ell$ of $\mathrm{ku}_{(p)}$ at an arbitrary prime $p \ge 2$.  In seminal work, Ausoni and Rognes computed $V(1)_* \TC(\ell)=\pi_* \left(\TC(\ell) / (p,v_1)\right)$ for primes $p \ge 5$ \cite{AusoniRognes}.  We demonstrate the computability of syntomic cohomology by giving an independent proof of their result.  Furthermore, our methods just as easily compute $V(1)_* \TC(\ell)$ at the prime $p=3$, which was not previously accessible.
Our main theorem is the following, stated in terms of mod $(p,v_1)$ syntomic cohomology as formally defined in \Cref{dfn:reduced-greven}:

\begin{theorem} \label{thm:introSynell} 
  For any prime $p \ge 2$, the mod $(p,v_1)$ syntomic cohomology groups of $\ell$, i.e.
  \[
    \pi_*(\gr^*_\mot \TC(\ell)/(p,v_1)),
  \]
  form a free $\mathbb{F}_p[v_2]$-module on finitely many generators. A complete list of module generators is given by:
  \begin{enumerate}
  \item $ \{1\}$, in Adams weight $0$ and degree $0$.
  \item $\{\partial,t^{d}\lambda_1,t^{dp}\lambda_2\text{ }|\text{ }0 \le d < p\}$, in Adams weight $1$.  Here, $|\partial|=-1$, $|t^d\lambda_1|=2p-2d-1$, and $|t^{dp}\lambda_2|=2p^2-2dp-1$.
  \item $\{t^d\lambda_1\lambda_2, t^{dp} \lambda_1 \lambda_2, \partial \lambda_1,\partial\lambda_2\text{ }|\text{ }0 \le d < p\}$, in Adams weight $2$.  Here, $|t^d \lambda_1\lambda_2|=2p^2-2p-2d-2$, $|t^{dp} \lambda_1\lambda_2|=2p^2-2p-2dp-2$, $|\partial \lambda_1|=2p-2$, and $|\partial\lambda_2|=2p^2-2$.
  \item $\{\partial\lambda_1\lambda_2\}$, in Adams weight $3$ and degree $2p^2+2p-3$.
  \end{enumerate}
The degree of $v_2$ is $2p^2-2$, and it has Adams weight $0$.
\end{theorem}

Here we have used the following convention:

\begin{definition}[Adams weight]
  \label{convention-adams-grading}
  For $M^*$ a graded spectrum, we say that an element of $\pi_n(M^m)$ has \emph{Adams weight} $2m-n$ and \emph{degree} $n$, and we write $|-|$ for degree.\footnote{This convention can be translated as follows. Suppose that for a graded spectrum $N^*$ we say that an element of $\pi_d(N^m)$ has \emph{cohomological degree} $-d$ and \emph{motivic weight} $m$. Then, for a graded spectrum $M^*$, if we set $N^m := M^m[-2m]$, a class in $\pi_\bullet(M^*)$ with Adams weight $a$ and degree $n$ corresponds to a class in $\pi_\bullet(N^*)$ of cohomological degree $a$ and motivic weight $\frac{a+n}{2}$.}
\end{definition}

When $p \ge 3$, so that $V(1)=\mathbb{S}/(p,v_1)$ exists as a spectrum, the mod $(p,v_1)$ syntomic cohomology described by \Cref{thm:introSynell} is the $\mathrm{E}_2$-page of a motivic spectral sequence converging to $V(1)_* \TC(\ell)$.  Below, we draw a picture of the $\mathrm{E}_2$-page of this spectral sequence for $p=5$, with the horizontal axis recording degree and the vertical axis recording Adams weight.  The Adams grading convention means that $d_r$ differentials decrease degree
by $1$ and increase Adams weight by $r$.

\includegraphics[scale=1,trim={4cm 18.5cm 3.5cm 2.7cm},clip]{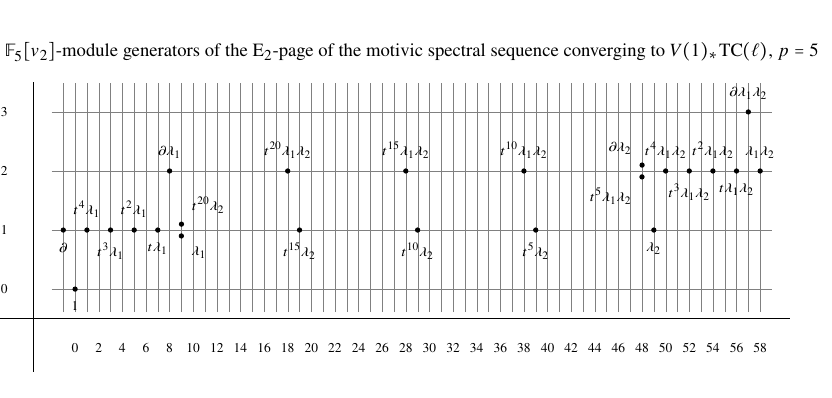}

The reader may notice the similarity between the picture of the motivic spectral sequence for $V(1)_* \TC(\ell)$, which we rigorously define here, and Rognes' conjectured picture for a conjectured motivic spectral sequence converging to $V(1)_* \mathrm{K}^{\mathrm{alg}}(\ell_p^{\wedge})$ \cite[Example 5.2]{RognesICM}.  The rotational symmetry present in the above picture is indicative of Rognes' conjectured Tate--Poitou duality, which we will explore in future work.

The $\mathrm{E}_2$-page of the motivic spectral sequence for $V(1)_* \TC(\ell)$ is concentrated on four horizontal lines, with classes of odd degree on the $1$- and $3$-lines and classes of even degree on the $0$- and $2$-lines.  It follows by parity considerations that the only possible differentials are $d_3$'s from the $0$-line to the $3$-line.  On the $0$-line is a copy of $\mathbb{F}_p[v_2]$, where the degree of $v_2$ is $2p^2-2$.  A $d_3$ differential off of the $0$-line would therefore have target of degree $1$ less than a multiple of $2p^2-2$.  However, the $3$-line is concentrated in degrees $2p-1$ more than a multiple of $2p^2-2$.  No differentials are possible, and we conclude the following corollary:

\begin{corollary} \label{cor:introTCell} Let $p \ge 3$, so that $V(1)=\mathbb{S}/(p,v_1)$ exists as a spectrum.  Then the motivic spectral sequence for $V(1)_* \TC(\ell)$ degenerates at the $\mathrm{E}_2$-page.
\end{corollary}

At primes $p \ge 5$, $v_2$ is a self map of $V(1)$ and the motivic spectral sequence for $V(1)_* \TC(\ell)$ is one of $\mathbb{F}_p[v_2]$-modules.  We note that there are no possible $\mathbb{F}_p[v_2]$-module extension problems, and conclude that $V(1)_{*} \TC(\ell)$ is a free $\mathbb{F}_p[v_2]$-module on generators with the same names as those in \Cref{thm:introSynell}.  At the prime $p=3$, $v_2^9$ is a self map of $V(1)$ \cite{v29exists}, and we may similarly conclude that $V(1)_{*} \TC(\ell)$ is a free $\mathbb{F}_3[v_2^9]$-module.

\begin{remark}
It would be excellent to see the results of \cite{EllipticTC} or \cite{MoravaKTC} similarly recovered via motivic spectral sequences.  These works compute the topological cyclic homologies of non-commutative ring spectra, which presents an obvious complication, but see \Cref{example:BPn}.
\end{remark}

Finally, we explain in \Cref{subsec:finalsyntomic} how \Cref{thm:introSynell} allows us to deduce the following two qualitative results:

\begin{theorem} \label{thm:introLQell}
For any prime $p \ge 2$ and type $3$ $p$-local finite complex $F$, $F_* \TC(\ell)$ is finite.  In particular,
\[\TC(\ell)_{(p)} \to L_2^{f} \TC(\ell)_{(p)}\]
is a $\pi_*$-iso for $* \gg 0$.
\end{theorem}

\begin{theorem} \label{thm:introTelescope}
The telescope conjecture is true of $\TC(\ell)$.  In other words, the natural map
\[L_2^{f} \TC(\ell) \to L_2 \TC(\ell)\]
is an equivalence.
\end{theorem}

\Cref{thm:introLQell} was previously proved for $p \ge 5$ in \cite{AusoniRognes}, and at all primes in \cite{HahnWilson}.  On the other hand, \Cref{thm:introTelescope} is new to this paper.  Of course, a natural next question is to compute the syntomic cohomology of $\ell$ without modding out by $p$ or $v_1$.  The authors 
and collaborators plan to say more about the prismatic and syntomic cohomologies of $\ell$, $\mathrm{ku}$, $\mathrm{ko}$, and $\MU$ in future work. We note that the prismatic cohomologies of $\E_\infty$-ring spectra may be of interest even to those studying prismatic cohomology of discrete rings. Specifically, the sequence of $\E_\infty$-ring maps $\mathbb{S} \to \MU \to \mathrm{ku} \to \mathbb{Z}$ suggests a natural factorization of known connections between $\mathrm{WCart}$ and the moduli stack of formal groups.

\subsection{Acknowledgments}
The authors thank Gabriel Angelini--Knoll, Bhargav Bhatt, Sanath Devalapurkar, Lars Hesselholt, Mike Hopkins, Achim Krause, Markus Land, David Jongwon Lee, Ishan Levy, Jacob Lurie, Haynes Miller, Andrew Senger, Guozhen Wang, and Allen Yuan for helpful conversations. We would particularly like to acknowledge Akhil Mathew and Thomas Nikolaus for teaching us a great deal about cyclotomic spectra. We thank John Rognes, both for his inspiring work and conjectures and for his many detailed comments on earlier drafts. We thank Ben Antieau too for many helpful comments and suggestions and for his kind enthusiasm. We are also grateful to Ferdinand Wagner for detailed feedback on the previous version of this paper, in particular catching some errors in that version which have now been corrected. Finally, we thank the referees for their readings and reports which have much improved the paper.

During the course of this work, the second author was supported by NSF grant DMS-2103152 and the Bell System Fellowship at the IAS, and the first author was supported by the Sloan foundation and a grant from the IAS School of Mathematics.

\subsection{Conventions} \label{ConventionsSection}

\begin{enumerate}[leftmargin=*]
\item $\Spc$ denotes the ($\infty$-)category of spaces; $\Spt$ denotes the category of spectra; $\CAlg$ denotes the category of $\E_\infty$-rings; $\Ab$ denotes the category of abelian groups.
\item The symbol $p$ always refers to a prime number.
\item $\CycSpt_p$ denotes the category of
bounded below
$p$-typical cyclotomic spectra (following the same conventions as in \cite[Definition 3.2.1]{HahnWilson}, in particular including the hypothesis of $p$-completeness); $\CycCAlg_p$ denotes the category of bounded below,
$p$-typical cyclotomic $\E_\infty$-rings, i.e. commutative algebra objects in $\CycSpt_p$.
\item $\FilSpt$ denotes the category of filtered spectra, i.e. the category of functors from the poset $(\mathbb{Z},\ge)$ to $\Spt$; we will generally denote a filtered spectrum by a symbol of the form $\fil^\star_? X$, referring to a diagram of spectra
  \[
    \cdots \to \fil^2_? X \to \fil^1_? X \to \fil^0_? X \to \fil^{-1}_? X \to \fil^{-2}_? X \to \cdots.
  \]
  Given a filtered spectrum $\fil^\star_? X$, we refer to $\colim_{n \to \infty} \fil^{-n}_? X$ as its \emph{underlying object}, and we say $\fil^\star_? X$ is \emph{complete} if $\lim_{n \to \infty} \fil^n_? X \iso 0$. The category $\BiFilSpt$ of bifiltered spectra is defined similarly, using instead the product poset $(\Z,\ge) \times (\Z,\ge)$; we will generally denote a bifiltered spectrum by a symbol of the form $\fil^\filledsquare_{??} \fil^\star_? X$, where again the $\filledsquare$ and $\star$ stand for the two integer indices.
\item $\FilCAlg$ denotes the category of filtered $\E_\infty$-rings, i.e. commutative algebra objects in $\FilSpt$ with respect to the Day convolution symmetric monoidal structure (determined by addition on $\Z$), and $\BiFilCAlg$ is defined similarly. For $\fil^\star_? A$ a filtered $\E_\infty$-ring, $\FilMod_{\fil^\star_? A}$ denotes the category of modules over $\fil^\star_? A$ in $\FilSpt$.
\item $\GrSpt$ denotes the category of graded spectra, i.e. the category of functors from the discrete set $\Z$ to $\Spt$, and the category $\BiGrSpt$ of bigraded spectra is defined similarly. A filtered spectrum $\fil^\star_? X$ has an \emph{associated graded} object $\gr^*_? X = \{\gr^n_? X\}_{n \in \Z}$, given by the formula
  \[
    \gr^n_? X \iso \cofib(\fil^{n+1}_? X \to \fil^n_? X).
  \]
  We recall that a map of complete filtered spectra is an equivalence if and only if the induced map of associated graded objects is.
\item \label{in--con--grab}
  $\GrAb$ denotes the category of graded abelian groups, defined similarly to $\GrSpt$ above. A graded abelian group $M_* = \{M_n\}_{n \in \Z}$ has an \emph{underlying abelian group} $\bigoplus_{n \in \Z} M_n$.  By a \emph{graded commutative ring}, we will mean a commutative algebra in $\GrAb$ with respect to the Day convolution symmetric monoidal structure; note that this involves no Koszul sign rule, so in particular a graded commutative ring has an underlying commutative ring. (Of course, graded rings that are commutative in the sense of the Koszul sign rule will also appear in this paper, but we will not refer to them as graded commutative rings.)
\item For $A$ a commutative ring and $a \in A$, when we refer to an $A$-module $M$ being \emph{$a$-complete}, we always mean in the derived sense. This applies in particular to the case that $A = \Z$ and $a = p$.
\item Given two elements $a,b$ in an $\mathbb{F}_p$-algebra, we will sometimes write $a \doteq b$ to mean that $a$ is equal to an element of $\mathbb{F}_p^{\times}$ times $b$. 
\end{enumerate}

%% file: EvenFiltration.tex
In this section, we define the even filtration and its many variants (\cref{ev--def}). We then formulate and prove flat descent properties for each of these variants (\cref{ev--desc}), and then use these to deduce further useful properties (\cref{ev--exh,ev--arith}). The arguments here are not complicated, but must be carefully repeated in slightly different contexts ($p$-complete, $\cir$-equivariant, cyclotomic, etc.).


\subsection{Defining the filtrations}
\label{ev--def}

The most basic version of the even filtration construction was roughly formulated in \cref{in--fil-ev}. We now state it more precisely.

\begin{notation}
  \label{ev--def--ev}
  Let $\CAlg^\ev$ denote the full subcategory of $\CAlg$ spanned by the even $\E_\infty$-rings.
\end{notation}

\begin{proposition}
  \label{ev--def--accessible}
  $\CAlg^\ev$ is an accessible subcategory of $\CAlg$.
\end{proposition}

\begin{proof}
  Let $\GrAb^\ev$ denote the full subcategory of $\Gr\Ab$ spanned by those graded abelian groups $\{M_n\}_{n \in \Z}$ such that $M_n \iso 0$ for all odd integers $n$. By definition, we have a pullback diagram of $\infty$-categories
  \[
    \begin{tikzcd}
      \CAlg^\ev \ar[r] \ar[d, "\pi_*", swap] &
      \CAlg \ar[d, "\pi_*"] \\
      \GrAb^\ev \ar[r] &
      \GrAb,
    \end{tikzcd}
  \]
  where the horizontal arrows are the inclusion functors. The categories $\CAlg$, $\GrAb$, and $\GrAb^\ev$ are accessible, and both $\pi_* : \CAlg \to \GrAb$ and the inclusion functor $\GrAb^\ev \inj \GrAb$ preserve filtered colimits, hence are accessible. Applying \cite[Proposition 5.4.6.6]{htt}, we obtain that $\CAlg^\ev$ is an accessible subcategory of $\CAlg$.
\end{proof}
 
\begin{construction}
  \label{ev--def--fil-ev}
  Recall that there is a functor $\tau_{\ge 2\star} : \Spt \to \FilSpt$ sending a spectrum $X$ to its double-speed Postnikov filtration
  \[
    \cdots \to \tau_{\ge 4}(X) \to \tau_{\ge 2}(X) \to \tau_{\ge 0}(X) \to \tau_{\ge -2}(X) \to \tau_{\ge -4}(X) \to \cdots .
  \]
  Moreover, there is a canonical lax symmetric monoidal structure on this functor, so that it induces a functor $\tau_{\ge 2\star} : \CAlg \to \FilCAlg$. It follows from \cref{ev--def--accessible} that the restriction $\tau_{\ge 2\star} : \CAlg^\ev \to \FilCAlg$ admits a right Kan extension along the inclusion $\CAlg^\ev \subseteq \CAlg$; we denote this right Kan extension by
  \[
    \fil^\star_\ev : \CAlg \to \FilCAlg
  \]
  and refer to this construction as the \emph{even filtration}.
\end{construction}

\begin{remark}
  \label{ev--def--fil-ev-formula}
  For any $\E_\infty$-ring $A$, the even filtration is given by the following limit expression:
  \[
    \fil^{\star}_\ev(A) \iso
    \lim_{A \to B, B \in \CAlg^{\ev}} \tau_{\ge 2\star}(B)
  \]
  (\cref{ev--def--accessible} implies that this limit is equivalent to a small limit and hence exists in the category $\FilCAlg$).
\end{remark}

\begin{remark}
  \label{ev--def--complete}
  For any spectrum $X$, the double speed Postnikov filtration $\tau_{\ge 2\star}(X)$ is complete (that is, $\lim_{n \to \infty} \tau_{\ge 2n}(X) \iso 0$). The collection of complete filtered spectra is closed under limits, so for any $\E_{\infty}$-ring $A$, the even filtration $\fil^\star_\ev(A)$ is complete.
\end{remark}

\begin{remark}
  \label{ev--def--exhaustive}
  By definition of $\fil^\star_\ev$, for an $\E_\infty$-ring $A$, there is a natural map of filtered $\E_\infty$-rings $\tau_{\ge 2\star}(A) \to \fil^\star_\ev(A)$. Passing to underlying objects, we find a natural map of $\E_\infty$-rings $A \to \colim(\fil^\star_\ev(A))$. Beware that this is not an equivalence in general, though it is in many cases of interest (see \cref{ev--exh}).
\end{remark}

We now define a variant of the even filtration in the setting where our $\E_\infty$-rings are required to be $p$-complete.

\begin{definition}
  \label{ev--def--bdd-torsion}
  As in \cite{BMS,BL}, we say that an abelian group $M$ has \emph{bounded $p$-power torsion} if the sequence of subgroups of $p$-power torsion
  \[
    M[p] \subseteq M[p^2] \subseteq M[p^3] \subseteq \cdots
  \]
  eventually stabilizes. We similarly say that a graded abelian group $M_*$ has bounded $p$-power torsion if the sequence of graded subgroups of $p$-power torsion
  \[
    M_*[p] \subseteq M_*[p^2] \subseteq M_*[p^3] \subseteq \cdots
  \]
  eventually stabilizes, or equivalently if the underlying abelian group $\bigoplus_{n \in \Z} M_n$ has bounded $p$-power torsion. Finally, we say that a spectrum $X$ has bounded $p$-power torsion if the graded abelian group $\pi_*(X)$ does.
\end{definition}

\begin{lemma}
  \label{ev--def--bdd-torsion-accessible}
  The abelian groups (resp. graded abelian groups) of bounded $p$-power torsion span an $\aleph_1$-accessible subcategory of $\Ab$ (resp. $\GrAb$).
\end{lemma}

\begin{proof}
  The graded statement can be deduced easily from the ungraded statement, so let us just address the latter.
  
  Suppose given an $\aleph_1$-filtered diagram of abelian groups $\{M_i\}_{i \in I}$ where each $M_i$ has bounded $p$-power torsion. Let $M$ denote the colimit of this diagram in $\Ab$. We claim that $M$ also has bounded $p$-power torsion. Indeed, suppose this were not the case: then for each positive integer $n$, we could find $i_n \in I$ and $x_n \in M_{i_n}[p^n]$ whose image in $M$ lies in $M[p^n] \setminus M[p^{n-1}]$. Since $I$ is $\aleph_1$-filtered, there would then exist some $i \in I$ with maps $i_n \to i$ in $I$ for each $n$, and the image of $x_n$ in the induced map $M_{i_n} \to M_i$ would necessarily be contained in $M_i[p^n] \setminus M_i[p^{n-1}]$, contradicting the assumption that $M_i$ have bounded $p$-power torsion.

  Having shown that the subcategory is closed under $\aleph_1$-filtered colimits in $\Ab$, it suffices to show that any object in it is an $\aleph_1$-filtered colimit of objects in it that are $\aleph_1$-compact in $\Ab$. This follows from the facts that the subcategory is closed under taking subgroups and that any abelian group is the colimit of its countably generated subgroups.
\end{proof}

\begin{notation}
  \label{ev--def--ev-p}
  Let $\CAlg_p$ denote the full subcategory of $\CAlg$ spanned by the $\E_\infty$-rings that are $p$-complete, and let $\CAlg_p^\ev$ denote the full subcategory of $\CAlg^\ev$ spanned by the even $\E_\infty$-rings that are $p$-complete and have bounded $p$-power torsion. 
\end{notation}

\begin{proposition}
  \label{ev--def--p-accessible}
  $\CAlg_p^\ev$ is an accessible subcategory of $\CAlg_p$.
\end{proposition}

\begin{proof}
  This follows from an argument similar to the one used in the proof of \cref{ev--def--accessible}, with the additional input of \cref{ev--def--bdd-torsion-accessible}.
\end{proof}

\begin{construction}
  \label{ev--def--fil-ev-p}
  We define 
  \begin{align*}
    \fil^\star_{\ev,p} &: \CAlg_p \to \FilCAlg
  \end{align*}
  to be the right Kan extension of
  \begin{align*}
    \tau_{\ge 2\star} &: \CAlg_p^\ev \to \FilCAlg
  \end{align*}
  along the inclusion $\CAlg_p^{\ev} \subseteq \CAlg_p$; this right Kan extension exists by \cref{ev--def--p-accessible}. We refer to this construction as the \emph{$p$-complete even filtration}.
\end{construction}

We next define variants of the even filtration in the setting of $\E_\infty$-rings with $\cir$-action.

\begin{notation}
  \label{ev--def--ev-cir}
  Say that an $\E_\infty$-ring with $\cir$-action is \emph{even} if its underlying $\E_\infty$-ring is even. Then $(\CAlg^\ev)^{\Bcir}$ is the full subcategory of $\CAlg^\Bcir$ spanned by the even $\E_\infty$-rings with $\cir$-action, and we similarly have a full subcategory $(\CAlg_p^\ev)^{\Bcir} \subseteq \CAlg_p^\Bcir$.
\end{notation}

\begin{remark}
  \label{ev--def--ht-ev}
  For $A$ an even $\E_\infty$-ring with $\cir$-action, the homotopy fixed points $A^{\h\cir}$ and Tate construction $A^{\tate\cir}$ are even $\E_\infty$-rings (see \cref{ev--desc--orientation}\cref{ev--desc--orientation--class} below).
\end{remark}

\begin{construction}
  \label{ev--def--fil-ev-cir}
  We define
  \begin{align*}
    &\fil^\star_{\ev,\h\cir} : \CAlg^\Bcir \to \FilCAlg,
    &&\fil^{\filledsquare}_{+}\fil^\star_{\ev,\h\cir} : \CAlg^\Bcir \to \BiFilCAlg, \\
    &\fil^\star_{\ev,\tate\cir} : \CAlg^\Bcir \to \FilCAlg,
    &&\fil^{\filledsquare}_+\fil^\star_{\ev,\tate\cir} : \CAlg^\Bcir \to \BiFilCAlg
  \end{align*}
  to be the right Kan extensions of
  \begin{align*}
    &\tau_{\ge 2\star}((-)^{\h\cir}) : (\CAlg^\ev)^\Bcir \to \FilCAlg,
    &&\tau_{\ge 2\star}((\tau_{\ge \filledsquare}(-))^{\h\cir}) : (\CAlg^\ev)^\Bcir \to \BiFilCAlg,\\
    &\tau_{\ge 2\star}((-)^{\tate\cir}): (\CAlg^\ev)^\Bcir \to \FilCAlg,
    &&\tau_{\ge 2\star}((\tau_{\ge \filledsquare}(-))^{\tate\cir}): (\CAlg^\ev)^\Bcir \to \BiFilCAlg,
  \end{align*}
  respectively, along the inclusion $(\CAlg^\ev)^\Bcir \subseteq \CAlg^\Bcir$ (their existence following from \cref{ev--def--accessible}), and we define
  \begin{align*}
    &\fil^\star_{\ev,p,\h\cir} : \CAlg_p^\Bcir \to \FilCAlg,
    &&\fil^{\filledsquare}_+\fil^\star_{\ev,p,\h\cir} : \CAlg_p^\Bcir \to \BiFilCAlg,\\
    &\fil^\star_{\ev,p,\tate\cir} : \CAlg_p^\Bcir \to \FilCAlg, &&\fil^{\filledsquare}_+\fil^\star_{\ev,p,\tate\cir} : \CAlg_p^\Bcir \to \BiFilCAlg,
  \end{align*}
  to be the right Kan extensions of
  \begin{align*}
    &\tau_{\ge 2\star}((-)^{\h\cir}) : (\CAlg_p^\ev)^\Bcir \to \FilCAlg,
    &&\tau_{\ge 2\star}((\tau_{\ge \filledsquare}(-))^{\h\cir}) : (\CAlg_p^\ev)^\Bcir \to \BiFilCAlg,\\
    &\tau_{\ge 2\star}((-)^{\tate\cir}) : (\CAlg_p^\ev)^\Bcir \to \FilCAlg,
    &&\tau_{\ge 2\star}((\tau_{\ge \filledsquare}(-))^{\tate\cir}) : (\CAlg_p^\ev)^\Bcir \to \BiFilCAlg,
  \end{align*}
  respectively, along the inclusion $(\CAlg_p^\ev)^\Bcir \subseteq \CAlg_p^\Bcir$ (their existence following from \cref{ev--def--p-accessible}).
\end{construction}

Finally, we define a variant of the even filtration in the setting of cyclotomic $\E_\infty$-rings.

\begin{notation}
  \label{ev--def--ev-cyc}
  Let $\CycCAlg^\ev_p$ denote the full subcategory of $\CycCAlg_p$ spanned by the $p$-typical cyclotomic $\E_\infty$-rings whose underlying $p$-complete $\E_\infty$-ring lies in $\CAlg_p^\ev \subseteq \CAlg_p$, i.e. is even and has bounded $p$-power torsion.
\end{notation}

\begin{remark}
  \label{ev--def--cyc-rmk}
  We remind the reader of our convention to only consider bounded below, $p$-complete cyclotomic spectra. In particular, for $A \in \CycCAlg_p$, the Tate spectrum $A^{\tate\cir}$ is $p$-complete (see \cite[\textsection 2.3]{BMS}), and hence, by \cite{NikolausScholze}, there is a natural Frobenius map $\varphi : A^{\h\cir} \to A^{\tate\cir}$, together with a natural identification $\TC(A) \iso \fib(\varphi - \can : A^{\h\cir} \to A^{\tate\cir})$.
\end{remark}

\begin{construction}
  \label{ev--def--fil-ev-cyc}
  We define
  \[
    \fil^\star_{\ev,p,\TC}(-) : \CycCAlg_p \to \FilCAlg,
  \]
  to be the right Kan extension of
  \[
    \fib(\varphi-\mathrm{can} : \tau_{\ge 2\star}((-)^{\h\cir}) \to \tau_{\ge 2\star}((-)^{\tate\cir})) : \CycCAlg_p^\ev \to \FilCAlg,
  \]
  along the inclusion $\CycCAlg_p^\ev \subseteq \CycCAlg_p$ (its existence following from \cref{ev--def--p-accessible}).
\end{construction}


\subsection{Descent properties of the filtrations}
\label{ev--desc}

The key to computing these even filtrations in our cases of interest is a simple flat descent property, which we formulate and prove in this subsection. We first recall the relevant notions of flatness in $p$-complete and graded algebra.

\begin{definition}
  \label{ev--desc--pcpl-flat}
  Let $A$ be a commutative ring and let $M \in \D(A)$. Following \cite{BMS,BL}, we say that $M$ is \emph{$p$-completely (faithfully) flat} if $(A/pA) \otimes^{\L}_{A} M$ is a (faithfully) flat $A/pA$-module concentrated in homological degree $0$.
\end{definition}

\begin{remark}
  \label{ev--desc--pcpl-flat-alt}
  In the situation of \cref{ev--desc--pcpl-flat}, it follows from \cite[Proposition 2.7.3.2(c)]{sag} that $M$ is $p$-completely flat if and only if $M \otimes^{\L}_{\Z}\Z/p$ is flat over $A \otimes^{\L}_{\Z} \Z/p$ in the sense of \cite[Definition 7.2.2.10]{ha}. The reader may thus replace the underived construction $A/pA$ with the derived construction $A\otimes^{\L}_{\Z}\Z/p$ in the above definition. 
\end{remark}

\begin{remark}
  \label{ev--desc--flat-implies-pcpl-flat}
  Let $A$ be a commutative ring and let $M$ be a classical $A$-module. If $M$ is (faithfully) flat, then it is $p$-completely (faithfully) flat.
\end{remark}

\begin{definition}
  \label{ev--desc--graded-flat}
  Let $A_*$ be a graded commutative ring\footnote{Recall that we mean commutative in the literal sense (\cref{ConventionsSection}\cref{in--con--grab}).} and let $M_*$ be a classical graded $A_*$-module. We say that $M_*$ is \emph{(faithfully) flat} if the underlying module $M = \bigoplus_{n \in \Z} M_n$ over the underlying commutative ring $A = \bigoplus_{n \in \Z} A_n$ is (faithfully) flat.\footnote{One could alternatively define these notions by imitating the usual definitions internally to the graded setting; see for example \cite[Definition 3.6]{dirac-i}. It is not hard to show that this gives equivalent notions.} We extend the notion of $p$-complete (faithful) flatness from \cref{ev--desc--pcpl-flat} to the graded setting in the same manner.
\end{definition}

\begin{remark}
  \label{ev--desc--pcpl-flat-bdd-torsion}
  We recall from \cite[Lemma 4.7]{BMS} that, given a commutative ring $A$ with bounded $p$-power torsion and $M \in \D(A)$ that is $p$-complete and $p$-completely flat, $M$ is concentrated in homological degree $0$, where it is classically $p$-adically complete and has bounded $p$-power torsion. The same goes through when $A$ is a graded commutative ring (with bounded $p$-power torsion as in \cref{ev--def--bdd-torsion}) and $M$ is also graded, by the same argument.
\end{remark}

We now define the relevant notions of flatness in the setting of even $\E_\infty$-rings.

\begin{definition}
  \label{ev--desc--flat}
  We say that a map of even $\E_\infty$-rings $f : A \to B$ is \emph{faithfully flat} (resp. \emph{$p$-completely faithfully flat}) if the induced map of graded commutative rings $\pi_{2*}(f) : \pi_{2*}(A) \to \pi_{2*}(B)$ is faithfully flat (resp. $p$-completely faithfully flat) in the sense of \cref{ev--desc--graded-flat}.
\end{definition}

\begin{warning}
  \label{ev--desc--lurie-flat}
  The notion of (faithful) flatness in \cref{ev--desc--flat} is distinct from that found in \cite[Definition 7.2.2.10]{ha} and \cite[Definition D.4.4.1]{sag}, the notion found there being stronger. It is the notion defined above that is used throughout this paper (except where explicitly stated otherwise).
\end{warning}

\begin{proposition}
  \label{ev--desc--pcpl-flat-homotopy}
  Let $A \to B$ be a $p$-completely faithfully flat map in $\CAlg_p^\ev$. Then for any map $A \to C$ in $\CAlg_p^\ev$, the $p$-completed pushout of $\E_\infty$-rings $\cpl{(B \otimes_A C)}_p$ also lies in $\CAlg_p^\ev$, and has homotopy groups given by $\cpl{(\pi_*(B) \otimes^{\L}_{\pi_*(A)} \pi_*(C))}_p$.
\end{proposition}

\begin{proof}
  The filtered object $\tau_{\ge 2\star}(B) \otimes_{\tau_{\ge 2\star}(A)} \tau_{\ge 2\star}(C)$ is complete, with associated graded object $\Sigma^{2*}(\pi_{2*}(B) \otimes^{\L}_{\pi_{2*}(A)} \pi_{2*}(C))$, and underlying object $B \otimes_A C$. We deduce that the filtered object $\cpl{(\tau_{\ge 2\star}(B) \otimes_{\tau_{\ge 2\star}(A)} \tau_{\ge 2\star}(C))}_p$ is complete, with associated graded object $\Sigma^{2*}\cpl{(\pi_{2*}(B) \otimes^{\L}_{\pi_{2*}(A)} \pi_{2*}(C))}_p$, and underlying object carrying a canonical map to $\cpl{(B \otimes_A C)}_p$ which becomes an equivalence after $p$-completion. By stability of $p$-complete flatness under derived base change, $\cpl{(\pi_{2*}(B) \otimes^{\L}_{\pi_{2*}(A)} \pi_{2*}(C))}_p$ is $p$-completely flat over $\pi_{2*}(C)$, and $\pi_{2*}(C)$ is assumed to have bounded $p$-power torsion, so by \cref{ev--desc--pcpl-flat-bdd-torsion}, $\cpl{(\pi_{2*}(B) \otimes^{\L}_{\pi_{2*}(A)} \pi_{2*}(C))}_p$ is concentrated in homological degree $0$ and has bounded $p$-power torsion. It follows that the spectral sequence associated to the filtered object $\cpl{(\tau_{\ge 2\star}(B) \otimes_{\tau_{\ge 2\star}(A)} \tau_{\ge 2\star}(C))}_p$ collapses to give the desired result (note that the collapse implies that the underlying object is $p$-complete, and hence equivalent to $\cpl{(B \otimes_A C)}_p$).
\end{proof}

As usual, the above notions of flatness give rise to Grothendieck topologies.

\begin{definition}
  \label{ev--desc--flat-topology-def}
  We say that a sieve on $A \in (\CAlg^{\ev})^{\op}$ is a \emph{flat covering sieve} (resp. \emph{$p$-completely flat covering sieve}) if it contains a finite collection of maps $\{A \to B_i\}_{1\le i \le n}$ such that the map $A \to \prod_i B_i$ is faithfully flat (resp. $p$-completely faithfully flat).
\end{definition}

\begin{proposition}
  \label{ev--desc--flat-topology-prop}
  The flat covering sieves of \cref{ev--desc--flat-topology-def} define a Grothendieck topology on $(\CAlg^\ev)^\op$ and the $p$-completely flat covering sieves of \cref{ev--desc--flat-topology-def} define a Grothendieck topology on $(\CAlg^\ev_p)^\op$.
  
  For a category $\mathcal{C}$ admitting small limits, a functor $F: \CAlg^\ev \to \mathcal{C}$ is a sheaf for the former topology if and only if the following conditions are satisfied:
  \begin{enumerate}
  \item $F$ preserves finite products.
  \item For every faithfully flat map of even $\E_\infty$-rings $A \to B$,
    the canonical map
    \[
      F(A) \to \lim_{\Delta} F(B^{\otimes_A\bullet+1})
    \]
    is an equivalence.
  \end{enumerate}
  The analogous claim holds for the latter topology as well.
\end{proposition}

\begin{proof}
  The proofs of \cite[A.3.2.1, A.3.3.1]{sag} carry over verbatim (the only pushouts required to exist are those along faithfully flat or $p$-completely faithfully flat maps, which do indeed exist, using \cref{ev--desc--pcpl-flat-homotopy} in the $p$-complete case.)
\end{proof}

\begin{definition}
  \label{ev--desc--flat-topology-name}
  We refer to the Grothendieck topologies of \cref{ev--desc--flat-topology-prop} as the \emph{flat topology} on $\CAlg^\ev$ and the \emph{$p$-completely flat topology} on $\CAlg^\ev_p$.
  
  Since pushouts in $(\CAlg^\ev)^\Bcir$ (resp. $(\CAlg^\ev_p)^\Bcir$ and $\CycCAlg^\ev_p$) are
  computed in $\CAlg^\ev$ (resp. $\CAlg_p^\ev$), the above induce topologies on $\smash{(\CAlg^\ev)^\Bcir}$, $\smash{(\CAlg^\ev_p)^\Bcir}$, and $\CycCAlg^\ev_p$, which we call by the same names.
\end{definition}

We now turn to studying the descent properties of our various filtrations.
We will need the following well known result.

\begin{lemma}
  \label{ev--desc--orientation}
  Let $A$ be an even $\mathbb{E}_{\infty}$-ring with $\cir$-action. Then:
  \begin{enumerate}
  \item \label{ev--desc--orientation--class}
    There is a class $t \in \pi_{-2}(A^{\h\cir})$ which is a nonzerodivisor in $\pi_*(A^{\h\cir})$ and such that the map $\pi_*(A^{\h\cir}) \to \pi_*(A)$ is given by quotienting by $t$ and the map $\pi_*(A^{\h\cir}) \to \pi_*(A^{\tate\cir})$ is given by inverting $t$.
  \item \label{ev--desc--orientation--equivalences}
    For any $\cir$-equivariant $A$-module $M$, the canonical maps
    \[
      M^{\h\cir} \otimes_{A^{\h\cir}} A \to M, \qquad
      M^{\h\cir} \otimes_{A^{\h\cir}} A^{\tate\cir} \to M^{\tate\cir},
    \]
    are equivalences, and $M^{\h\cir}$ is $A$-complete as an $A^{\h\cir}$-module (equivalently, $t$-complete for any choice of $t$ as in \cref{ev--desc--orientation--class}).
  \end{enumerate}
\end{lemma}

\begin{proof}
  \begin{enumerate}[leftmargin=*]
  \item The evenness of $A$ guarantees the collapse of the homotopy fixed point spectral sequence for $A^{\h\cir}$ and the Tate spectral sequence for $A^{\tate\cir}$, from which the claim follows.

  \item We argue as in \cite[Lemma IV.4.12]{NikolausScholze}. We may replace $A$ by $\tau_{\ge 0}A$ and reduce to the case when $A$ is connective. Let $t$ be as in \cref{ev--desc--orientation--class}. Then we have an equivalence of $A^{\h\cir}$-modules $A \iso A^{\h\cir}/t$. In particular, $A$ is a perfect $A^{\h\cir}$-module and hence $(-)\otimes_{A^{\h\cir}}A$ commutes with all limits and colimits. Thus, using the equivalences
  \[
    M^{\h\cir} \iso \colim {(\tau_{\ge n}M)^{\h\cir}}, \quad
    M^{\h\cir} \iso \lim {(\tau_{\le m}M)^{\h\cir}},
  \]
  we may reduce the claim about the first map to the case where $M$ is discrete, which follows by direct calculation.

  To prove the claim about the second map, it suffices by \cref{ev--desc--orientation--class} and \cite[Theorem I.4.1(iii)]{NikolausScholze} to prove that the fiber of $M^{\h\cir} \to M^{\h\cir}[t^{-1}]$ commutes with all colimits in the variable $M$. The fiber is given, up to suspension, by the colimit of the functors $M \mapsto M^{\h\cir}/t^n$. For $n=1$, this functor commutes with colimits by what has been said above, and by induction we see that each functor commutes with colimits. The result follows.

  Finally, we show that $M^{\h\cir}$ is $A$-complete, or as is equivalent by \cref{ev--desc--orientation--class}, is $t$-complete. Since $M^{\h\cir} \iso \lim {(\tau_{\le n}M)^{\h\cir}}$, we may reduce to the case where $M$ is bounded above. Then the terms $\Sigma^{-2n}M^{\h\cir}$ become increasingly coconnective and hence 
  \[
    \lim {(\cdots \to \Sigma^{-2n}M^{\h\cir} \stackrel{t}{\to} \Sigma^{-2n+2}M^{\h\cir}
      \to \cdots \to M^{\h\cir})} \iso 0. \qedhere
  \]
  \end{enumerate}
\end{proof}

\begin{lemma}
  \label{ev--desc--desc}
  \begin{enumerate}[leftmargin=*]
  \item \label{ev--desc--desc--plain}
    The functor $\pi_{2*} : \CAlg^\ev \to \GrSpt$ is a sheaf for the flat topology, and it restricts to a sheaf for the $p$-completely flat topology on $\CAlg^\ev_p$.
  \item \label{ev--desc--desc--ht}
    The functors $\pi_{2*}((-)^{\h\cir}), \pi_{2*}((-)^{\tate\cir}): (\CAlg^\ev)^{\Bcir} \to \GrSpt$ are sheaves for the flat topology, and they restrict to sheaves for the $p$-completely flat topology on $(\CAlg^\ev_p)^{\Bcir}$.
  \item \label{ev--desc--desc--ht-stupid}
    The functors $\pi_{2*}(\Sigma^\filledsquare(\pi_{\filledsquare}(-))^{\h\cir}), \pi_{2*}(\Sigma^\filledsquare(\pi_\filledsquare(-))^{\tate\cir}) : (\CAlg^\ev)^{\Bcir} \to \BiGrSpt$
    are sheaves for the flat topology, and they restrict to sheaves for the $p$-completely flat topology on $(\CAlg^\ev_p)^{\Bcir}$.
  \end{enumerate}
\end{lemma}

\begin{proof}
  We will prove the $p$-complete statements; the integral statements can be addressed similarly. We begin by proving \cref{ev--desc--desc--plain}. Let $A \to B$ be a $p$-completely faithfully flat map in $\CAlg_p^\ev$. We need to prove that the canonical map
  \[
    \pi_{2*}(A) \to 
    \lim_\Delta \pi_{2*}({\cpl{(B^{\otimes_A\bullet+1})}_p}) \iso \lim_\Delta {\cpl{(\pi_{2*}(B)^{\otimes^{\L}_{\pi_{2*}(A)}\bullet+1})}_p}
  \]
  (where the limits are taken in $\GrSpt$ and the identification follows from \cref{ev--desc--pcpl-flat-homotopy}) is an equivalence. Both sides being $p$-complete, it suffices to show that it is an equivalence after derived base change along $\Z \to \Z/p$. But, by \cref{ev--desc--pcpl-flat-alt},
  \[
    \pi_{2*}(A) \otimes^\L_\Z \Z/p \to
    \pi_{2*}(B) \otimes^\L_\Z \Z/p
  \]
  is faithfully flat in the sense of \cite[Definition D.4.4.1]{sag} (after forgetting the gradings), so the claim follows from faithfully flat descent (\cite[Theorem D.6.3.5]{sag}).

  Let us now prove \cref{ev--desc--desc--ht}. Let $A \to B$ be a $p$-completely faithfully flat map in $(\CAlg_p^\ev)^{\Bcir}$. We need to prove that the canonical maps
  \[
    \pi_{2*}(A^{\h\cir}) \to  \lim_\Delta \pi_{2*}((\cpl{(B^{\otimes_A\bullet+1})}_p)^{\h\cir}),
    \quad \pi_{2*}(A^{\tate\cir}) \to  \lim_\Delta \pi_{2*}((\cpl{(B^{\otimes_A\bullet+1})}_p)^{\tate\cir})
  \]
  are equivalences (each limit taken in $\GrSpt$). Choose $t \in \pi_{-2}(A^{\h\cir})$ as in \cref{ev--desc--orientation}\cref{ev--desc--orientation--class}. For the $(-)^{\h\cir}$ case, by \cref{ev--desc--orientation}\cref{ev--desc--orientation--equivalences}, both sides are $t$-complete, so
  it suffices to prove the claim after taking the derived base change along $\Z[t] \to \Z$. By \cref{ev--desc--orientation}\cref{ev--desc--orientation--equivalences}, this recovers the map
  \[
    \pi_{2*}(A) \to 
    \lim_\Delta \pi_{2*}({\cpl{(B^{\otimes_A\bullet+1})}_p}).
  \]
  Thus, the $(-)^{\h\cir}$ case follows from \cref{ev--desc--desc--plain}, proved above. The $(-)^{\tate\cir}$ case then follows from the $(-)^{\h\cir}$ case: by \cref{ev--desc--orientation}\cref{ev--desc--orientation--equivalences}, the $(-)^{\tate\cir}$ diagram is obtained from the $(-)^{\h\cir}$ diagram by inverting $t$, and our cosimplicial diagram is one of discrete abelian groups, so taking its limit commutes with inverting $t$.

  Finally, \cref{ev--desc--desc--ht-stupid} can be proved by an argument very similar to the one just used to prove \cref{ev--desc--desc--ht}.
\end{proof}

\begin{theorem}
  \label{ev--desc--fil-desc}
  \begin{enumerate}[leftmargin=*]
  \item \label{ev--desc--fil-desc--plain}
    The functor $\tau_{\ge 2\star}(-) : \CAlg^\ev \to \FilCAlg$ is a sheaf for the flat topology, and it restricts to a sheaf for the $p$-completely flat topology on $\CAlg^\ev_p$.
  \item \label{ev--desc--fil-desc--ht}
    The functors $\tau_{\ge 2\star}((-)^{\h\cir}), \tau_{\ge 2\star}((-)^{\tate\cir}) : (\CAlg^\ev)^{\Bcir} \to \FilSpt$ are sheaves for the flat topology, and they restrict to sheaves for the $p$-completely flat topology on $(\CAlg^\ev_p)^{\Bcir}$.
  \item \label{ev--desc--fil-desc--nygaard}
    The functors $\tau_{\ge 2\star}((\tau_{\ge \filledsquare}(-))^{\h\cir}), \tau_{\ge 2\star}((\tau_{\ge \filledsquare}(-))^{\tate\cir}) : (\CAlg^\ev)^{\Bcir} \to \BiFilCAlg$ are sheaves for the flat topology, and they restrict to sheaves for the $p$-completely flat topology on $(\CAlg^\ev_p)^{\Bcir}$.
  \item \label{ev--desc--fil-desc--tc} The functor
    \[
      \fib(\varphi-\can : \tau_{\ge 2\star}((-)^{\h\cir}) \to \tau_{\ge 2\star}((-)^{\tate\cir})) : \CycCAlg_p^\ev \to \FilCAlg
    \]
    is a sheaf for the $p$-completely flat topology.
  \end{enumerate}
\end{theorem}

\begin{proof}
  For claims \cref{ev--desc--fil-desc--plain,ev--desc--fil-desc--ht,ev--desc--fil-desc--nygaard}, all the filtrations are complete, so it suffices to check after passage to associated graded (in both directions in the case of \cref{ev--desc--fil-desc--nygaard}). Then these claims follow from \cref{ev--desc--desc}. Claim \cref{ev--desc--fil-desc--tc} follows from claim \cref{ev--desc--fil-desc--ht}.
\end{proof}

The upshot of \cref{ev--desc--fil-desc} is the following collection of descent statements for the even filtration and its variants.

\begin{definition}
  \label{ev--desc--eff}
  A map of $\E_\infty$-rings $A \to B$ is:
  \begin{enumerate}
  \item \emph{eff} (\emph{evenly faithfully flat}) if for any even $\E_\infty$-ring $C$ and map of $\E_\infty$-rings $A \to C$, the pushout $B \otimes_A C$ is even and faithfully flat over $C$.
  \item \emph{$p$-completely eff} if for any even $\E_\infty$-ring $C$ that is $p$-complete and of bounded $p$-power torsion, and any map of $\E_\infty$-rings $A \to C$, the $p$-completed pushout $\cpl{(B \otimes_A C)}_p$ is even and $p$-completely faithfully flat over $C$.
  \item \emph{evenly free} if for any nonzero even $\E_\infty$-ring $C$ and map $A \to C$, the pushout $B \otimes_A C$ is equivalent as a $C$-module to a nonzero direct sum of even shifts of $C$.
  \item \emph{$p$-completely evenly free} if for any nonzero even $\E_\infty$-ring $C$ that is $p$-complete and of bounded $p$-power torsion, and map of $\E_\infty$-rings $A \to C$, the $p$-completed pushout $\cpl{(B \otimes_A C)}_p$ is equivalent as a $C$-module to the $p$-completion of a nonzero direct sum of even shifts of $C$.
  \end{enumerate}
  
  A map of $\E_\infty$-rings with $\cir$-action or a map of $p$-typical cyclotomic $\E_\infty$-rings is said to be \emph{eff} or \emph{$p$-completely eff} or \emph{evenly free} or \emph{$p$-completely evenly free} if the underlying map of $\E_\infty$-rings is.
\end{definition}

\begin{remark}
  \label{ev--desc--eff-implies-pcpl-eff}
  Let $f : A \to B$ be a map of $\E_\infty$-rings. If $f$ is ($p$-completely) evenly free, then $f$ is ($p$-completely) eff. If $f$ is eff (resp. evenly free), then it is $p$-completely eff (resp. $p$-completely evenly free).
\end{remark}

\begin{corollary}[Eff descent]
  \label{ev--desc--eff-desc}
  \begin{enumerate}[leftmargin=*]
  \item \label{ev--desc--eff-desc--plain}
    For $A \to B$ an eff map of $\E_\infty$-rings, the canonical map
    \[
      \fil^\star_{\ev}(A) \to \lim_{\Delta} \fil^\star_{\ev}(B^{\otimes_A\bullet+1})
    \]
    is an equivalence. For $A \to B$ a $p$-completely eff map of $p$-complete $\E_\infty$-rings, the canonical map
    \[
      \fil^\star_{\ev,p}(A) \to \lim_{\Delta} \fil^\star_{\ev,p}(\cpl{(B^{\otimes_A\bullet+1})}_p)
    \]
    is an equivalence.
  \item \label{ev--desc--eff-desc--ht}
    For $A \to B$ an eff map of $\E_\infty$-rings with $\cir$-action, the canonical maps
    \begin{align*}
      &\fil^\star_{\ev,\h\cir}(A) \to \lim_{\Delta} \fil^\star_{\ev,\h\cir} (B^{\otimes_A\bullet+1}), \\
      &\fil^\star_{\ev,\tate\cir}(A) \to \lim_{\Delta} \fil^\star_{\ev,\tate\cir}(B^{\otimes_A\bullet+1})
    \end{align*}
    are equivalences. For $A \to B$ a $p$-completely eff map of $p$-complete $\E_\infty$-rings with $\cir$-action, the canonical maps
    \begin{align*}
     &\fil^\star_{\ev,p,\h\cir}(A) \to \lim_{\Delta} \fil^\star_{\ev,p,\h\cir}(\cpl{(B^{\otimes_A\bullet+1})}_p),
       \\
      &\fil^\star_{\ev,p,\tate\cir}(A) \to \lim_{\Delta} \fil^\star_{\ev,p,\tate\cir}(\cpl{(B^{\otimes_A\bullet+1})}_p)
      \end{align*}
    are equivalences.
  \item \label{ev--desc--eff-desc--tc}
    For $A \to B$ a $p$-completely eff map of $p$-typical cyclotomic $\E_\infty$-rings, the canonical map
    \[
      \fil^\star_{\ev,p,\TC}(A) \to \lim_{\Delta} \fil^\star_{\ev,p,\TC}(\cpl{(B^{\otimes_A\bullet+1})}_p)
    \]
    is an equivalence.
  \end{enumerate}
\end{corollary}

Eff descent is our main tool for controlling the even filtration. For example, we obtain the following further corollary, which will be applied later in the paper to prove \cref{filteredFrob}; see also the results in \cref{ev--exh} below, which will be used to prove \cref{motConv}.

\begin{corollary}
  \label{ev--desc--compute-tc-as-fiber}
  Let $A$ be a $p$-typical cyclotomic $\mathbb{E}_{\infty}$-ring such that there exist $B \in \CycCAlg_p^\ev$ and a $p$-completely eff map $A \to B$. Then the cyclotomic Frobenius and canonical maps $\varphi, \can : A^{\h\cir} \to A^{\tate\cir}$ refine to maps $\varphi,\can: \fil^{\star}_{\ev,p,\h\cir}(A) \to \fil^{\star}_{\ev,p,\tate\cir}(A)$, naturally in $A$, and there is a canonical equivalence
  \[
    \fil^\star_{\ev,p,\TC}(A) \simeq
    \fib(\varphi - \mathrm{can}: \fil^{\star}_{\ev,p,\h\cir}(A) \to 
    \fil^{\star}_{\ev,p,\tate\cir}(A)).
  \]
\end{corollary}

\begin{proof}
  Let us temporarily denote by $G$ and $H$ the right Kan extension of the functors
  \[
    \tau_{\ge 2\star}((-)^{\h\cir}), \tau_{\ge 2\star}((-)^{\tate\cir}) : \CycCAlg_p^\ev \to \FilCAlg,
  \]
  respectively, along the inclusion $\CycCAlg_p^{\ev} \subseteq \CycCAlg_p$. The maps $\tau_{\ge 2\star}(\varphi)$ and $\tau_{\ge 2\star}(\mathrm{can})$ between these functors right Kan extend to maps $\varphi,\can : G \to H$, and we by definition have a fiber sequence
  \[
    \fil^\star_{\ev,p,\TC}(M) \simeq \fib(\varphi - \mathrm{can}: G \to H).
  \]
  It remains to identify $G(A)$ with $\fil^{\star}_{\ev,p,\h\cir}(A)$ and $H(A)$ with $\fil^{\star}_{\ev,p,\tate\cir}(A)$. In each case, there is by definition a natural map from the latter to the former, and the fact that it is an equivalence follows from \cref{ev--desc--eff-desc}\cref{ev--desc--eff-desc--tc} (which applies also to $G$ and $H$, by the same reasoning) together with the fact that, for $A \to B$ as in the statement, $B^{\otimes_A m+1}$ is even and has bounded $p$-power torsion for each $m \ge 0$.
\end{proof}

\begin{remark}
  \label{rmk-geometric-stack}
  In general, if $\mathcal{T}$ is a subcanonical site, then one says that a sheaf $X$ on $\mathcal{T}$ is a ``geometric stack'' if it admits an effective epimorphism $h_T \to X$ from a representable sheaf and the diagonal $X \to X \times X$ is representable, i.e. whenever $h_{S}, h_{S'} \to X$ are maps from representables, then $h_{S}\times_Xh_{S'}$ is representable. (This notion is heavily dependent on the choice of site presenting the topos $\mathrm{Shv}(\mathcal{T})$.) Those $\mathbb{E}_{\infty}$-rings $A$ with an eff map to an even $\mathbb{E}_{\infty}$-ring give examples of geometric stacks on the site  $(\CAlg^\ev)^{\op}$, and similarly in the equivariant and cyclotomic examples. The good behavior guaranteed in the preceding \cref{ev--desc--compute-tc-as-fiber} (as well as in \cref{ev--exh--main} below) is part of a general paradigm where properties of geometric stacks more closely resemble those of ``affines''.
\end{remark}

We end this subsection with the following key example of an eff map.

\begin{proposition}
  \label{ev--desc--novikov-eff}
  The unit map $\S \to \MU$ is evenly free.
\end{proposition}

\begin{proof}
  For $C$ a nonzero, even $\mathbb{E}_{\infty}$-ring, $C$ is complex orientable and hence $\pi_*(C \otimes_{\mathbb{S}}\mathrm{MU})$ is isomorphic to a polynomial ring over $\pi_*(C)$ on a single generator in each positive, even degree \cite[Lemma 4.5(ii)]{AdamsBlue}. In particular, this is a nonzero free module over $\pi_*(C)$ with generators in even degrees.
\end{proof}

\begin{corollary}[Novikov descent]
  \label{ev--desc--novikov}
  The following statements hold:
  \begin{enumerate}[leftmargin=*]
    \item For $A$ an $\E_\infty$-ring, the canonical map
    \[
      \fil^\star_{\ev}(A) \to \lim_{\Delta} \fil^\star_{\ev}(A \otimes \MU^{\otimes \bullet+1})
    \]
    is an equivalence. For $A$ a $p$-complete $\E_\infty$-ring, the canonical map
    \[
      \fil^\star_{\ev,p}(A) \to \lim_{\Delta} \fil^\star_{\ev,p}(\cpl{(A \otimes \MU^{\otimes \bullet+1})}_p)
    \]
    is an equivalence.
    \item For $A$ an $\E_\infty$-ring with $\cir$-action, the canonical maps
    \begin{align*}
      &\fil^\star_{\ev,\h\cir}(A) \to \lim_{\Delta} \fil^\star_{\ev,\h\cir}(A \otimes \MU^{\otimes \bullet+1}),\\
      &\fil^\star_{\ev,\tate\cir}(A) \to \lim_{\Delta} \fil^\star_{\ev,\tate\cir}(A \otimes \MU^{\otimes \bullet+1})
    \end{align*}
    are equivalences. For $A$ a $p$-complete $\E_\infty$-ring with $\cir$-action, the canonical maps
    \begin{align*}
      &\fil^\star_{\ev,p,\h\cir}(A) \to \lim_{\Delta} \fil^\star_{\ev,p,\h\cir}(\cpl{(A \otimes \MU^{\otimes \bullet+1})}_p),\\
      &\fil^\star_{\ev,p,\tate\cir}(A) \to \lim_{\Delta} \fil^\star_{\ev,p,\tate\cir}(\cpl{(A \otimes \MU^{\otimes \bullet+1})}_p)
    \end{align*}
    are equivalences, where $\MU$ is considered to have trivial $\cir$-action.
    \item For $A$ a $p$-typical cyclotomic $\mathbb{E}_{\infty}$-ring, the canonical map
    \[
      \fil^\star_{\ev,p,\TC}(A) \to \lim_{\Delta} \fil^\star_{\ev,p,\TC}(\cpl{(A \otimes \MU^{\otimes \bullet+1})}_p)
    \]
    is an equivalence, where $\MU$ is considered to have trivial cyclotomic structure.
  \end{enumerate}
\end{corollary}

\begin{proof}
  As evenly free maps are closed under pushouts along arbitrary maps, it follows from \cref{ev--desc--novikov-eff} that the map $A \to A \otimes \MU$ is evenly free, hence eff and $p$-completely eff. The assertions thus follow from \cref{ev--desc--eff-desc}.
\end{proof}


\subsection{The underlying objects of the filtrations}
\label{ev--exh}

Recall from \cref{ev--def--exhaustive} that for any $\E_\infty$-ring $A$, there is a natural map $A \to \colim(\fil^\star_\ev(A))$. Our goal in this subsection is to establish some tools for showing that this map (or its analogue for the variants of the even filtration introduced above) is an equivalence.

\begin{remark}
  \label{ev--exh--underlying}
  For any even $\E_\infty$-ring $B$, the fiber of the map $\tau_{\ge 2n}(B) \to B$ is $(2n-3)$-truncated. As truncatedness is preserved under limits, it follows that for any small diagram of even $\E_\infty$-rings $\{B_i\}_{i \in I}$, the fiber of the map
  \[
    \lim_{i \in I} \tau_{\ge 2n}(B_i) \to \lim_{i \in I} B_i
  \]
  is $(2n-3)$-truncated, and we deduce from this that the map
  \[
    \colim_{n \to -\infty}\left(\lim_{i \in I} \tau_{\ge 2n}(B_i)\right) \to \lim_{i \in I} B_i
  \]
  is an equivalence.

  For example, for any $\E_\infty$-ring $A$, we have an equivalence
  \[
    \colim(\fil^\star_\ev(A)) \iso \lim_{A \to B, B \in \CAlg^\ev} B,
  \]
  and the fiber of the map $\fil^n_\ev(A) \to \colim(\fil^\star_\ev(A))$ is $(2n-3)$-truncated; under this equivalence, the natural map $A \to \colim(\fil^\star_\ev(A))$ identifies with the canonical map
  \[
    A \to \lim_{A \to B, B \in \CAlg^\ev} B.
  \]
\end{remark}

\begin{definition}
  \label{ev--exh--descent}
  We say that a map of $\E_\infty$-rings $A \to B$ \emph{satisfies descent} (resp. \emph{satisfies $p$-complete descent}) if the canonical map $A \to \lim_\Delta B^{\otimes_A \bullet+1}$ (resp. the canonical map $\smash{\cpl{A}_p \to \lim_\Delta \cpl{(B^{\otimes_A \bullet+1})}_p}$) is an equivalence.

  We say that a map of $\cir$-equivariant $\E_\infty$-rings $A \to B$ \emph{satisfies descent} (resp. \emph{satisfies $p$-complete descent}) if the underlying map of $\E_\infty$-rings does, and we say that it \emph{satisfies Tate descent} (resp. \emph{satisfies $p$-complete Tate descent}) if the canonical map $A^{\tate\cir} \to \lim_\Delta (B^{\otimes_A\bullet+1})^{\tate\cir}$ (resp. the canonical map $(\cpl{A}_p)^{\tate\cir} \to \lim_\Delta (\cpl{(B^{\otimes_A\bullet+1})}_p)^{\tate\cir}$) is an equivalence.
\end{definition}

\begin{remark}
  \label{ev--exh--descent-implies-pcpl-descent}
  Since $p$-completion preserves limits, a map of $\E_\infty$-rings $A \to B$ that satisfies descent also satisfies $p$-complete descent. Similarly, a map of connective $\cir$-equivariant $\E_\infty$-rings $A \to B$ that satisfies Tate descent also satisfies $p$-complete Tate descent; here we additionally use the fact that $\cpl{(X^{\tate\cir})}_p \iso (\cpl{X}_p)^{\tate\cir}$ for $X$ a connective $\cir$-equivariant spectrum \cite[\textsection 2.3]{BMS}.
\end{remark}

\begin{proposition}
  \label{ev--exh--main}
  \begin{enumerate}[leftmargin=*]
  \item \label{ev--desc--exhaustivity--plain}
    Let $A$ be an $\E_\infty$-ring. Suppose that there exists a map of $\E_\infty$-rings $A \to B$ that is eff and satisfies descent (resp. is $p$-completely eff and satisfies $p$-complete descent), where $B$ is even (resp. $\cpl{B}_p$ is even and has bounded $p$-power torsion). Then the map $A \to \colim(\fil^\star_\ev(A))$ (resp. the map $\cpl{A}_p \to \colim(\fil^\star_{\ev,p}(\cpl{A}_p))$) is an equivalence.
  \item \label{ev--desc--exhaustivity--h}
    Let $A$ be an $\cir$-equivariant $\E_\infty$-ring. Suppose that there exists a map of $\cir$-equivariant $\E_\infty$-rings $A \to B$ that is eff and satisfies descent (resp. is $p$-completely eff and satisfies $p$-complete descent), where $B$ is even (resp. $\cpl{B}_p$ is even and has bounded $p$-power torsion). Then the map $\smash{A^{\h\cir} \to \colim(\fil^\star_{\ev,\h\cir}(A))}$ (resp. the map $\smash{(\cpl{A}_p)^{\h\cir} \to \colim(\fil^\star_{\ev,p,\h\cir}(\cpl{A}_p))}$) is an equivalence.
  \item \label{ev--desc--exhaustivity--t}
    Let $A$ be an $\cir$-equivariant $\E_\infty$-ring. Suppose that there exists a map of $\cir$-equivariant $\E_\infty$-rings $A \to B$ that is eff and satisfies Tate descent (resp. is $p$-completely eff and satisfies $p$-complete Tate descent), where $B$ is even (resp. $\cpl{B}_p$ is even and has bounded $p$-power torsion). Then the map $\smash{A^{\tate\cir} \to \colim(\fil^\star_{\ev,\tate\cir}(A))}$ (resp. the map $\smash{(\cpl{A}_p)^{\tate\cir} \to \colim(\fil^\star_{\ev,p,\tate\cir}(\cpl{A}_p))}$) is an equivalence.
  \end{enumerate}
\end{proposition}

\begin{proof}
  We will prove the integral statements; the $p$-complete statements can be proved in the same way. 
  \begin{enumerate}[leftmargin=*]
  \item Let $A \to B$ be a map as in the statement. By \cref{ev--desc--eff-desc}\cref{ev--desc--eff-desc--plain},
    \[
      \fil^\star_\ev(A) \iso \lim_\Delta \tau_{\ge 2\star}(B^{\otimes_A\bullet+1}),
    \]
    so by \cref{ev--exh--underlying}, the map $A \to \colim(\fil^\star_\ev(A))$ identifies with the canonical map $A \to \lim_\Delta B^{\otimes_A\bullet+1}$, which is an equivalence by the hypothesis that $A \to B$ satisfies descent.
  \item Let $A \to B$ be a map as in the statement. By \cref{ev--desc--eff-desc}\cref{ev--desc--eff-desc--ht},
    \[
      \fil^\star_{\ev,\h\cir}(A) \iso \lim_\Delta \tau_{\ge 2\star}((B^{\otimes_A\bullet+1})^{\h\cir}),
    \]
    so by \cref{ev--exh--underlying}, the map $A^{\h\cir} \to \colim(\fil^\star_{\ev,\h\cir}(A))$ identifies with the canonical map
    \[
      A^{\h\cir} \to \lim_\Delta {(B^{\otimes_A\bullet+1})^{\h\cir}} \iso (\lim_\Delta B^{\otimes_A\bullet+1})^{\h\cir},
    \]
    which is an equivalence by the hypothesis that $A \to B$ satisfies descent.
  \item Let $A \to B$ be a map as in the statement. By \cref{ev--desc--eff-desc}\cref{ev--desc--eff-desc--ht},
    \[
      \fil^\star_{\ev,\tate\cir}(A) \iso \lim_\Delta \tau_{\ge 2\star}((B^{\otimes_A\bullet+1})^{\tate\cir}),
    \]
    so by \cref{ev--exh--underlying}, the map $A^{\tate\cir} \to \colim(\fil^\star_{\ev,\tate\cir}(A))$ identifies with the canonical map
    \[
      A^{\tate\cir} \to \lim_\Delta {(B^{\otimes_A\bullet+1})^{\tate\cir}},
    \]
    which is an equivalence by the hypothesis that $A \to B$ satisfies Tate descent. \qedhere
  \end{enumerate}
\end{proof}

\begin{proposition}
  \label{ev--exh--connective-descent}
  \begin{enumerate}[leftmargin=*]
  \item \label{ev--exh--connective-descent--plain}
    Let $A \to B$ be a $1$-connective map of connective $\E_\infty$-rings. Then $A \to B$ satisfies descent and satisfies $p$-complete descent.
  \item \label{ev--exh--connective-descent--tate}
    Let $A \to B$ be a $1$-connective map of connective $\cir$-equivariant $\E_\infty$-rings. Then $A \to B$ satisfies Tate descent and satisfies $p$-complete Tate descent.
  \end{enumerate}
\end{proposition}

\begin{proof}
  In both cases, $p$-complete descent follows from descent by \cref{ev--exh--descent-implies-pcpl-descent}, so we need only prove the latter:
  \begin{enumerate}[leftmargin=*]
  \item Let $I := \fib(A \to B)$. For each $n \ge 0$, the fiber of the map $A \to \lim_{\Delta_{\le n}} B^{\otimes_A \bullet+1}$ is equivalent to $I^{\otimes_A n+1}$ (see \cite[Proposition 2.14]{MNN}), which is $(n+1)$-connective since $I$ is $1$-connective. It follows that the fiber of the map $A \to \lim_\Delta B^{\otimes_A\bullet+1} \iso \lim_{n \to \infty} \lim_{\Delta_{\le n}} B^{\otimes_A\bullet+1}$ vanishes.
  \item By \cref{ev--exh--connective-descent--plain}, the map $A \to \lim_{\Delta} B^{\otimes_A \bullet+1}$ is an equivalence. It follows that the map
    \[
      A^{\h\cir} \to \lim_\Delta {(B^{\otimes_A\bullet+1})^{\h\cir}} \iso (\lim_\Delta B^{\otimes_A\bullet+1})^{\h\cir}
    \]
    is an equivalence. Thus, to prove the claim, it suffices to show that the map
    \[
      A_{\h\cir} \to \lim_\Delta {(B^{\otimes_A\bullet+1})_{\h\cir}}
    \]
    is an equivalence. This follows from the same argument as in \cref{ev--exh--connective-descent--plain}, using that $(-)_{\h\cir}$ preserves connectivity. \qedhere
  \end{enumerate}
\end{proof}

\begin{proposition}
  \label{ev--exh--ff-descent}
  \begin{enumerate}[leftmargin=*]
  \item \label{ev--exh--ff-descent--plain}
    Let $A \to B$ be a map of $\E_\infty$-rings that is faithfully flat in the sense of \cite[Definition D.4.4.1]{sag}. Then $A \to B$ satisfies descent and satisfies $p$-complete descent.
  \item \label{ev--exh--ff-descent--tate}
    Let $A \to B$ be a map of connective $\cir$-equivariant $\E_\infty$-rings such that the underlying map of $\E_\infty$-rings is faithfully flat in the sense of \cite[Definition D.4.4.1]{sag}. Then $A \to B$ satisfies Tate descent and satisfies $p$-complete Tate descent.
  \end{enumerate}
\end{proposition}

\begin{proof}
  In both cases, $p$-complete descent follows from descent by \cref{ev--exh--descent-implies-pcpl-descent}, so we need only prove the latter. For \cref{ev--exh--ff-descent--plain}, see \cite[Theorem D.6.3.5]{sag}. Let us prove \cref{ev--exh--ff-descent--tate}. As $A$ and $B$ are connective, it suffices to show that the canonical map of filtered $\E_\infty$-rings
  \[
    (\tau_{\ge\star}(A))^{\tate\cir} \to \lim_\Delta {(\tau_{\ge\star}(B^{\otimes_A \bullet+1}))^{\tate\cir}}
  \]
  is an equivalence. These filtered objects are complete, so it further suffices to show that the map of graded objects
  \[
    (\pi_*(A))^{\tate\cir} \to \lim_\Delta {(\pi_*(B^{\otimes_A \bullet+1}))^{\tate\cir}}
  \]
  is an equivalence. In fact, a stronger statement holds: the augmented cosimplicial object $(\pi_*(A))^{\tate\cir} \to (\pi_*(B^{\otimes_A \bullet+1}))^{\tate\cir}$ is weakly $0$-rapidly converging in the sense of \cite[Definition 2.33]{clausen-mathew}; that is, letting $X_n$ denote the fiber of the map
  \[
    (\pi_*(A))^{\tate\cir} \to \lim_{\Delta_{\le n}} {(\pi_*(B^{\otimes_A \bullet+1}))^{\tate\cir}},
  \]
  we will show that for each $n \ge 0$, the canonical map $X_{n+1} \to X_n$ induces the zero map on homotopy groups (in each grading).

  Note that $(-)^{\tate\cir}$ commutes with the (finite) limit over $\Delta_{\le n}$ and that the flatenss of $A \to B$ implies that we have $\pi_*(B^{\otimes_A\bullet+1}) \iso \pi_*(B)^{\otimes_{\pi_*(A)}\bullet+1}$. Letting $I$ denote the fiber of the map $\pi_*(A) \to \pi_*(B)$, it follows from \cite[Proposition 2.14]{MNN} that
  \[
    X_n \iso (I^{\otimes_{\pi_*(A)} n+1})^{\tate\cir},
  \]
  with the canonical map $X_{n+1} \to X_n$ being induced by the map $I \to \pi_*(A)$. The faithful flatness of $A \to B$ further implies that $\Sigma I \iso \pi_*(B)/\pi_*(A)$ is a flat discrete $\pi_*(A)$-module, and hence $I^{\otimes_{\pi_*(A)} n+1}$ is homologically concentrated in degree $-(n+1)$. It follows that the map $I^{\otimes_{\pi_*(A)} n+2} \to I^{\otimes_{\pi_*(A)} n+1}$ is zero on homotopy groups. Using that the $\cir$-actions on $\pi_*(A)$ and $\pi_*(B^{\otimes_A \bullet+1})$ are trivial, we then conclude that the map
  \[
    X_{n+1} \iso (I^{\otimes_{\pi_*(A)} n+2})^{\tate\cir} \to (I^{\otimes_{\pi_*(A)} n+1})^{\tate\cir} \iso X_n
  \]
  is also zero on homotopy groups.
\end{proof}

\begin{remark}
  \label{ev--exh--universal}
  Consider the following property of a map of connective $\E_\infty$-rings $A \to B$: for any map of connective $\E_\infty$-rings $A \to C$, the induced map $C \to B \otimes_A C$ satisfies descent. We make two observations about this property:
  \begin{enumerate}
  \item It is closed under composition \cite[Lemma 3.1.2]{liu-zheng--six-functors}.
  \item It holds for maps $A \to B$ that are $1$-connective as well as for those that are faithfully flat in the sense of \cite[Definition D.4.4.1]{sag}, by virtue of \cref{ev--exh--connective-descent,ev--exh--ff-descent} and the fact that $1$-connectivity and faithful flatness are stable under the relevant base changes.
  \end{enumerate}
  Similar comments apply for $p$-complete descent, Tate descent, and $p$-complete Tate descent.
\end{remark}


\subsection{$p$-Completion and rationalization}
\label{ev--arith}

In this subsection, we record some statements concerning the interaction of the even filtration with $p$-completion and rationalization.

\begin{proposition}
  \label{ev--arith--p-completion}
  Let $A$ be an $\E_\infty$-ring. Suppose that there exists an eff map of $\E_\infty$-rings $A \to B$ where $B$ is even and $\cpl{B}_p$ is even with bounded $p$-power torsion. Then the canonical map
  \[
    \cpl{(\fil^\star_\ev(A))}_p \to \fil^\star_{\ev,p}(\cpl{A}_p)
  \]
  is an equivalence.
\end{proposition}

\begin{proof}
  Fix a map $A \to B$ as in the statement. By \cref{ev--desc--eff-desc}, we have
  \[
    \fil^\star_\ev(A) \iso \lim_{\Delta} \tau_{\ge 2\star}(B^{\otimes_A\bullet+1}), \quad
    \fil^\star_{\ev,p}(\cpl{A}_p) \iso \lim_{\Delta} \tau_{\ge 2\star}(\cpl{(B^{\otimes_A\bullet+1})}_p),
  \]
  with $B^{\otimes_A m+1}$ and $\cpl{(B^{\otimes_A m+1})}_p$ even for each $m \ge 0$. For any even spectrum $E$ whose $p$-completion is also even, the canonical map $\cpl{(\tau_{\ge 2\star}(E))}_p \to \tau_{\ge 2\star}(\cpl{E}_p)$ is an equivalence, and the claim follows.
\end{proof}

\begin{proposition}
  \label{ev--arith--rationalization}
  Let $A$ be an $\E_\infty$-ring. Suppose that there exists an even $\E_\infty$-ring $B$ and an eff map $A \to B$. Then the canonical map
  \[
    \fil^\star_\ev(A) \otimes \Q \to \fil^\star_\ev(A \otimes \Q)
  \]
  is an equivalence after completion with respect to filtration.
\end{proposition}

\begin{proof}
  Fix a map $A \to B$ as in the statement. Then $B \otimes \Q$ is also even and $A \otimes \Q \to B \otimes \Q$ is also eff. Thus, by \cref{ev--desc--eff-desc}, we have
  \[
    \fil^\star_\ev(A) \iso \lim_{\Delta} \tau_{\ge 2\star}(B^{\otimes_A\bullet+1}), \quad
    \fil^\star_\ev(A \otimes \Q) \iso \lim_{\Delta} \tau_{\ge 2\star}(B^{\otimes_A\bullet+1} \otimes \Q),
  \]
  so it suffices to show that the canonical map
  \[
    \left(\lim_{\Delta}\tau_{\ge 2\star}(B^{\otimes_A\bullet+1})\right) \otimes \Q \to
    \lim_{\Delta} \tau_{\ge 2\star}(B^{\otimes_A\bullet+1} \otimes \Q)
  \]
  is an equivalence after completion, i.e. that the induced map on each associated graded piece
  \[
    \left(\lim_{\Delta} \pi_{2n}(B^{\otimes_A \bullet+1})\right) \otimes \Q \to
    \lim_{\Delta} \pi_{2n}(B^{\otimes_A\bullet+1} \otimes \Q) \iso \lim_{\Delta} {(\pi_{2n}(B^{\otimes_A\bullet+1}) \otimes \Q)}
  \]
  is an equivalence. This follows from the fact that, for any cosimplicial abelian group $X^\bullet$, the canonical map $\pi^i(X^\bullet) \otimes \Q \to \pi^i(X^\bullet \otimes \Q)$ is an isomorphism for each $i$.
\end{proof}

\begin{definition}
  \label{ev--arith--profinite-completion}
  For a spectrum $M$, we let $\cpl{M}$ denote its profinite completion, so $\cpl{M} \iso \prod_p \cpl{M}_p$. We say that a spectrum is \emph{profinitely complete} if the canonical map $M \to \cpl{M}$ is an equivalence.
\end{definition}

\begin{lemma}
  \label{ev--arith--profinite-torsion}
  Let $M$ be a spectrum. Then the following are equivalent:
  \begin{enumerate}
  \item $M$ is torsion (i.e. $M \otimes \Q \iso 0$) and profinitely complete;
  \item for each integer $n$, there exists an integer $d$ such that $\pi_n(M)$ is $d$-torsion.
  \end{enumerate}
\end{lemma}

\begin{proof}
  A spectrum $M$ is profinitely complete (resp. torsion) if and only if each of its homotopy groups is profinitely complete (resp. torsion); indeed, for profinite completeness, one can argue similarly to the proof of \cite[Theorem 7.3.4.1]{sag}, and for torsionness this follows from the isomorphism $\pi_*(M \otimes \Q) \iso \pi_*(M) \otimes \Q$. Thus, we may reduce to the case that $M$ is discrete. It is then easy to prove that the second condition implies the first. Let us prove that the first implies the second.

  Assume that $M$ is discrete, torsion, and profinitely complete. For each prime $p$, the $p$-completion $\cpl{M}_p$ is is a retract of $M$ (since $M \iso \prod_p \cpl{M}_p$), hence is discrete, torsion, and $p$-complete. By \cite[Lemma 0.2]{bhatt--torsion-completion}, it follows that $\cpl{M}_p$ is $p^k$-torsion for some positive integer $k$. To finish, it suffices to show that $\cpl{M}_p \iso 0$ for all but finitely many primes $p$. Let $S$ be the set of primes $p$ such that $\cpl{M}_p$ is nonzero. Then we may find an injection $\prod_{p \in S} \Z/p \inj M$. If $S$ is infinite, then the canonical map $\Z \to \prod_{p \in S} \Z/p$ is an injection, so that we obtain an injection $\Z \inj M$; but this contradicts that $M$ is torsion.
\end{proof}

\begin{lemma}
  \label{ev--arith--profinite-tensor}
  \begin{enumerate}[leftmargin=*]
  \item \label{ev--arith--profinite-tensor--module-vanishing}
    Let $A$ be a connective $\E_\infty$-ring, let $M$ and $N$ be connective $A$-modules, and assume that $\cpl{M} \otimes \Q \iso 0$. Then the canonical map $\cpl{M} \otimes_A \cpl{N} \to \cpl{(M \otimes_A N)}$ is an equivalence.
  \item \label{ev--arith--profinite-tensor--ring-equiv}
    Let $A \to A'$ be a map of connective $\E_\infty$-rings such that the induced map $\cpl{A} \otimes \Q \to \cpl{(A')} \otimes \Q$ is an equivalence. Then, for any connective $A$-module $N$, the canonical map $\cpl{(A')} \otimes_{\cpl{A}} \cpl{N} \to \cpl{(A' \otimes_A N)}$ is an equivalence.
  \end{enumerate}
\end{lemma}

\begin{proof}
  \begin{enumerate}[leftmargin=*]
  \item The map is an equivalence after profinite completion, so it suffices to show that $\cpl{M} \otimes_A \cpl{N}$ is profinitely complete. By \cref{ev--arith--profinite-torsion}, for each integer $n$, there exists an integer $d$ such that $\pi_n(\cpl{M})$ is $d$-torsion. Considering the Tor spectral sequence computing $\pi_*(\cpl{M} \otimes_A \cpl{N})$, we find that the same must be true about these homotopy groups. It then follows from the converse direction of \cref{ev--arith--profinite-torsion} that $\cpl{M} \otimes_A \cpl{N}$ is profinitely complete.
  \item We may replace $A$, $A'$, and $N$ by their profinite completions. The claim can then be deduced from \cref{ev--arith--profinite-tensor--module-vanishing}, applied in the case where $M = \fib(A \to A')$. \qedhere
  \end{enumerate}
\end{proof}

\begin{proposition}
  \label{ev--arith--profinite-rationalization}
  Let $A \to A'$ be a map of connective $\E_\infty$-rings such that the induced map $\cpl{A} \otimes \Q \to \cpl{(A')} \otimes \Q$ is an equivalence. Suppose that there exists a connective $\E_\infty$-ring $B$ and an eff map $A \to B$ such that both $\cpl{B}$ and $\cpl{(A' \otimes_A B)}$ are even and have bounded $p$-power torsion for all primes $p$. Then the canonical map
  \[
    \left(\prod_p\fil^\star_{\ev,p}(\cpl{A}_p)\right) \otimes \Q \to
    \left(\prod_p\fil^\star_{\ev,p}(\cpl{(A')}_p)\right) \otimes \Q
  \]
  is an equivalence after completion with respect to filtration.
\end{proposition}

\begin{proof}
  Fix a map $A \to B$ as in the statement. By \cref{ev--desc--eff-desc}, we have
  \[
    \prod_p \fil^\star_{\ev,p}(\cpl{A}_p) \iso \prod_p \lim_{\Delta} \tau_{\ge 2\star}(\cpl{(B^{\otimes_A\bullet+1})}_p) \iso \lim_{\Delta} \tau_{\ge 2\star}(\cpl{(B^{\otimes_A\bullet+1})}),
  \]
  and as in the proof of \cref{ev--arith--rationalization}, we find that the canonical map
  \[
    \left(\prod_p \fil^\star_{\ev,p}(\cpl{A}_p)\right) \otimes \Q \to \lim_{\Delta} {\tau_{\ge 2\star}(\cpl{(B^{\otimes_A\bullet+1})} \otimes \Q)}
  \]
  is an equivalence after completion with respect to filtration. Letting $B' := A' \otimes_A B$, the above also holds when $A,B$ are replaced by $A',B'$. We may now conclude using the fact that $\cpl{A} \otimes \Q \to \cpl{(A')} \otimes \Q$ being an equivalence implies, by \cref{ev--arith--profinite-tensor}, that $\cpl{(B^{\otimes_A\bullet+1})} \otimes \Q \to \cpl{((B')^{\otimes_{(A')}\bullet+1})} \otimes \Q$ is a (levelwise) equivalence.
\end{proof}

%% file: Wilson.tex
The purpose of this section is to produce cyclotomic bases with
excellent evenness properties. These will be used in the
next section to show that, for certain $\mathbb{E}_{\infty}$-rings,
syntomic cohomology is computable from prismatic cohomology.

\subsection{Preliminaries}

We begin by recording a few useful facts about the following homotopy types, which were first studied by Steve Wilson in his PhD thesis \cite{WilsonThesis}:

\begin{definition}
For each integer $i \ge 0$, the $(2i)$th \emph{Wilson space} is defined to be 
\[W_{2i}=\Omega^{\infty} \Sigma^{2i} \mathrm{MU}.\]
We furthermore define $W_0'$ to be $\Omega^{\infty} \tau_{>0} \mathrm{MU}$, so that $W_0'$ is a path component inside of $W_0$.
\end{definition}

\begin{theorem}[Wilson] \label{WilsonThesisTheorem}
Suppose $R$ is an even $\mathbb{E}_{\infty}$-ring. Then, for each integer $i>0$, $R_*W_{2i}$ is a polynomial $R_*$ algebra, with polynomial generators in even degrees.  Furthermore, $R_*W'_0$ is a polynomial $R_*$ algebra, with polynomial generators in even degrees.
\end{theorem}
\begin{proof}
In \cite[Theorem 3.3, Corollary 3.4]{WilsonThesis}, Steve Wilson proved for every $i>0$ that $\mathrm{H}_*(W_{2i};\mathbb{Z})$ is a polynomial $\mathbb{Z}$-algebra, with polynomial generators in even degrees.  He also proved that $\mathrm{H}_*(W_0';\mathbb{Z})$ is a polynomial $\mathbb{Z}$-algebra with polynomial generators in even degrees.  It follows by the Atiyah–-Hirzebruch spectral sequence that the same is true of $R_*W_{2i}$ and $R_*W'_{0}$.
\end{proof}

As explained in the following proposition, $R_*W_0$ is obtained from the polynomial algebra $R_*W'_0$ by freely adjoining an invertible degree $0$ class.

\begin{proposition} \label{WilsonZeroInvertibleGenerator}
Suppose $R$ is an even $\mathbb{E}_{\infty}$-ring. Then $R_*W_0 \cong (R_*W'_0)[x^{\pm 1}]$, where $|x|=0$.
\end{proposition}

\begin{proof}
This follows from the fact that the natural map
\[W_0 = \Omega^{\infty} \MU \to \Omega^{\infty} \pi_0\MU \cong \mathbb{Z}\]
is split as a map of $H$-spaces. A splitting map may be obtained as $\Omega^2$ of the canonical complex orientation $\mathbb{CP}^{\infty} \to \Omega^{\infty} \Sigma^2 \MU$.
\end{proof}

Before continuing, we recall the basic properties of the $\mathrm{gl}_1$ functor, as summarized in \cite[\textsection 5]{ABGHR}.

\begin{recollection} \label{rec:gl1}
There is an adjunction
\[\Sigma^{\infty}_+\Omega^{\infty}(\--):\Spt_{\ge 0} \leftrightarrows \CAlg:\mathrm{gl}_1(\--),\]
where $\Spt_{\ge 0} \subset \Spt$ denotes the full subcategory of connective spectra. Given an $\mathbb{E}_{\infty}$-ring $R$, $\Omega^{\infty} \mathrm{gl}_1(R)$ sits in a natural pullback square
\[
\begin{tikzcd}
\Omega^{\infty} \mathrm{gl}_1(R) \arrow{r} \arrow{d} & \Omega^{\infty} R \arrow{d} \\
\pi_0(R)^{\times} \arrow[hookrightarrow]{r} & \pi_0(R).
\end{tikzcd}
\]
In particular, $\pi_0(\mathrm{gl}_1(R)) \cong \pi_0(R)^{\times}$, while the higher homotopy groups of $\mathrm{gl}_1(R)$ match the higher homotopy groups of $R$.  Thus, if $\pi_*R$ is even, then $\pi_*\mathrm{gl}_1(R)$ is as well.
\end{recollection}

For each $i \ge 0$, the left adjoint of \cref{rec:gl1} equips $\Sigma^\infty_+ W_{2i}=\Sigma^{\infty}_+\Omega^{\infty} \Sigma^{2i} \MU$ with the structure of an $\mathbb{E}_{\infty}$-ring spectrum. We will need the following fact about these $\mathbb{E}_{\infty}$-rings:

\begin{proposition} \label{prop--Wilson-maps}
Suppose $R$ is a connective, even $\mathbb{E}_{\infty}$-ring, and that $x \in \pi_{2i} R$ is an element in the homotopy groups of $R$ with $i \ge 0$. If $i=0$, we further suppose that $x \in (\pi_0R)^{\times}$ (i.e., that $x$ is a multiplicative unit). Then there exists an $\mathbb{E}_{\infty}$-ring map 
\[f:\Sigma^{\infty}_+ W_{2i} \to R\] such that $x$ is in the image of $\pi_{2i} f$.
\end{proposition}




\begin{proof}
A map $\Sigma^{\infty}_+W_{2i} \to R$ is the same data as a map of spectra $\Sigma^{2i}\MU \to \mathrm{gl}_1(R)$.  We note that there is an extension
\[
\begin{tikzcd}
S^{2i} \arrow{r}{x} \arrow{d} & \mathrm{gl}_1(R) \\
\Sigma^{2i}\MU, \arrow[dashed]{ur}
\end{tikzcd}
\]
because $\MU$ admits an even cell decomposition in the category of spectra and $\pi_*\mathrm{gl}_1(R)$ is concentrated in even degrees.
\qedhere
\end{proof}

\begin{warning}
The map guaranteed to exist by \Cref{prop--Wilson-maps} is not canonical in any sense.  In particular, we do not claim any sort of functorial construction. The same warning applies to many of the other constructions in this section, all made via obstruction theory.
\end{warning}

Since each Wilson space $W_{2i}$ is an $\mathbb{E}_{\infty}$-space, any finite product of Wilson spaces has a canonical $\mathbb{E}_{\infty}$-space structure.  In the category of $\mathbb{E}_{\infty}$-spaces, a finite product is a finite coproduct, and we will more generally have occasion to consider infinite coproducts:

\begin{definition}
A \emph{weak product of Wilson spaces} is an infinite loop space of the form $\Omega^{\infty} (\bigoplus_j \Sigma^{2i_j} \MU)$, where $\{i_j\}$ is a collection of non-negative integers. 
\end{definition}



\begin{proposition} \label{WilsonSurjectivity}
Suppose that $R$ is a connective, even $\mathbb{E}_{\infty}$-ring such that $\pi_0(R)$ is generated as a ring by elements in $(\pi_0R)^{\times}$. Then there exists a weak product of Wilson spaces $W$ together with an $\mathbb{E}_{\infty}$-ring map 
\[f:\Sigma^\infty_+ W \to R\]
such that $\pi_*f$ is surjective.
\end{proposition}

\begin{proof}
For each element $x \in \pi_{2i}\mathrm{gl}_1(R)$, \Cref{prop--Wilson-maps} allows us to choose an $\mathbb{E}_{\infty}$-ring map $\Sigma^{\infty}_+W_{2i} \to R$ whose image in $\pi_*$ contains $x$.  Taking the coproduct of such maps, over all elements $x$ in the homotopy groups of $\mathrm{gl}_1(R)$, yields the result.
\end{proof}

Finally, the following technical result is a key ingredient in our proof that motivic filtrations respect cyclotomic structure:

\begin{theorem} \label{WilsonS1unobstruct}
Let $R$ be an even $\mathbb{E}_{\infty}$-ring with $\cir$-action.   Then, for any $i \ge 0$ and any $\cir$-equivariant $\mathbb{E}_{\infty}$-ring map $\THH(\Sigma^{\infty}_+W_{2i}) \to R$, there exists a factorization
\[
\begin{tikzcd}
\THH(\Sigma^{\infty}_+W_{2i}) \arrow{r} \arrow{d}{\pi} & R \\
\Sigma^{\infty}_+ W_{2i}. \arrow{ur}
\end{tikzcd}
\]
Here, $\pi$ is the canonical $\cir$-equivariant projection $\THH(\Sigma^{\infty}_+W_{2i}) \to \Sigma^{\infty}_+W_{2i}$, where $\Sigma^{\infty}_+W_{2i}$ has trivial $\cir$-action.
\end{theorem}

\begin{proof}
By the universal property of $\THH$, this is asking whether every non-equivariant $\mathbb{E}_{\infty}$-ring map $\Sigma^{\infty}_+W_{2i} \to R$ lifts to a $\cir$-equivariant $\mathbb{E}_{\infty}$-ring map $\Sigma^{\infty}_+W_{2i} \to R$.  The space of non-equivariant $\mathbb{E}_{\infty}$-ring maps is equivalent to the space $\Map(\Sigma^{2i}\MU,\mathrm{gl}_1(R))$ of spectrum maps from $\Sigma^{2i}\MU$ to $\mathrm{gl}_1(R)$.  By the Atiyah--Hirzebruch spectral sequence, this mapping space has even homotopy groups.  It follows that the $\cir$ homotopy fixed point spectral sequence computing the space $\Map(\Sigma^{2i}\MU,\mathrm{gl}_1(R))^{\h\cir}$ of equivariant $\mathbb{E}_{\infty}$-ring maps collapses, and in particular that the natural map $\Map(\Sigma^{2i}\MU,\mathrm{gl}_1(R))^{\h\cir} \to \Map(\Sigma^{2i}\MU,\mathrm{gl}_1(R))$ is a surjection on path components.
\end{proof}

\begin{corollary} \label{WilsonUnobstructed}
Let $R$ be an even $\mathbb{E}_{\infty}$-ring with $\cir$-action and $W$ a weak product of Wilson spaces.  Then, for any $\mathbb{E}_{\infty}$-ring map $\THH(\Sigma^{\infty}_+W) \to R$, there exists a factorization
\[
\begin{tikzcd}
\THH(\Sigma^{\infty}_+W) \arrow{r} \arrow{d}{\pi} & R \\
\Sigma^{\infty}_+ W. \arrow{ur}
\end{tikzcd}
\]
Here, $\pi$ is the canonical $\cir$-equivariant projection $\THH(\Sigma^{\infty}_+W) \to \Sigma^{\infty}_+W$, where $\Sigma^{\infty}_+W$ has trivial $\cir$-action.
\end{corollary}
\begin{proof}
Since $\THH(\Sigma^{\infty}_+\--)$ and $\Sigma^{\infty}_+$ both preserve coproducts of $\mathbb{E}_{\infty}$-spaces, this follows immediately from the preceding theorem.
\end{proof}

\subsection{A sufficient supply of cyclotomic bases} \label{cycbaseSection}

In order to understand cyclotomic structure on $\THH$ of discrete rings, it has proven extremely useful to understand cyclotomic structure on Hochschild homology relative to $\Sigma^{\infty}_+ \mathbb{N}$ \cite{LiuWang, Zpn}.  As noted for example in \cite{LawsonTate}, it is rarely the case that Hochschild homology relative to an $\mathbb{E}_{\infty}$-ring $A$ carries cyclotomic structure.  In such cases, we say that $A$ is a cyclotomic base:

\begin{definition} \label{cycbasedef}
A \emph{cyclotomic base} is an $\mathbb{E}_{\infty}$-ring $A$ together with the data of a commutative diagram of $\cir$-equivariant $\mathbb{E}_{\infty}$-rings
\[
\begin{tikzcd}
\THH(A) \arrow{r}{\varphi} \arrow{d}{\pi} &  \THH(A)^{\tate\Cp} \arrow{d}{\pi^{\tate\Cp}} \\
A \arrow{r} & A^{\tate\Cp},
\end{tikzcd}
\]
where $\pi:\THH(A) \to A$ is the canonical projection to $A$ with trivial action.
\end{definition}

We warn the reader that the above definition is not quite the same as the one adopted in \cite[Appendix A]{BeilinsonSquare}, but has been discussed in \cite[\textsection 11.1]{BMS} and \cite[Remark 12.1]{LawsonTate}.  If $A$ is a cyclotomic base, and $R$ is an $\mathbb{E}_{\infty}$-$A$-algebra, then we may use the formula
\[\THH(R/A) \simeq \THH(R) \otimes_{\THH(A)} A\]
to define the structure of a cyclotomic $\mathbb{E}_{\infty}$-ring on $\THH(R/A)$.  By definition, this cyclotomic $\mathbb{E}_{\infty}$-ring structure is compatible with the map $\THH(R) \to \THH(R/A)$.

To date, practically all known examples of cyclotomic bases have been built out of $A=\mathbb{S}$ and $A=\Sigma^{\infty}_+ \mathbb{N}$. In this section, our goal will be to use Wilson spaces to construct many additional examples of cyclotomic bases, by obstruction theory.

\begin{definition}
  \label{wi--strongly-even}
A connective $\mathbb{E}_{\infty}$-ring $A$ is said to be \emph{strongly even} if:
\begin{enumerate}
\item There exists a homotopy commutative ring map $\MU \to A$, such that, as an $\MU$-module, $A$ splits as a direct sum of even suspensions of $\MU$.
\item In the category of $\mathbb{E}_{\infty}$-ring spectra, $A$ admits an even cell decomposition.
\end{enumerate}
\end{definition}

For each $i>0$,  the $\mathbb{E}_{\infty}$-ring $\Sigma^{\infty}_+W_{2i}$ admits an even $\mathbb{E}_{\infty}$ cell decomposition, since $\Sigma^{2i}\MU$ admits an even cell decomposition as a spectrum.  However, for no $i$ is $\Sigma^{\infty}_+ W_{2i}$ strongly even.  Indeed, as a spectrum $\Sigma^{\infty}_+W_{2i}$ contains a unit summand equivalent to $\mathbb{S}$, and so $\Sigma^{\infty}_+W_{2i}$ cannot be a wedge of even suspensions of $\MU$.  At the end of this section, we will prove that at least one strongly even $\mathbb{E}_{\infty}$-ring exists, by direct construction.  First, we list some of the properties enjoyed by strongly even $\mathbb{E}_{\infty}$-rings:

\begin{proposition} \label{MWeff}
Suppose $A$ is strongly even.  Then the unit map $\mathbb{S} \to A$ is evenly free.
\end{proposition}

\begin{proof}
By \cite[Theorem 1.2]{ChadwickMandell}, we may find an $\mathbb{E}_2$-ring map $\MU \to A$ which makes $A$ into a free $\MU$-module.  Now, suppose that $B$ is any even $\mathbb{E}_{\infty}$-ring.  We calculate $B \otimes A \simeq (B \otimes \MU) \otimes_{\MU} A$, where the equivalence is as $\mathbb{E}_1$-ring spectra. Since $B \otimes \MU$ is a free $B$-module, and $A$ is a free $\MU$ module, we learn that $B \otimes A$ is a free $B$-module, on even suspensions of $B$.
\end{proof}



\begin{proposition}
  \label{StronglyEvenUnobstructed}
  Suppose that $A$ is strongly even, and that $R$ is an even $\mathbb{E}_{\infty}$-ring with $\cir$-action.  Then any $\cir$-equivariant $\mathbb{E}_{\infty}$-ring map $\THH(A) \to R$ factors through the projection $\pi:\THH(A) \to A$.
\end{proposition} 

\begin{proof}
  This follows from the assumed even cell decomposition, by arguments analogous to the proof of \cref{WilsonS1unobstruct}. More precisely, the even cell decomposition ensures that the space of non-equivariant $\mathbb{E}_{\infty}$-ring maps from $A$ to $R$ has even homotopy groups.  It follows that the homotopy fixed point spectral sequence computing the space of $\cir$-equivariant $\mathbb{E}_{\infty}$-ring maps collapses, and in particular every non-equivariant $\mathbb{E}_{\infty}$-ring map from $A$ to $R$ admits an $\cir$-equivariant refinement. By the universal property of $\THH$, non-equivariant $\mathbb{E}_{\infty}$-ring maps from $A$ to $R$ correspond to $\cir$-equivariant $\mathbb{E}_{\infty}$-ring maps from $\THH(A) \to R$. Under this correspondence, a $\cir$-equivariant refinement corresponds to a $\cir$-equivariant factorization through the projection $\THH(A) \to A$.  
\end{proof}

\begin{lemma} \label{p-torsion-tate-even}
Suppose that $A$ is an even $\mathbb{E}_{\infty}$-ring such that $\pi_*A$ has no $p$-torsion.  Then $A^{\tate\Cp}$ is even, where $A^{\tate\Cp}$ denotes the Tate construction for the trivial $\Cp$-action on $A$.
\end{lemma}

\begin{proof}
The $\Cp$ Tate spectral sequence has even $\mathrm{E}_2$-page.
\end{proof}

\begin{theorem}
  \label{wi--cyc--strongly-even-base}
  Let $A$ be a strongly even $\E_\infty$-ring. Then $A$ admits the structure of a cyclotomic base.
\end{theorem}

\begin{proof}
  We need to produce a commutative diagram of $\cir$-equivariant $\mathbb{E}_{\infty}$-rings
  \[
    \begin{tikzcd}
      \THH(A) \arrow{d} \arrow{r} & \THH(A)^{\tate\Cp} \arrow{d} \\
      A \arrow[dashed]{r} & A^{\tate\Cp}.
    \end{tikzcd}
  \]
  Since $A$ splits as a direct sum of even suspensions of $\MU$, $A^{\tate\Cp}$ is even by \Cref{p-torsion-tate-even}. The claim then follows from \Cref{StronglyEvenUnobstructed}.
\end{proof}

\begin{theorem} \label{finalAppThm}
  Let $A$ be a cyclotomic base such that $A$ is even and $\pi_*(A)$ has no $p$-torsion, and let $W$ be a weak product of Wilson spaces. Then $A \otimes \Sigma^{\infty}_+ W$ admits the structure of a cyclotomic base, with a map of cyclotomic bases $A \to A \otimes \Sigma^{\infty}_+ W$.
\end{theorem}

\begin{proof}
  We need to produce a commutative diagram of $\cir$-equivariant $\mathbb{E}_{\infty}$-rings
  \[
    \begin{tikzcd}
      \THH(A \otimes \Sigma^{\infty}_+ W) \arrow{d} \arrow{r} & \THH(A \otimes \Sigma^{\infty}_+ W)^{\tate\Cp} \arrow{d} \\
      A \otimes \Sigma^{\infty}_+ W \arrow[dashed]{r} & (A \otimes \Sigma^{\infty}_+ W)^{\tate\Cp},
    \end{tikzcd}
  \]
  compatible with the given such diagram for $A$ itself. Using the equivalence $\THH(A \otimes \Sigma^\infty_+ W) \iso \THH(A) \otimes \THH(\Sigma^\infty_+ W)$, we see that it suffices to produce a commutative diagram of $\cir$-equivariant $\E_\infty$-rings
    \[
    \begin{tikzcd}
      \THH(\Sigma^{\infty}_+ W) \arrow{d} \arrow{r} & \THH(A \otimes \Sigma^{\infty}_+ W)^{\tate\Cp} \arrow{d} \\
      \Sigma^{\infty}_+ W \arrow[dashed]{r} & (A \otimes \Sigma^{\infty}_+ W)^{\tate\Cp}.
    \end{tikzcd}
  \]
  Now, $A \otimes \Sigma^{\infty}_+W$ splits as a direct sum of even suspensions of $A$ by \Cref{WilsonThesisTheorem} and \Cref{WilsonZeroInvertibleGenerator}. Hence, $\pi_*(A \otimes \Sigma^{\infty}_+W)$ is even and has no $p$-torsion. It follows that $(A \otimes \Sigma^{\infty}_+ W)^{\tate\Cp}$ is even, so we may finish by applying \Cref{WilsonUnobstructed}.
\end{proof}

\begin{remark}
  \label{wi--cyc--wilson-base-addendum}
  \begin{enumerate}[leftmargin=*]
  \item Let $A$ be a strongly even $\E_\infty$-ring. Then $A$ is even and has no $p$-power torsion as it splits as a direct sum of even suspensions of $\MU$. That is, it satisfies the hypotheses of \cref{finalAppThm} (for any choice of cyclotomic base structure on $A$, guaranteed to exist by \cref{wi--cyc--strongly-even-base}).
  \item Let $A$ and $W$ be as in \cref{finalAppThm} and let $A \to B$ be any map of cyclotomic bases. Then $B \otimes \Sigma^\infty_+ W$ admits the structure of a cyclotomic base, with a map of cyclotomic bases $B \to B \otimes \Sigma^\infty_+ W$: indeed, we may choose a cyclotomic base structure on $A \otimes \Sigma^\infty_+ W$ as in \cref{finalAppThm} and then take the induced cyclotomic base structure on its pushout along the map $A \to B$.
  \end{enumerate}
\end{remark}

Finally, we prove that at least one strongly even $\mathbb{E}_{\infty}$-ring exists, by direct construction:

\begin{definition}
  \label{wi--cyc--MW}
Fix any map of spectra $\gamma:\MU \to \mathrm{ku}$ that is a surjection on $\pi_*$.
We denote by $\MW$ the $\mathbb{E}_{\infty}$-ring that is the Thom spectrum associated to the map 
\[\Omega^{\infty} \Sigma^2 \gamma:W_2 \to \mathrm{BU}\]
\end{definition}

\begin{theorem}
The $\mathbb{E}_{\infty}$-ring $\MW$ is strongly even.
\end{theorem}

\begin{proof}
By Thomifying $\Omega^{\infty} \Sigma^2(\--)$, applied to an even cell decomposition of the spectrum $\MU$, we see that $\MW$ has an even cell decomposition as an $\mathbb{E}_{\infty}$-ring spectrum. It remains to check that there exists a homotopy commutative ring homomorphism $\MU \to \MW$ that makes $\MW$ a free $\MU$-module concentrated in even degrees.

First, we claim that the infinite loop map $\Omega^{\infty} \Sigma^2 \gamma:W_2 \to \mathrm{BU}$ has a double loop map section.  To construct this section, we may take $\Omega^2$ of a section of the pointed space map $\Omega^{\infty}\Sigma^4\gamma:W_4 \to \mathrm{BSU}$.  To see that the latter section exists, we note that \emph{any} map $\mathrm{BSU} \to \mathrm{BSU}$ lifts through $\Omega^{\infty} \Sigma^4 \gamma$.  This is because the obstruction to such a lift is a map from $\mathrm{BSU}$ to a space that, by assumption, has only odd homotopy groups.

Now, the above section induces a splitting of $W_2$, as a double loop space, into the product of $\mathrm{BU}$ and $\mathrm{fib}(\Omega^{\infty} \Sigma^2 \gamma)$.  It follows that $\MW$ is, as an $\mathbb{E}_2$-ring, the tensor product of $\MU$ and $ \Sigma^{\infty}_+\mathrm{fib}(\Omega^{\infty}\Sigma^2 \gamma)$. It therefore suffices to prove that $\MU_*\mathrm{fib}(\Omega^{\infty}\Sigma^2 \gamma)$ is a free $\MU_*$-module concentrated in even degrees.  By the Atiyah-–Hirzebruch spectral sequence, it suffices to check that $\mathrm{H}_*\left(\mathrm{fib}(\Omega^{\infty}\Sigma^2 \gamma);\mathbb{Z}\right)$ is a free $\mathbb{Z}$-module concentrated in even degrees.  This is a submodule of $H_*\left(W_2;\mathbb{Z}\right)$, by the above double loop space splitting, and we finish by noting that submodules of free $\mathbb{Z}$-modules are free $\mathbb{Z}$-modules. \qedhere

\end{proof}

\begin{remark}
Though we will not need it, we note that results of Steve Wilson imply that $\pi_*\MW_{(p)}$ is a polynomial $\BP_*$-algebra. To see this, first observe that $\mathrm{fib}(\Omega^{\infty} \Sigma^2 \gamma) \simeq \Omega^2\mathrm{fib}(\Omega^{\infty} \Sigma^4 \gamma)$. Also, $\mathrm{fib}(\Omega^{\infty} \Sigma^4\gamma)$ has finitely generated, free, and even homotopy and $\mathbb{Z}_{(p)}$-homology groups (because is a pointed space retract of $W_4=\Omega^{\infty} \Sigma^4 \MU$).  Reading \cite[Theorem 6.2]{WilsonII}, we learn that $\mathrm{fib}(\Omega^{\infty} \Sigma^4 \gamma) \simeq \prod_{i} Y_{2k_i}$ is a product of spaces that Wilson calls $Y_{2k_i}$.  Furthermore, for each integer $j>0$, only finitely many $i\le j$ appear in the product, and all other terms in the product are $2j$-connective. It then suffices to observe that each $\mathrm{BP}_*\Omega^2 Y_{k_i}$ is a polynomial $\mathrm{BP}_*$-algebra, which follows by combining \cite[Corollary 5.2]{WilsonII} and \cite[Corollary 6.8]{WilsonII}.
\end{remark}

%% file: MotivicFiltration.tex
\section{The motivic filtration}

In this section, we define and prove basic properties of the motivic filtrations on $\THH(R)$, $\TC^{-}(R)$, $\TP(R)$, and $\cpl{\TC(R)}_p$, whenever $R$ is a well-behaved commutative ring spectrum.  We begin by specifying that well-behaved means \emph{chromatically quasisyntomic}, which is meant to be a generalization of the main definition of \cite[\textsection 4]{BMS} and \cite[Appendix C]{BL} to the setting of not necessarily discrete $\E_\infty$-ring spectra.


\subsection{Chromatically quasisyntomic $\E_\infty$-rings}
\label{mot--xq}

Let us first recall the notion of quasisyntomicity in ordinary algebra, introduced by Bhatt--Morrow--Scholze \cite[\textsection 4]{BMS} (though we will use some terminology slightly differently, due to our interest in the integral setting, in contrast to their focus on the $p$-complete setting).

\begin{definition}
  \label{mot--xq--classical-quasilci}
  We say that a map of commutative rings $A \to B$ is:
  \begin{enumerate}
  \item \emph{quasismooth} (resp. \emph{$p$-quasismooth}) if it is flat (resp. $p$-completely flat) and $\smash{\L^\alg_{B/A}}$ has Tor-amplitude in $[0,0]$ (resp. $p$-complete Tor-amplitude in $[0,0]$) over $B$;
  \item \emph{quasi-lci} (resp. \emph{$p$-quasi-lci}) if $\smash{\L^\alg_{B/A}}$ has Tor-amplitude in $[0,1]$ (resp. $p$-complete Tor-amplitude in $[0,1]$) over $B$.
  \item a \emph{quasiregular quotient} (resp. \emph{$p$-quasiregular quotient}) if it is surjective and quasi-lci (resp. $p$-quasi-lci).
  \end{enumerate}
\end{definition}

\begin{definition}
  \label{mot--xq--classical-quasisyntomic}
  We say that a commutative ring $R$ is \emph{$p$-quasisyntomic} if $R$ has bounded $p$-power torsion and the unit map $\Z \to R$ is $p$-quasi-lci, and is \emph{integrally quasisyntomic} if $R$ has bounded $p$-power torsion for all primes $p$ and the unit map $\Z \to R$ is quasi-lci.
\end{definition}

\begin{proposition}
  \label{mot--xq--classical-quasiregular-equivalence}
  Let $S \to R$ be a quasiregular quotient (resp. $p$-quasiregular quotient) of commutative rings. Then $\smash{\L^\alg_{R/S}}$ has Tor-amplitude in $[1,1]$ (resp. $p$-complete Tor-amplitude in $[1,1]$) over $R$.
\end{proposition}

\begin{proof}
  By definition, $\L^\alg_{R/S}$ has Tor-amplitude in $[0,1]$ (resp. $p$-complete Tor-amplitude in $[0,1]$). Moreover, $\pi_0(\smash{\L^\alg_{R/S}}) \iso \Omega^1_{R/S} \iso 0$ by the surjectivity of $S \to R$, so $\smash{\L^\alg_{R/S}}$ is also $1$-connective. The claim follows.
\end{proof}

\begin{proposition}
  \label{mot--xq--classical-quasilci-permanence}
  Suppose given a commutative diagram of commutative rings
  \[
    \begin{tikzcd}
      &
      S \ar[dr, "h"] &
      \\
      k \ar[ur, "g"] \ar[rr, "f"] &
      &
      R
    \end{tikzcd}
  \]
  in which $g$ is ($p$-)quasismooth. Then $f$ is ($p$-)quasi-lci if and only if $h$ is ($p$-)quasi-lci.
\end{proposition}

\begin{proof}
  This follows from considering the transitivity fiber sequence
  \[
    R \otimes^\L_S \L^\alg_{S/k} \to \L^\alg_{R/k} \to \L^\alg_{R/S}. \qedhere
  \]
\end{proof}

\begin{remark}
  \label{mot--xq--classical-quasilci-factorization}
  Any map of commutative rings $k \to R$ may be factored by maps of commutative rings $k \to S \to R$ where $S$ is a polynomial $k$-algebra and $S \to R$ is surjective. By \cref{mot--xq--classical-quasilci-permanence}, if $k \to R$ is ($p$-)quasi-lci, then the map $S \to R$ is moreover a ($p$-)quasiregular quotient.
\end{remark}

We now move on to extend the above notions to the setting of even $\E_\infty$-rings.

\begin{definition}
  \label{mot--xq--even-quasilci}
  We say that a map of graded commutative rings\footnote{Recall that we mean commutative in the literal sense (\cref{ConventionsSection}\cref{in--con--grab}).} $A_* \to B_*$ is \emph{($p$-)quasismooth} or \emph{($p$-)quasi-lci} or a \emph{($p$-)quasiregular quotient}  if the map of underlying commutative rings $\bigoplus_{n \in \Z} A_n \to \bigoplus_{n \in \Z} B_n$ is so. We say that a map of even $\E_\infty$-rings $A \to B$ is \emph{($p$-)quasismooth} or \emph{($p$-)quasi-lci} or a \emph{($p$-)quasiregular quotient} if the induced map of graded commutative rings $\pi_{2*}(A) \to \pi_{2*}(B)$ is so.
\end{definition}

\begin{definition}
  \label{mot--xq--even-quasisyntomic}
  We say that an even $\E_\infty$-ring $R$ is \emph{$p$-quasisyntomic} if $R$ has bounded $p$-power torsion (\cref{ev--def--bdd-torsion}) and the unit map $\Z \to \pi_{2*}(R)$ is $p$-quasi-lci, and is \emph{integrally quasisyntomic} if $R$ has bounded $p$-power torsion for all primes $p$ and the unit map $\Z \to \pi_{2*}(R)$ is quasi-lci.
\end{definition}

Though we have defined the notion of ``quasi-lci'' for maps of even $\E_\infty$-rings purely at the level of homotopy groups, it has (at least in the connective setting) the following stronger characterization (\cref{mot--xq--even-polynomial-surjection}), almost parallel to the situation in classical algebra (\cref{mot--xq--classical-quasilci-factorization}).

\begin{definition}
  \label{mot--xq--ind-zariski}
  We say that a map of $\E_\infty$-rings $A \to B$ is a \emph{Zariski localization} if it isomorphic to one of the form $A \to \prod_{i=1}^n A[f_i^{-1}]$ for some $f_1,\ldots,f_n \in \pi_0(A)$, and we say that it is an \emph{ind--Zariski localization} if it can be obtained as a filtered colimit of Zariski localizations.
\end{definition}

\begin{lemma}
  \label{mot--xq--unit-generators}
  Let $A$ be an $\E_\infty$-ring. Then there exists a faithful ind--Zariski localization of $\E_\infty$-rings $A \to B$ such that $\pi_0(B)$ is generated by units as a commutative ring.
\end{lemma}

\begin{proof}
  For each finite subset $S \subseteq \pi_0(A)$ and each function $\sigma : S \to \{0,1\}$, define
  \[
    s_\sigma := \prod_{s \in S}(s - \sigma(s)) \in \pi_0(A).
  \]
  Then we may take
  \[
    B := \colim_{S \subseteq \pi_0(A)} \prod_{\sigma \in \{0,1\}^S} A[s_\sigma^{-1}]. \qedhere
  \]
\end{proof}

\begin{proposition}
  \label{mot--xq--even-polynomial-surjection}
  Any map of connective, even $\E_{\infty}$-rings $k \to R$ extends to a commutative diagram of $\E_\infty$-rings
  \[
    \begin{tikzcd}
      k \ar[r] \ar[d] &
      S \ar[d] \\
      R \ar[r] &
      R'
    \end{tikzcd}
  \]
  where $R \to R'$ is a faithful ind--Zariski localization, $S$ is connective and even, $\pi_*(S)$ is a localization of a polynomial $\pi_*(k)$-algebra, and $\pi_*(S) \to \pi_*(R')$ is surjective. Moreover:
  \begin{enumerate}
  \item \label{mot--xq--even-polynomial-surjection--quasilci}
    If the map $k \to R$ is ($p$-)quasi-lci, then the map $S \to R'$ is necessarily a ($p$-)quasiregular quotient, and conversely, if there is such a diagram where $S \to R'$ is a ($p$-)quasiregular quotient, then $k \to R$ is ($p$-)quasi-lci.
  \item \label{mot--xq--even-polynomial-surjection--cyclotomic}
    Suppose that $k$ is equipped with a cyclotomic base structure and that there exists a map of cyclotomic bases $A \to k$ where $A$ is even and $\pi_*(A)$ has no $p$-torsion. Then the map $k \to S$ may be taken to be a map of cyclotomic bases.
  \end{enumerate}
\end{proposition}

\begin{proof}
  By \cref{mot--xq--unit-generators}, we may choose $R \to R'$ to be a faithful ind--Zariski localization such that $\pi_0(R')$ is generated by units as a commutative ring. The existence of the remainder of the commutative diagram follows from results of \Cref{WilsonAppendix}: by \Cref{WilsonSurjectivity}, we may choose an $\E_\infty$-ring map $\Sigma^{\infty}_+W \to R'$ from the suspension spectrum of a weak product of Wilson spaces which is surjective on homotopy groups, and by \Cref{WilsonThesisTheorem,WilsonZeroInvertibleGenerator}, $S:=k\otimes \Sigma^{\infty}_+W$ is connective and even with homotopy groups a localization of a polynomial algebra over $\pi_*(k)$. In addition, if $k$ is a cyclotomic base as in \cref{mot--xq--even-polynomial-surjection--cyclotomic}, then there is a compatible cyclotomic base structure on $S$ (as just defined) by \cref{finalAppThm,wi--cyc--wilson-base-addendum}, proving \cref{mot--xq--even-polynomial-surjection--cyclotomic}. Finally, assertion \cref{mot--xq--even-polynomial-surjection--quasilci} follows from \cref{mot--xq--classical-quasilci-permanence}.
\end{proof}

We now extend the above notions to the setting of $\E_\infty$-rings that are not necessarily even.

\begin{definition}
  \label{mot--xq--xquasilci}
  We say that a map of $\E_\infty$-rings $k \to R$ is \emph{chromatically quasi-lci} (resp. \emph{chromatically $p$-quasi-lci}) if both $k \otimes \MU$ and $R \otimes \MU$ are even and the induced map $k \otimes \MU \to R \otimes \MU$ is quasi-lci (resp. $p$-quasi-lci).
\end{definition}

\begin{definition}
  \label{mot--xq--xquasisyntomic}
  We say that an $\E_\infty$-ring $R$ is \emph{chromatically $p$-quasisyntomic} (resp. \emph{chromatically quasisyntomic}) if $R \otimes \MU$ is even and $p$-quasisyntomic (resp. integrally quasisyntomic).
\end{definition}

\begin{example}
  \label{mot--xq--xquasisyntomic-examples}
  If $R$ is any of $\mathbb{S}$, $\mathrm{ko}$, or $\mathrm{tmf}$, then $\MU_*(R)$ is a polynomial $\mathbb{Z}$-algebra on even generators.  Thus, all three of $\mathbb{S}$, $\mathrm{ko}$, and $\mathrm{tmf}$ are chromatically quasisyntomic $\mathbb{E}_{\infty}$-rings.
\end{example}

\begin{remark}
  \label{mot--xq--xquasisyntomic-unit-map}
  Let $R$ be a chromatically ($p$-)quasisyntomic $\E_\infty$-ring. Then the unit map $\S \to R$ is chromatically ($p$-)quasi-lci. This follows from the fact that $\pi_*(\MU)$ is a polynomial ring on even generators, together with \cref{mot--xq--classical-quasilci-permanence}.
\end{remark}

Back in the setting of even $\mathbb{E}_{\infty}$-rings, the chromatic definitions collapse to the preceding ones:

\begin{proposition}
  \label{mot--xq--xq-even}
  A map of even $\mathbb{E}_{\infty}$-rings $k \to R$ is ($p$-)quasi-lci if and only if it is chromatically ($p$-)quasi-lci.
\end{proposition}

\begin{proof}
  By complex orientation theory, there are compatible isomorphisms
  \[
    \MU_*(k) \iso \pi_*(k)[b_1,b_2,\ldots], \quad \MU_*(R) \iso \pi_*(R)[b_1,b_2,\ldots],
  \]
  where $b_i$ has degree $2i$. In particular, $k \otimes \MU$ and $R \otimes \MU$ are also even, and we calculate that 
  \[
    \L^{\alg}_{\MU_{2*}(R)/\MU_{2*}(k)} \iso \Z[b_1,b_2,\cdots] \otimes^{\L}_{\mathbb{Z}} \L^{\alg}_{\pi_{2*}(R)/\pi_{2*}(k)} \iso \MU_{2*}(R) \otimes_{\pi_{2*}(R)} \L^{\alg}_{\pi_{2*}(R)/\pi_{2*}(k)},
  \]
  since the algebraic cotangent complex is compatible with derived base change. It is now easy to see that the Tor amplitude conditions on $\L^{\alg}_{\MU_{2*}(R)/\MU_{2*}(k)}$ and $\L^{\alg}_{\pi_{2*}(R)/\pi_{2*}(k)}$ are equivalent (one direction using that $\MU_{2*}(R)$ is a free $\pi_{2*}(R)$-module).
\end{proof}

\begin{proposition}
  \label{mot--xq--xqsyn-even}
  Let $R$ be an even $\E_\infty$-ring. Then $R$ is $p$-quasisyntomic (resp. integrally quasisyntomic) if and only if it is chromatically $p$-quasisyntomic (resp. chromatically quasisyntomic).
\end{proposition}

\begin{proof}
  Similar to the proof of \cref{mot--xq--xq-even}, this follows from the isomorphism $\MU_*(R) \iso \pi_*(R)[b_1,b_2,\ldots]$, which is a free $\pi_*(R)$-module.
\end{proof}

Finally, the following result will be used in the proof of \cref{mot--fil--compute-tc-as-fiber}.

\begin{proposition}
  \label{mot--xq--strongly-even}
  Let $A$ be a strongly even $\E_\infty$-ring (\cref{wi--strongly-even}). Then the following statements hold.
  \begin{enumerate}
  \item \label{mot--xq--strongly-even--even}
    Let $R$ be an $\E_\infty$-ring such that $R \otimes \MU$ is even. Then $R \otimes A$ is even.
  \item \label{mot--xq--strongly-even--bounded}
    Let $R$ be an $\E_\infty$-ring such that $R \otimes \MU$ has bounded $p$-power torsion. Then $R \otimes A$ has bounded $p$-power torsion.
  \item \label{mot--xq--strongly-even--quasilci}
    Let $k \to R$ be a chromatically ($p$-)quasi-lci map of $\E_\infty$-rings. Then $k \otimes A \to R \otimes A$ is a ($p$-)quasi-lci map of even $\E_\infty$-rings.
  \end{enumerate}
\end{proposition}

\begin{proof}
  Statements \cref{mot--xq--strongly-even--even} and \cref{mot--xq--strongly-even--bounded} are immediate from $A$ splitting as a direct sum of even suspensions of $\MU$. Let us prove \cref{mot--xq--strongly-even--quasilci}. By \cref{mot--xq--strongly-even--even}, $k \otimes A$ and $R \otimes A$ are even. The map $\pi_{2*}(k \otimes A \otimes \MU) \to \pi_{2*}(R \otimes A \otimes \MU)$ thus identifies with the tensor product of the map $\pi_{2*}(k \otimes A) \to  \pi_{2*}(R \otimes A)$ with the polynomial ring $\Z[b_1,b_2,\ldots]$ over $\Z$, and it follows that it is enough to show that the map $\pi_{2*}(k \otimes A \otimes \MU) \to \pi_{2*}(R \otimes A \otimes \MU)$ is ($p$-)quasi-lci. We now conclude by observing that this map identifies with the base change of the map $\pi_{2*}(k \otimes \MU) \to \pi_{2*}(R \otimes \MU)$, which is ($p$-)quasi-lci by hypothesis, along the map $\pi_{2*}(\MU) \to \pi_{2*}(A \otimes \MU)$, which is flat; these observations follow from the formula $(-) \otimes A \otimes \MU \iso ((-) \otimes \MU) \otimes_\MU (A \otimes \MU)$ and the fact that $\pi_{2*}(A \otimes \MU)$ is a free module over $\pi_{2*}(\MU)$, as $A$ splits as a direct sum of suspensions of $\MU$.
\end{proof}


\subsection{Motivic filtrations}
\label{mot--fil}

With the definitions and constructions of \cref{SecEvenFiltration} and \cref{mot--xq} in place, we now define our motivic filtrations on Hochschild homology and its associated invariants.

\begin{definition}
  \label{mot--fil--motivic-fil}
  For a chromatically quasi-lci map of connective $\E_\infty$-rings $k \to R$, we define
  \begin{align*}
    &\fil^\star_\mot\THH(R/k) := \fil^\star_\ev\THH(R/k), \\
    &\fil^\star_\mot\TC^-(R/k) := \fil^\star_{\ev,\h\cir}\THH(R/k), \\
    &\fil^\star_\mot\TP(R/k) := \fil^\star_{\ev,\tate\cir}\THH(R/k).
  \end{align*}
  For a chromatically $p$-quasi-lci map of connective $\E_\infty$-rings $k \to R$ such that $R \otimes \MU$ has bounded $p$-power torsion, we define
  \begin{align*}
    &\fil^\star_\mot\cpl{\THH(R/k)}_p := \fil^\star_{\ev,p}\cpl{\THH(R/k)}_p, \\
    &\fil^\star_\mot\cpl{\TC^-(R/k)}_p := \fil^\star_{\ev,p,\h\cir}\cpl{\THH(R/k)}_p, \\
    &\fil^\star_\mot\cpl{\TP(R/k)}_p := \fil^\star_{\ev,p,\tate\cir}\cpl{\THH(R/k)}_p,
  \end{align*}
  and if $k$ is furthermore a cyclotomic base (\cref{cycbasedef}), we define
  \[
    \fil^\star_\mot\cpl{\TC(R/k)}_p := \fil^\star_{\ev,p,\TC}\cpl{\THH(R/k)}_p.
  \]
\end{definition}

\begin{remark}
  \label{mot--fil--absolute-example}
  \cref{mot--fil--motivic-fil} applies in particular in the case that $k=\S$ and $R$ is a connective, chromatically ($p$-)quasisyntomic $\E_\infty$-ring (see \cref{mot--xq--xquasisyntomic-unit-map}). In this case, we omit the base $k=\S$ from the notation as usual.
\end{remark}

\begin{example} \label{ell-example}
  Let $\ell$ denote the connective Adams summand of $\mathrm{ku}_{(p)}$. We check that 
  \[
    \fil^{\star}_{\mot} \THH(\ell) \iso \lim_{\Delta}\tau_{\ge 2\star}\THH(\ell/\MU^{\otimes \bullet+1}),
  \]
  which is the descent studied by the first and third authors in \cite[\textsection 6.1]{HahnWilson}. Since
  \[
    \THH(\ell/\MU^{\otimes \bullet+1}) \iso \THH(\ell) \otimes_{\THH(\MU)} (\MU)^{\otimes_{\THH(\MU)} \bullet+1},
  \]
  it suffices to check that:
  \begin{enumerate}
  \item $\THH(\ell/\MU)$ is even.
  \item $\THH(\MU) \to \MU$ is evenly free.
  \end{enumerate}
  The first of these points is \cite[Theorem E]{HahnWilson}.  To prove the second point, given any map of $\mathbb{E}_{\infty}$-rings $\THH(\MU) \to A$ where $A$ is even and nonzero, we must show that $A \otimes_{\THH(\MU)} \MU$ is equivalent to a nonzero direct sum of even suspensions of $A$.

  Recall that $\THH_*(\MU) \iso \Lambda_{\MU_*}(\sigma b_1,\sigma b_2,\cdots)$ is an exterior algebra over $\MU_*$ generated by classes $\sigma b_i$ in odd degrees \cite{BCS,RognesMU}.  Because $A$ is even, each of the classes $\sigma b_i$ must map to zero along the map $\THH_*(\MU) \to \pi_*(A)$.  The K\"unneth spectral sequence for $\pi_*\left(A \otimes_{\THH(\MU)} \MU\right)$ therefore has $\mathrm{E}_2$-term given by
  $\pi_*(A) \otimes_{\MU_*} \Gamma_{\MU_*}(\sigma^2b_i),$
  where $\Gamma_{\MU_*}(\sigma^2b_i)$ denotes a divided power algebra on even degree generators.  The spectral sequence collapses for degree reasons, and the result is a nonzero, free $\pi_*(A)$-module.
\end{example}


In the remainder of the section we set up some basic theory, culminating in proofs of \Cref{motConv} and \Cref{filteredFrob} from the introduction.  
Roughly speaking, these theorems state that the motivic filtrations converge and that the motivic filtration on $\cpl{\TC}_p$ is compatible with the motivic filtrations on $\TC^{-}$ and $\TP$.

\begin{proposition}
  \label{mot--fil--quasismooth-THH}
  Let $k \to S$ be a quasismooth (resp. $p$-quasismooth) map of connective, even $\E_\infty$-rings. Then the map $\THH(S/k) \to S$ is eff (resp. $p$-completely eff). If $\pi_*(S)$ is in fact a polynomial $\pi_*(k)$-algebra, or more generally if $\Omega^1_{\pi_*(S)/\pi_*(k)}$ is free over $\pi_*(S)$, then the map $\THH(S/k) \to S$ is in fact evenly free.
\end{proposition}

\begin{proof}
  We will prove the integral statements; the $p$-complete statement can be proved similarly. We begin by calculating the homotopy groups of $\THH(S/k)$. Consider the filtered $\E_\infty$-ring $\THH(\tau_{\ge 2\star}(S)/\tau_{\ge 2\star}(k))$,
  which is complete, has underlying $\E_\infty$-ring $\THH(S/k)$, and has associated graded $\E_\infty$-ring $\THH(\Sigma^{2*}\pi_{2*}(S)/\Sigma^{2*}\pi_{2*}(k)) \iso \Sigma^{2*}\THH(\pi_{2*}(S)/\pi_{2*}(k))$. This determines a convergent multiplicative spectral sequence
  \begin{equation}
    \label{mot--fil--quasismooth-THH--postnikov-sseq}
    E_1^{s,t} \iso \pi_t(\THH(\pi_{2*}(S)/\pi_{2*}(k))_{\frac{s-t}{2}}) \Rightarrow \pi_s\THH(S/k)
  \end{equation}
  (with the convention that $E_1^{s,t} = 0$ when $s-t$ is odd) with differentials of signature $d_r : E_r^{s,t} \to E_r^{s-1,t-2r-1}$.

  Let us first analyze $\THH(\pi_{2*}(S)/\pi_{2*}(k))$. On this we have the HKR filtration, which is complete and has associated graded $\LSym^\filledsquare(\Sigma\smash{\L^\alg_{\pi_{2*}(S)/\pi_{2*}(k)}})$, where $\LSym^\filledsquare$ denotes derived symmetric powers (over $\pi_{2*}(S)$). A theorem of Illusie (see \cite[Proposition 25.2.4.2]{sag}) gives
  \begin{equation}
    \label{mot--fil--quasismooth-THH--lsym}
    \L\mathrm{Sym}^n (\Sigma\L^{\alg}_{\pi_{2*}(S)/\pi_{2*}(k)})
    \iso
    \Sigma^n\L\Lambda^n(\L^{\alg}_{\pi_{2*}(S)/\pi_{2*}(k)}),
  \end{equation}
  where $\L\Lambda^\filledsquare$ denotes derived exterior powers. As $k \to S$ is quasismooth,
  \[
    \L^\alg_{\pi_{2*}(S)/\pi_{2*}(k)} \iso \Omega^1_{\pi_{2*}(S)/\pi_{2*}(k)}
  \]
  and this is a flat $\pi_{2*}(S)$-module, so its $n$-th derived exterior powers agree with the classical $n$-exterior power $\Omega^n_{\pi_{2*}(S)/\pi_{2*}(k)}$ (\cite[Proposition 25.2.3.4]{sag}). Thus, the object \cref{mot--fil--quasismooth-THH--lsym} is concentrated in homological degree $n$, from which it follows that the HKR spectral sequence collapses to give a ring isomorphism
  \[
    \pi_\filledsquare \THH(\pi_{2*}(S)/\pi_{2*}(k)) \iso \Omega^\filledsquare_{\pi_{2*}(S)/\pi_{2*}(k)}.
  \]

  We now analyze the spectral sequence \cref{mot--fil--quasismooth-THH--postnikov-sseq}; our argument will be an adaptation of \cite[Proof of Proposition 3.6]{angeltveit}. For degree reasons, the classes in $E_1^{*,0}$ and $E_1^{*,1}$ cannot support differentials in the spectral sequence, and as they are multiplicative generators by the calculation above, this implies that the spectral sequence must collapse. Furthermore, there are no multiplicative extensions; this boils down to the following assertion: if $x \in \pi_s\THH(S/k)$ is a class detected by $\overline{x} \in \Omega^1_{\pi_{2*}(S)/\pi_{2*}(k)}$, then $x^2 = 0$. We see this as follows.

  Let us say that a class $x \in \pi_s \THH(S/k)$ has \emph{filtration $\le t$} if it is in the image of the map from the $(s-t)/2$-th stage of the (decreasing) filtration $\THH(\tau_{\ge2\star}(S)/\tau_{\ge2\star}(k))$. Then, if $x$ is detected by $\overline{x} \in \Omega^1_{\pi_{2*}(S)/\pi_{2*}(k)}$, it has filtration $\le 1$, and so $x^2$ has filtration $\le 2$. But since $\overline{x}^2 = 0$, in fact $x^2$ must have filtration $\le 1$, and since $x^2$ has even degree, it must then have filtration $\le 0$. Finally, the canonical map $\THH(S/k) \to S$ induces an isomorphism from the filtration $\le 0$ piece of $\pi_*\THH(S/k)$ to $\pi_*(S)$, so it suffices to show that this map sends $x^2$ to $0$. This follows from the fact that the map sends $x$ to $0$, $x$ being of odd degree and $\pi_*(S)$ being concentrated in even degrees.

  The conclusion of the above discussion is an identification
  \begin{equation}
    \label{mot--fil--quasismooth-THH--homotopy}
    \THH_*(S/k) \iso \Lambda_{\pi_*(S)}(\Omega^1_{\pi_{*}(S)/\pi_{*}(k)}[1]).
  \end{equation}
  (where here $[1]$ denotes a shift in grading). With this calculation in hand, we may prove the claim. Consider a pushout square of the form
  \[
    \begin{tikzcd}
      \THH(S/k) \arrow{d} \arrow{r} & S \arrow{d} \\
      C \arrow{r} & C \otimes_{\THH(S/k)} S,
    \end{tikzcd}
  \]
  where $C$ is an even $\mathbb{E}_{\infty}$-ring. Under the identification \cref{mot--fil--quasismooth-THH--homotopy}, the map $\THH_*(S/k) \to \pi_*(C)$ must send each exterior generator $x \in \Omega^1_{\pi_*(S)/\pi_*(k)}$ to $0$, because these classes have odd degree. It follows from this and $\Omega^1_{\pi_*(S)/\pi_*(k)}$ being a flat $\pi_*(S)$-module that the Tor spectral sequence calculating $\pi_*(C \otimes_{\THH(S/k)} S)$ has $\mathrm{E}_2$-page given by
  \[
    \Tor^{\THH_*(S/k)}_\filledsquare(\pi_*(C),\pi_*(S)) \iso \pi_*(C) \otimes_{\pi_*(S)} \Gamma^\filledsquare_{\pi_*(S)}(\Omega^1_{\pi_*(S)/\pi_*(k)}[1]),
  \]
  where $\Gamma^\filledsquare$ denotes divided powers. We see that this spectral sequence must degenerate for degree reasons, and the result is a faithfully flat $\pi_*(C)$-module in even degrees. Moreover, if $\pi_*(S)$ is polynomial over $\pi_*(k)$, or more generally if $\Omega^1_{\pi_*(S)/\pi_*(k)}$ is free over $\pi_*(S)$, the result is in fact a free $\pi_*(C)$-module.
\end{proof}

\begin{proposition}
  \label{mot--fil--quasiregular-quotient-THH}
  \begin{enumerate}[leftmargin=*]
  \item Let $S \to R$ be a quasiregular quotient of connective, even $\E_\infty$-rings. Then $\THH(R/S)$ is even.
  \item Let $S \to R$ be a $p$-quasiregular quotient of connective, even $\E_\infty$-rings and suppose that $R$ has bounded $p$-power torsion. Then $\cpl{\THH(R/S)}_p$ is even and has bounded $p$-power torsion.
  \end{enumerate}
\end{proposition}

\begin{proof}
  We prove the second statement, following the same general strategy as in the proof of \cref{mot--fil--quasismooth-THH}; the first statement can be proved similarly. We have the filtered spectrum $\cpl{\THH(\tau_{\ge 2\star}(R)/\tau_{\ge 2\star}(S))}_p$, which is complete, has underlying spectrum $\cpl{\THH(R/S)}_p$, and has associated graded spectrum $\cpl{\THH(\Sigma^{2*}\pi_{2*}(R)/\Sigma^{2*}\pi_{2*}(S))}_p \iso \Sigma^{2*}\cpl{\THH(\pi_{2*}(R)/\pi_{2*}(S))}_p$. This determines a convergent spectral sequence
  \[
    \pi_\filledsquare(\cpl{\THH(\pi_{2*}(R)/\pi_{2*}(S))}_p) \Rightarrow \pi_{2*+\filledsquare}(\cpl{\THH(R/S)}_p).
  \]
  We will show in the next paragraph that $\pi_\filledsquare(\cpl{\THH(\pi_{2*}(R)/\pi_{2*}(S))}_p)$ is concentrated in even $\filledsquare$-degrees, and that in each even $\filledsquare$-degree it is $p$-completely flat over $\pi_{2*}(R)$. The evenness property implies degeneration of the spectral sequence, which gives that $\cpl{\THH(R/S)}_p$ is even, and moreover that $\pi_{2*}(\cpl{\THH(R/S)}_p)$ admits an exhaustive increasing filtration by $\pi_{2*}(R)$-modules whose associated graded is $\pi_\filledsquare(\cpl{\THH(\pi_{2*}(R)/\pi_{2*}(S))}_p)$. Then the $p$-complete flatness property gives that $\pi_{2*}(\cpl{\THH(R/S)}_p)$ is $p$-completely flat over $\pi_{2*}(R)$, as $p$-complete flatness is closed under extensions and filtered colimits. We conclude that $\cpl{\THH(R/S)}_p$ has bounded $p$-power torsion, using that $R$ has bounded $p$-power torsion and \cref{ev--desc--pcpl-flat-bdd-torsion}.

  We now justify the above assertions about $\cpl{\THH(\pi_{2*}(R)/\pi_{2*}(S))}_p$. This carries the $p$-completed HKR filtration, which is complete and has associated graded
  \[
    \cpl{\L\mathrm{Sym}^\filledsquare(\Sigma \smash{\L^{\alg}_{\pi_{2*}(R)/\pi_{2*}(S)}})}_p,
  \]
  where $\L\mathrm{Sym}^\filledsquare$ denotes derived symmetric powers (over $\pi_{2*}(R)$). We will complete the proof by showing that each graded piece of $\cpl{\L\mathrm{Sym}^n(\Sigma \smash{\L^{\alg}_{\pi_{2*}(R)/\pi_{2*}(S)}})}_p$ is homologically concentrated in degree $2n$, and there is a $p$-completely flat module over $\pi_{2*}(R)$. By hypothesis, $\pi_{2*}(S) \to \pi_{2*}(R)$ is surjective, so $\smash{\L^{\alg}_{\pi_{2*}(R)/\pi_{2*}(S)}}$ is $1$-connective. By a theorem of Illusie (see \cite[Proposition 25.2.4.2]{sag}), 
  \[
    \L\mathrm{Sym}^n (\Sigma\L^{\alg}_{\pi_{2*}(R)/\pi_{2*}(S)})
    \iso
    \Sigma^{2n}\L\Gamma^n(\Sigma^{-1}\L^{\alg}_{\pi_{2*}(R)/\pi_{2*}(S)}),
  \]
  where $\L\Gamma^\filledsquare$ denotes derived divided powers. By \cref{mot--xq--classical-quasiregular-equivalence}, since $S \to R$ is a $p$-quasiregular quotient, $\Sigma^{-1}\L^{\alg}_{\pi_{2*}(R)/\pi_{2*}(S)}$ is $p$-completely flat over $\pi_{2*}(R)$. The formation of derived divided powers commutes with base change and derived divided powers of a flat module are flat (\cite[Proposition 25.2.3.4]{sag}), so we deduce that $\L\Gamma^n(\Sigma^{-1}\smash{\L^{\alg}_{\pi_{2*}(R)/\pi_{2*}(S)}})$ is also $p$-completely flat, which also implies that it is discrete after $p$-completion (again using that $\pi_{2*}(R)$ has bounded $p$-power torsion and \cref{ev--desc--pcpl-flat-bdd-torsion}).
\end{proof}

\begin{corollary}
  \label{mot--fil--quasilci-THH}
  Let $k \to S \to R$ be maps of connective, even $\E_\infty$-rings such that $k \to S$ is quasismooth (resp. $p$-quasismooth). Then:
  \begin{enumerate}
  \item \label{mot--fil--quasilci-THH--eff}
    The map $\THH(R/k) \to \THH(R/S)$ is eff (resp. $p$-completely eff). If $\pi_*(S)$ is in fact a polynomial $\pi_*(k)$-algebra, or more generally if $\Omega^1_{\pi_*(S)/\pi_*(k)}$ is free over $\pi_*(S)$, then the map $\THH(R/k) \to \THH(R/S)$ is in fact evenly free.
  \item \label{mot--fil--quasilci-THH--even}
    If the map $S \to R$ is surjective on homotopy groups and the composite $k \to R$ is quasi-lci (resp. $p$-quasi-lci, with $R$ having bounded $p$-power torsion), then $\THH(R/S)$ is even (resp. $\cpl{\THH(R/S)}_p$ is even and has bounded $p$-power torsion).
  \end{enumerate}
\end{corollary}

\begin{proof}
  The first statement follows from \cref{mot--fil--quasismooth-THH}, as the map $\THH(R/k) \to \THH(R/S)$ is obtained by pushing out the map $\THH(S/k) \to S$ along the map $\THH(S/k) \to \THH(R/k)$. The second statement follows from \cref{mot--fil--quasiregular-quotient-THH}, as $S \to R$ being surjective on homotopy groups and $k \to R$ being ($p$-)quasi-lci together imply that $S \to R$ is a ($p$-)quasiregular quotient (\cref{mot--xq--even-polynomial-surjection}\cref{mot--xq--even-polynomial-surjection--quasilci}).
\end{proof}

This last result, together with \cref{mot--xq--even-polynomial-surjection}, gives us a good supply of eff covers by even objects in the context of our motivic filtrations. Let us illustrate this by deducing the following good behavior of these filtrations, with \cref{mot--fil--exhaustivity,mot--fil--compute-tc-as-fiber} below generalizing \cref{motConv,filteredFrob} from the introduction.

\begin{proposition}
  \label{mot--fil--p-compatibility}
  Let $k \to R$ be a chromatically quasi-lci map of connective $\E_\infty$-rings such that $R \otimes \MU$ has bounded $p$-power torsion. Then
  \[
    \cpl{(\fil^\star_\mot\THH(R/k))}_p \iso \fil^\star_\mot\cpl{\THH(R/k)}_p.
  \]
\end{proposition}

\begin{proof}
  Choose a commutative square
  \[
    \begin{tikzcd}
      k \otimes \MU \ar[r] \ar[d] &
      S \ar[d] \\
      R \otimes \MU \ar[r] &
      {R'}
    \end{tikzcd}
  \]
  as in \cref{mot--xq--even-polynomial-surjection}, and consider the composition of maps of $\cir$-equivariant $\E_\infty$-rings
  \begin{align*}
    \THH(R/k) &\to \THH(R/k) \otimes \MU \iso \THH(R \otimes \MU/k \otimes \MU) \\
              &\to \THH(R'/k \otimes \MU) \\
              &\to \THH(R'/S).
  \end{align*}
  The first map is evenly free by \cref{ev--desc--novikov-eff}. The canonical map
  \[
    R' \otimes_{R \otimes \MU} \THH(R \otimes \MU/k \otimes \MU) \to \THH(R'/k \otimes \MU)
  \]
  is an equivalence because $R \otimes \MU \to R'$ is an ind--Zariski localization (see \cite[Lemma 5.7]{mccarthy-minasian--HKR} or \cite[Theorem 1.3]{mathew--THH-base-change}), and since it is also assumed to be faithful, we deduce that the second map in the composition above is faithfully flat in the sense of \cite[Definition D.4.4.1]{sag}, hence eff. The third map is evenly free by \cref{mot--fil--quasilci-THH}\cref{mot--fil--quasilci-THH--eff}. Thus, the composite map is eff. The map $k \otimes \MU \to R'$ is quasi-lci (since $k \to R$ is chromatically quasi-lci and $R'$ is an ind--Zariski localization of $R \otimes \MU$), so by \cref{mot--fil--quasilci-THH}\cref{mot--fil--quasilci-THH--even}, $\THH(R'/S)$ is even and $\cpl{\THH(R'/S)}_p$ is even and has bounded $p$-power torsion. The claim now follows from \cref{ev--arith--p-completion}.
\end{proof}

\begin{theorem}
  \label{mot--fil--exhaustivity}
  \begin{enumerate}[leftmargin=*]
  \item \label{mot--fil--exhaustivity--integral}
    Let $k \to R$ be a chromatically quasi-lci map of connective $\E_{\infty}$-rings. Then the canonical maps
    \begin{align*}
      &\THH(R/k) \to \colim(\fil^\star_\mot \THH(R/k)),\\
      &\TC^-(R/k) \to \colim(\fil^\star_\mot \TC^-(R/k)),\\
      &\TP(R/k) \to \colim(\fil^\star_\mot \TP(R/k))
    \end{align*}
    are equivalences.
  \item \label{mot--fil--exhaustivity--p}
    Let $k \to R$ be a chromatically $p$-quasi-lci map of connective $\E_{\infty}$-rings such that $R \otimes \MU$ has bounded $p$-power torsion. Then the canonical maps
    \begin{align*}
      &\cpl{\THH(R/k)}_p \to \colim(\fil^\star_\mot \cpl{\THH(R/k)}_p),\\
      &\cpl{\TC^-(R/k)}_p \to \colim(\fil^\star_\mot \cpl{\TC^-(R/k)}_p),\\
      &\cpl{\TP(R/k)}_p \to \colim(\fil^\star_\mot \cpl{\TP(R/k)}_p)
    \end{align*}
    are equivalences.
  \end{enumerate}
\end{theorem}

\begin{proof}
 For either statement, choose a commutative square
  \[
    \begin{tikzcd}
      k \otimes \MU \ar[r] \ar[d] &
      S \ar[d] \\
      R \otimes \MU \ar[r] &
      {R'}
    \end{tikzcd}
  \]
  as in \cref{mot--xq--even-polynomial-surjection}, and consider the composition of maps of $\cir$-equivariant $\E_\infty$-rings
  \begin{align*}
    \THH(R/k) &\to \THH(R/k) \otimes \MU \iso \THH(R \otimes \MU/k \otimes \MU) \\
              &\to \THH(R'/k \otimes \MU) \\
              &\to \THH(R'/S).
  \end{align*}
  Arguing as in the proof of \cref{mot--fil--p-compatibility}, we find that the composite is eff, that $\THH(R'/S)$ is even in the case of \cref{mot--fil--exhaustivity--integral} and $\cpl{\THH(R'/S)}_p$ is even and has bounded $p$-power torsion in the case of in the case of \cref{mot--fil--exhaustivity--p}, and that the second map in the composition is  is faithfully flat in the sense of \cite[Definition D.4.4.1]{sag}. Moreover, the first and third maps in the composite are $1$-connective, so by \cref{ev--exh--universal}, we deduce that the composite map satisfies descent, $p$-complete descent, Tate descent, and $p$-complete Tate descent. The claim now follows from \cref{ev--exh--main}.
\end{proof}

\begin{theorem}
  \label{mot--fil--compute-tc-as-fiber}
  For $k$ a connective cyclotomic base and $k \to R$ a chromatically $p$-quasi-lci map of connective $\E_{\infty}$-rings such that $R \otimes \MU$ has bounded $p$-power torsion, there are natural maps
  \[
    \varphi,\can : \fil^{\star}_\mot \cpl{\TC^{-}(R/k)}_p \to \fil^{\star}_\mot \cpl{\TP(R/k)}_p,
  \]
  converging to $\varphi,\can : \cpl{\TC^{-}(R/k)}_p \to \cpl{\TP(R/k)}_p$, respectively, together with a natural equivalence
  \[
    \fil^{\star}_{\mot} \cpl{\TC(R/k)}_p \iso \mathrm{fib}\left(\varphi-\mathrm{can} : \fil^{\star}_\mot \cpl{\TC^{-}(R/k)}_p \to \fil^{\star}_\mot \cpl{\TP(R/k)}_p \right).
  \]
\end{theorem}

\begin{proof}
  Let $A$ be a strongly even $\E_\infty$-ring (\cref{wi--strongly-even}). By \cref{wi--cyc--strongly-even-base}, we may choose a cyclotomic base structure on $A$. By \cref{mot--xq--strongly-even}, the induced map $k \otimes A \to R \otimes A$ is a $p$-quasi-lci map of even $\E_\infty$-rings and $R \otimes A$ has bounded $p$-power torsion. Choose a commutative square
  \[
    \begin{tikzcd}
      k \otimes A \ar[r] \ar[d] &
      S \ar[d] \\
      R \otimes A \ar[r] &
      {R'},
    \end{tikzcd}
  \]
  with $k \otimes A \to S$ a map of cyclotomic bases, as in \cref{mot--xq--even-polynomial-surjection}. Now, consider the composition
  \begin{align*}
    \THH(R/k) &\to \THH(R/k) \otimes A \iso \THH(R \otimes A/k \otimes A) \\
              &\to \THH(R'/k \otimes A) \\
              &\to \THH(R'/S)
  \end{align*}
  (each map being a map of $p$-typical cyclotomic $\E_\infty$-rings upon $p$-completion). Arguing as in the proof of \cref{mot--fil--p-compatibility}, we find that the composite is eff and that $\cpl{\THH(R'/S)}_p$ is even and has bounded $p$-power torsion. The claim now follows from \cref{ev--desc--compute-tc-as-fiber}.
\end{proof}


\begin{remark}
  It follows from the proof of the previous theorem that for $R$ a connective, chromatically quasisyntomic $\E_\infty$-ring, $\THH(R)$ admits an eff map to an even cyclotomic $\E_\infty$-ring. As indicated in \cref{rmk-geometric-stack}, the convenience of the motivic filtration should hold under this weaker requirement alone. It would be very interesting to have a checkable characterization of those $R$ such that $\THH(R)$ admits an eff map to an even cyclotomic ring. 
\end{remark}

%% file: ComparisonTheorems.tex
\section{Comparison theorems}

In this section, we will compare the motivic filtrations defined above with filtrations defined previously in special cases.

\begin{notation}
  \label{cmp--notation}
  \begin{enumerate}[leftmargin=*]
  \item For $k \to R$ any map of commutative rings, we have the Hochschild--Kostant--Rosenberg (HKR) filtration on Hochschild homology, $\fil^\star_\HKR\HH(R/k)$ (see, for example, \cite[\textsection 2]{BMS}). We denote the $p$-completion of the HKR filtration by $\fil^\star_\HKR\cpl{\HH(R/k)}_p$.
  \item For $k \to R$ a quasi-lci map of commutative rings, we have Antieau's Beilinson filtrations on negative and periodic cyclic homology, $\fil^\star_\Beil\HC^-(R/k)$ and $\fil^\star_\Beil\HP(R/k)$ \cite{AntieauHP}.
  \item For $k \to R$ a $p$-quasi-lci map of $p$-quasisyntomic, $p$-complete commutative rings, we have the Bhatt--Morrow--Scholze (BMS) filtrations on $p$-completed negative and periodic cyclic homology, $\fil^\star_\BMS\cpl{\HC^-(R/k)}_p$ and $\fil^\star_\BMS\cpl{\HP(R/k)}_p$ \cite[\textsection 5]{BMS}.
  \item For $R$ a $p$-quasisyntomic, $p$-complete commutative ring, we have the BMS filtrations on $p$-completed topological Hochschild, topological negative cyclic, topological periodic cyclic, and topological cyclic homology \cite[\textsection 7]{BMS},
    \[
      \fil^\star_\BMS\cpl{\THH(R)}_p, \quad
      \fil^\star_\BMS\cpl{\TC^-(R)}_p, \quad
      \fil^\star_\BMS\cpl{\TP(R/k)}_p, \quad
      \fil^\star_\BMS\cpl{\TC(R)}_p.
    \]
  \item For $R$ any commutative ring, we have the Morin/Bhatt--Lurie filtration on topological Hochschild, topological negative cyclic, and topological periodic cyclic homology (\cite{Morin} and \cite[\textsection 6.4]{BL}),
    \[
      \fil^\star_\MBL\THH(R), \quad
      \fil^\star_\MBL\TC^-(R), \quad
      \fil^\star_\MBL\TP(R).
    \]
  \end{enumerate}
\end{notation}

\begin{theorem}
  \label{cmp--hkr}
  For $k \to R$ a quasi-lci map of discrete commutative rings, there are natural identifications
  \begin{align*}
    &\fil^\star_\mot\HH(R/k) \iso \fil^\star_\HKR\HH(R/k), \\
    &\fil^\star_\mot\HC^-(R/k) \iso \fil^\star_\Beil\HC^-(R/k), \\
    &\fil^\star_\mot\HP(R/k) \iso \fil^\star_\Beil\HP(R/k).
  \end{align*}
  For $k \to R$ a $p$-quasi-lci map of $p$-quasisyntomic, $p$-complete commutative rings, there are natural identifications
  \begin{align*}
    &\fil^\star_\mot\cpl{\HH(R/k)}_p \iso \fil^\star_\HKR\cpl{\HH(R/k)}_p, \\
    &\fil^\star_\mot\cpl{\HC^-(R/k)}_p \iso \fil^\star_\BMS\cpl{\HC^-(R/k)}_p, \\
    &\fil^\star_\mot\cpl{\HP(R/k)}_p \iso \fil^\star_\BMS\cpl{\HP(R/k)}_p.
  \end{align*}
\end{theorem}

\begin{proof}
  We will establish the first identification; the rest can be established similarly. Let $k \to R$ be a quasi-lci map of commutative rings. Let $S$ be the polynomial $k$-algebra with generators indexed by the elements of $R$, which comes equipped with a canonical surjection $S \to R$. Then we have
  \[
    \fil^\star_\mot\HH(R/k) \iso \lim_{\Delta} \tau_{\ge 2\star}(\HH(R/S^{\otimes_k \bullet+1})) \iso \lim_{\Delta} \fil^\star_\HKR \HH(R/S^{\otimes_k \bullet+1}),
  \]
  where the first equivalence follows from \cref{mot--fil--quasilci-THH} and the second equivalence follows from the identifications
  \[
    \gr^i_\HKR \HH(R/S^{\otimes_k \bullet+1}) \iso \Sigma^i\L\Lambda^i_R(\L^\alg_{R/S^{\otimes_k \bullet+1}})
  \]
  and the fact that $\L^\alg_{R/S^{\otimes_k n+1}}$ has Tor-amplitude concentrated in degree $1$ (cf. the proof of \cref{mot--fil--quasiregular-quotient-THH}). It thus suffices to show that the canonical map
  \[
    \fil^\star_\HKR\HH(R/k) \to \lim_{\Delta} \fil^\star_\HKR \HH(R/S^{\otimes_k \bullet+1})
  \]
  is an equivalence. We can check this after passing to associated graded objects, so it is enough to show that the canonical maps
  \[
    \L\Lambda^i_R(\L^\alg_{R/k}) \to \lim_{\Delta} \L\Lambda^i_R(\L^\alg_{R/S^{\otimes_k \bullet+1}})
  \]
  are equivalences.

  Consider the commutative diagram of cosimplicial $R$-modules
  \[
    \begin{tikzcd}
      R \otimes_S \L^\alg_{S/k} \ar[r] \ar[d] &
      \L^\alg_{R/k} \ar[r] \ar[d] &
      \L^\alg_{R/S} \ar[d] \\
      R \otimes_S \L^\alg_{S/S^{\otimes_k\bullet+1}} \ar[r] & \L^\alg_{R/S^{\otimes_k\bullet+1}} \ar[r] &
      \L^\alg_{R/S}
    \end{tikzcd}
  \]
  in which the top row consists of constant cosimplicial objects, each row is a transitivity cofiber sequence, and the map between the two rows is induced by the map of cosimplicial commutative rings $k \to S^{\otimes_k\bullet+1}$. In \cite[Proof of Corollary 2.7]{BhattDR}, it is shown that the map $\L^\alg_{S/k} \to \L^\alg_{S/S^{\otimes_k\bullet+1}}$ is a homotopy equivalence of cosimplicial $S$-modules. Thus, the left-hand vertical map is a homotopy equivalence of cosimplicial $R$-modules, which implies that the same is true of the middle vertical map, since the right-hand vertical map is the identity map. It follows that the map $\L\Lambda^i_R(\L^\alg_{R/k}) \to \L\Lambda^i_R(\L^\alg_{R/S^{\otimes_k \bullet+1}})$ is a homotopy equivalence of cosimplicial $R$-modules for all $i$, implying the desired claim.
\end{proof}

\begin{theorem}
  \label{cmp--bms}
  For $R$ a $p$-quasisyntomic, $p$-complete commutative ring, there are natural identifications
  \begin{align*}
    &\fil^\star_\mot\cpl{\THH(R)}_p \iso \fil^\star_\BMS\cpl{\THH(R)}_p, \\
    &\fil^\star_\mot\cpl{\TC^-(R)}_p \iso \fil^\star_\BMS\cpl{\TC^-(R)}_p, \\
    &\fil^\star_\mot\cpl{\TP(R)}_p \iso \fil^\star_\BMS\cpl{\TP(R)}_p, \\
    &\fil^\star_\mot\cpl{\TC(R)}_p \iso \fil^\star_\BMS\cpl{\TC(R)}_p.
  \end{align*}
\end{theorem}

\begin{proof}
  \footnote{We thank Bhargav Bhatt for suggesting the argument written here (though any error is our own responsibility). Our original argument used \cref{cmp--hkr} to establish quasisyntomic descent for our motivic filtrations.} Let $R'$ be the polynomial ring over $\Z$ on generators indexed by the set underlying $R$, so that we have a natural surjection $R' \surj R$. Let $S'$ be the ring obtained by adjoining all $p$-power roots of the polynomial generators of $R'$, and form the $p$-completed pushout $S := \cpl{(R \otimes_{R'} S')}_p$. Then $S$ is quasiregular semiperfectoid and the map $R \to S$ is a cover in the quasisyntomic topology of \cite[\textsection 4]{BMS}, so from \cite[\textsection 7]{BMS} we have that $\cpl{\THH(S)}_p$ is even and a natural equivalence
  \[
    \fil^\star_\BMS \cpl{\THH(R)}_p \iso \lim_\Delta(\tau_{\ge 2\star}(\cpl{\THH(S^{\otimes_R\bullet+1})}_p)),
  \]
  and similarly for $\TC^-$, $\TP$, and $\TC$. It now follows from \cref{ev--desc--eff-desc} that to prove the claim, it suffices to show that the map $\cpl{\THH(R)}_p \to \cpl{\THH(S)}_p$ is $p$-completely eff.

  Let $\S_{R'}$ be the polynomial $\E_\infty$-ring over $\S$ on generators indexed by the set underlying $R$ (i.e. the tensor product over this set of copies of the monoid ring $\S[\mathbb{N}]$) and let $\S_{S'}$ be the $\E_\infty$-ring obtained by adjoining all $p$-power roots of the polynomial generators of $\S_{R'}$ (i.e. a tensor product of copies of the monoid ring $\S[\mathbb{N}[1/p]]$). Consider the commutative diagram
  \[
    \begin{tikzcd}
      \THH(\S_{R'}) \ar[r] \ar[d] &
      \S_{R'} \ar[r] \ar[d] &
      \S_{S'} \ar[d] \\
      \THH(R) \ar[r] &
      \THH(R/\S_{R'}) \ar[r] &
      \THH(S/\S_{S'}).
    \end{tikzcd}
  \]
  By \cite[Proposition 11.7]{BMS}, the map $\THH(S) \to \THH(S/\S_{S'})$ is an equivalence after $p$-completion. It is thus enough to show that each of the bottom horizonal maps is $p$-completely eff. Since each square in the diagram is in fact a pushout square, it furthermore suffices to show that each of the top horizontal maps is $p$-completely eff. In fact, they are evenly free: this is clear for the right-hand map, as $\S_{S'}$ is free as a module over $\S_{R'}$; for the left-hand map, we may check after tensoring with $\MU$ (since $\S \to \MU$ is evenly free), and then we may use \cref{mot--fil--quasismooth-THH}.
\end{proof}

\begin{theorem}
  \label{mo--bl-comparison}
  For $R$ an integrally quasisyntomic commutative ring (i.e. having bounded $p$-power torsion for all primes $p$ and with algebraic cotangent complex $\L^{\mathrm{alg}}_R$ having Tor amplitude contained in $[0,1]$), there are natural identifications
  \begin{align*}
    &\fil^\star_\mot\THH(R) \iso \fil^\star_\MBL\THH(R), \\
    &\fil^\star_\mot\TC^-(R) \iso \fil^\star_\MBL\TC^-(R), \\
    &\fil^\star_\mot\TP(R) \iso \fil^\star_\MBL\TP(R).
  \end{align*}
\end{theorem}

\begin{proof}
  Let us just establish the identification for $\THH$. Let $R$ be an integrally quasisyntomic commutative ring. In the following argument, we will use that $R$ is chromatically quasisyntomic (\cref{mot--xq--xqsyn-even}), so the unit map $\S \to R$ is chromatically quasi-lci (\cref{mot--xq--xquasisyntomic-unit-map}), and that the unit map $\Z \to R$ is also chromatically quasi-lci (\cref{mot--xq--xq-even}).

  We have a commutative square
  \[
    \begin{tikzcd}
      \fil^\star_\mot \THH(R) \ar[r] \ar[d] &
      \prod_p \cpl{(\fil^\star_\mot \THH(R))}_p \ar[d] \\
      \fil^\star_\mot \HH(R) \ar[r] &
      \prod_p \cpl{(\fil^\star_\mot \HH(R))}_p,
    \end{tikzcd}
  \]
  and we have a defining pullback square
  \[
    \begin{tikzcd}
      \fil^\star_\MBL \THH(R) \ar[r] \ar[d] &
      \prod_p \fil^\star_\BMS \cpl{\THH(R)}_p \ar[d] \\
      \fil^\star_\HKR \HH(R) \ar[r] &
      \prod_p \fil^\star_\HKR \cpl{\HH(R)}_p
    \end{tikzcd}
  \]
  From \cref{cmp--hkr,mot--fil--p-compatibility}, we obtain an identification between the lower arrows of the two squares. Noting that the $p$-completion of $R$ is a $p$-quasisyntomic discrete commutative ring and that $\cpl{\THH(R)}_p \iso \cpl{\THH(\cpl{R}_p)}_p$, \cref{cmp--bms,mot--fil--p-compatibility} give an identification between the upper right objects of the two squares. The natural transformation $\fil^\star_\BMS \cpl{\THH(-)}_p \to \fil^\star_\HKR \cpl{\HH(-)}_p$ on $p$-quasisyntomic, $p$-complete commutative rings is the unique natural map of filtered objects compatible with the canonical map $\cpl{\THH(-)}_p \to \cpl{\HH(-)}_p$, since the filtrations are descended from the double-speed Postnikov filtration for quasiregular semiperfectoid rings. It follows that the right-hand arrows of the squares identify as well. Thus, to finish the proof, it suffices to show that the first square is a pullback diagram.

  Consider the extended diagram
  \[
    \begin{tikzcd}
      \fil^\star_\mot \THH(R) \ar[r] \ar[d] &
      \prod_p \cpl{(\fil^\star_\mot \THH(R))}_p \ar[d] \\
      \fil^\star_\mot \HH(R) \ar[r] \ar[d] &
      \prod_p \cpl{(\fil^\star_\mot \HH(R))}_p \ar[d] \\
      (\fil^\star_\mot \HH(R)) \otimes \Q \ar[r] &
      (\prod_p \cpl{(\fil^\star_\mot \HH(R))}_p) \otimes \Q
    \end{tikzcd}
  \]
  The lower square is an arithmetic square, hence a pullback square, and the upper square consists of complete filtered objects, so it suffices to show that the outer square is a pullback square after filtration completion. In fact, the outer square is also an arithmetic square after filtration completion; this is a consequence of \cref{ev--arith--rationalization,ev--arith--profinite-rationalization}, using the following observations:
  \begin{itemize}
  \item The canonical maps $\THH(R) \otimes \Q \to \HH(R) \otimes \Q$ and $\cpl{\THH(R)} \otimes \Q$ and $\cpl{\HH(R)} \otimes \Q$ are equivalences (where $\cpl{(-)}$ denotes profinite completion). This follows from the equivalence $\HH(R) \iso \THH(R) \otimes_{\THH(\Z)} \Z$ and the fact that each homotopy group of $\fib(\THH(\Z) \to \Z)$ is finite, noting that the latter fact means that we also have an equivalence $\cpl{\THH(R)} \otimes_{\THH(\Z)} \Z \to \cpl{\HH(R)}$, by \cref{ev--arith--profinite-tensor} (cf. \cite[Lemma 2.5]{BMS}).
  \item Choose a commutative square
    \[
      \begin{tikzcd}
        \MU \ar[r] \ar[d] &
        S \ar[d] \\
        R \otimes \MU \ar[r] &
        {R'}
      \end{tikzcd}
    \]
    as in \cref{mot--xq--even-polynomial-surjection}, and note that the induced square
    \[
      \begin{tikzcd}
        \Z \otimes \MU \ar[r] \ar[d] &
        \Z \otimes S \ar[d] \\
        R \otimes \MU \ar[r] &
        {R'}
      \end{tikzcd}
    \]
    is also such a square. Then, as in the proof of \cref{mot--fil--p-compatibility}, the map $\THH(R) \to \THH(R'/S)$ is eff by \cref{ev--desc--novikov-eff,mot--fil--quasismooth-THH}, and both $\THH(R'/S)$ and its base change $\HH(R') \otimes_{\THH(R')} \THH(R'/S) \iso \THH(R'/\Z \otimes S)$ are even and have even profinite completions with bounded $p$-power torsion for all primes $p$ by \cref{mot--fil--quasiregular-quotient-THH}. \qedhere
  \end{itemize}
\end{proof}

%% file: AdamsSummand.tex
In this section we compute the mod $(p,v_1)$ syntomic cohomology of the connective Adams summand $\ell$ at a fixed prime $p$, where $\ell$ is given its canonical $\mathbb{E}_{\infty}$-ring structure constructed in \cite[Proposition 6.4]{BakerRichter}.  In other words, we compute $\pi_*\left(\gr^*_{\mot}\TC(\ell)  / (p,v_1) \right)$.  We allow $p$ to be any prime, including the primes $2$ and $3$ where the Smith Toda complex $V(1)=\mathbb{S}/(p,v_1)$ does not exist as a homotopy ring spectrum.  At primes $p \ge 5$, where $V(1)$ does exist as a homotopy ring spectrum, this gives an independent proof of the seminal calculation of $V(1)_*\mathrm{TC}(\ell)$ by Ausoni--Rognes \cite{AusoniRognes}. Before stating our main result, we precisely define the mod $(p,v_1,\cdots,v_k)$ reductions of prismatic and syntomic cohomology, even in the absence of a Smith--Toda complex: 


\begin{recollection}
The category of modules over $\fil^\star_{\ev} \mathbb{S}$ agrees with Pstragowski's category of even $\MU$-synthetic spectra \cite{Pstragowski}. Up to $p$-completion, this is also the category of cellular $\mathbb{C}$-motivic spectra. Proofs of these equivalences may be found in \cite[Theorem 7.34]{Pstragowski}, \cite[Theorem 6.12]{Cmot}, and \cite[Example C.16 and Proposition C.22]{Rmot}.

As a result of these equivalences, there is a symmetric monoidal inclusion of even $\MU_*\MU$-comodules into $\gr^*_{\ev} \mathbb{S}$-modules. We will only need this inclusion after $p$-completion, where it has been explored in detail in \cite{SpecialFiber}. Specifically, after $p$-completion, $\gr^*_{\ev} \mathbb{S}$ corresponds to the $p$-completed motivic cofiber of $\tau$.  Under the equivalence of \cite[Theorem 1.13(2)]{SpecialFiber}, which is symmetric monoidal by \cite[Remark 4.15]{SpecialFiber}, it follows that any $p$-torsion even commutative $\MU_*\MU$-comodule algebra gives rise to an $\mathbb{E}_{\infty}$-$\gr^*_{\ev} \mathbb{S}$-algebra.
\end{recollection}

\begin{definition} \label{dfn:reduced-greven}
For each $k \ge 0$, there is an even commutative $\MU_*\MU$ comodule algebra $\MU_* / (p,v_1,\cdots,v_k)$, obtained as an iterated cofiber of $\MU_*$, and we denote
the corresponding $\gr^*_\ev \mathbb{S}$-algebra by 
$\gr^*_\ev \mathbb{S}/(p, \cdots, v_k)$. More generally,
for any $\gr^*_\ev \mathbb{S}$-module $M$, we define
	\[
	M/(p, ..., v_k) := M\otimes_{\gr^*_\ev \mathbb{S}} \left(\gr^*_\ev \mathbb{S}/(p, \cdots , v_k) \right).
	\]
\end{definition}

\begin{remark}
If $R$ is any $\mathbb{E}_{\infty}$-ring, the unit map $\mathbb{S} \to R$ endows $\gr^*_{\ev} R$ with the natural structure of a $\gr^*_{\ev} \mathbb{S}$-algebra.  The same is true for the equivariant, $p$-complete, and cyclotomic variants of the even filtration, so for example $\gr^*_{\ev,\mathrm{TC},p} \THH(R)$ is a $\gr^*_{\ev} \mathbb{S}$-module.  In particular, if $R$ is chromatically quasisyntomic $\mathbb{E}_{\infty}$-ring, we may speak of its mod $(p,\cdots,v_k)$ syntomic cohomology $\pi_*\gr^*_{\mot} \TC(R) / (p,v_1,\cdots,v_k)=\pi_* \left( \gr^*_{\mot} \TC(R) / (p,v_1,\cdots,v_k) \right)$.
\end{remark}


Our main result is as follows
(where we recall \cref{convention-adams-grading} for the notion of
Adams weight):

\begin{theorem} \label{thm:ellsyntomic}
The mod $(p,v_1,v_2)$ syntomic cohomology of $\ell$ is a finite $\mathbb{F}_p$-vector space. As a vector space, it is isomorphic to
\begin{enumerate}
\item $\mathbb{F}_p \{1\}$, in Adams weight $0$ and degree $0$.
\item $\mathbb{F}_p\{\partial,t^{d}\lambda_1,t^{dp}\lambda_2\text{ }|\text{ }0 \le d < p\}$, in Adams weight $1$.  Here, $|\partial|=-1$, $|t^d\lambda_1|=2p-2d-1$, and $|t^{dp}\lambda_2|=2p^2-2dp-1$.
\item $\mathbb{F}_p\{t^d\lambda_1\lambda_2, t^{dp} \lambda_1 \lambda_2, \partial \lambda_1,\partial\lambda_2\text{ }|\text{ }0 \le d < p\}$, in Adams weight $2$.  Here, $|t^d \lambda_1\lambda_2|=2p^2-2p-2d-2$, $|t^{dp} \lambda_1\lambda_2|=2p^2-2p-2dp-2$, $|\partial \lambda_1|=2p-2$, and $|\partial\lambda_2|=2p^2-2$.
\item $\mathbb{F}_p\{\partial\lambda_1\lambda_2\}$, in Adams weight $3$ and degree $2p^2+2p-3$.
\end{enumerate}
The $v_2$-Bockstein spectral sequence (converging to the mod $(p,v_1)$ syntomic cohomology of $\ell$ as an $\mathbb{F}_p[v_2]$-module) collapses with no differentials.
\end{theorem}

\includegraphics[scale=1,trim={4cm 18.5cm 3.5cm 2.7cm},clip]{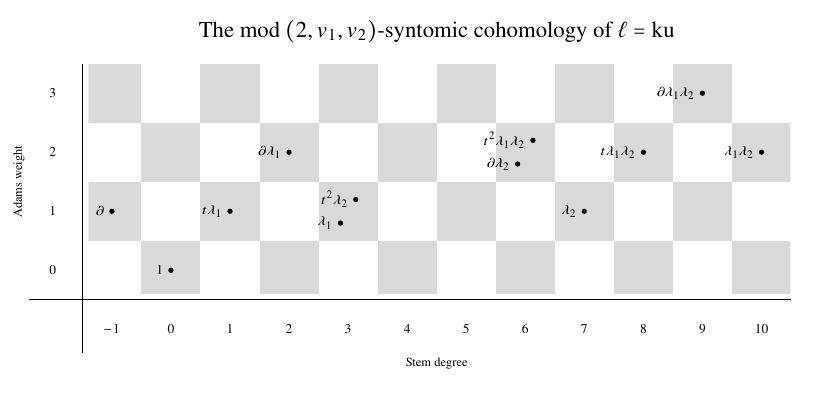}

We will spend most of the remainder of this section proving the first part of \Cref{thm:ellsyntomic}.  The statement about the $v_2$-Bockstein spectral sequence is comparatively straightforward.  Indeed, differentials in the $v_2$-Bockstein spectral sequence increase Adams weight by $1$.  Since the degree of $v_2$ is $2p^2-2$, no differentials are possible for bidegree reasons.  There are no $\mathbb{F}_p[v_2]$-module extension problems to solve, and we conclude \Cref{thm:introSynell}.

\begin{remark} \label{rmk:GammaStar}
Suppose that $p \ge 3$, so that the Smith--Toda complex $V(1)=\mathbb{S}/(p,v_1)$ exists as a spectrum.  Then the mod $(p,v_1)$ syntomic cohomology of $\ell$ is the $\mathrm{E}_2$-page of a spectral sequence converging to $V(1)_*\TC(R)$.  This is the spectral sequence associated to the filtered spectrum \[\fil^{\star}_{\mot}\TC(\ell) \otimes_{\fil^{\star}_\ev \mathbb{S}} \fil^\star_{\ev/\mathbb{S}}(V(1)).\] Here, $\fil^\star_{\ev/\mathbb{S}}$ refers to the even filtration for $\mathbb{S}$-modules developed in \Cref{mod} (or, equivalently, the $\Gamma_\star$ functor from \cite{Cmot}).
\end{remark}

\subsection{Recollections from previous work}

Our computations rely upon a few fundamental facts about $\THH(\ell)$ that can be proven by non-motivic means.  The first fact, due originally to McClure--Staffeldt at odd primes  \cite{McClureStaffeldt} and Angeltveit--Rognes at the prime $2$ \cite{AngeltveitRognes}, is the calculation of $\THH_*(\ell)/(p,v_1)$:

\begin{theorem} \label{thm:THHellwithFpcoeffs}
There is an isomorphism of algebras
\[\pi_{*} \left(\THH(\ell) \otimes_{\ell} \mathbb{F}_p\right) \cong \Lambda(\lambda_1,\lambda_2) \otimes \mathbb{F}_p[\mu],\] 
for classes $\lambda_1$ in degree $2p-1$, $\lambda_2$ in degree $2p^2-1$, and $\mu$ in degree $2p^2$.
\end{theorem}

\begin{remark}
$\THH(\ell) \otimes_{\ell} \mathbb{F}_p$ can be interpreted as the $\THH$ of $\ell$ with coefficients in the $\ell$-bimodule $\mathbb{F}_p$, and so is often denoted $\THH(\ell;\mathbb{F}_p)$.
\end{remark}

\begin{remark} \label{rmk:algebra-generators-ill-defined}
As stated, \Cref{thm:THHellwithFpcoeffs} does not precisely pin down the algebra generators $\mu$, $\lambda_1$, and $\lambda_2$; for example, we could always multiply a generator by an element of $\mathbb{F}_p^{\times}$ to obtain a different generator.  Using the motivic cobar complex, we will in \Cref{thm:THH-mot-collapse} precisely pin down correspondingly named elements of $\pi_* \gr^*_{\mot} \THH(\ell) / (p,v_1)$.
\end{remark}

To state the second fact, note that the sequence of $\mathbb{E}_{\infty}$-ring maps
$$\ell \to \THH(\ell) \xrightarrow{\varphi} \THH(\ell)^{\tate \Cp}$$
allows us to view the cyclotomic Frobenius $\varphi$ as a map of $\ell$-algebras.  We then have the following result, which is a version of the Segal conjecture:

\begin{theorem} \label{thm:ellSegal}
The mod $(p,v_1)$ Frobenius map $\pi_{*} \left(\THH(\ell) \otimes_{\ell} \mathbb{F}_p\right) \xrightarrow{\varphi} \pi_{*}\left(\THH(\ell)^{\tate \Cp} \otimes_{\ell} \mathbb{F}_p\right)$ is identified under the isomorphism of \Cref{thm:THHellwithFpcoeffs} with the ring map
\[\Lambda(\lambda_1,\lambda_2) \otimes \mathbb{F}_p[\mu] \to \Lambda(\lambda_1,\lambda_2) \otimes \mathbb{F}_p[\mu^{\pm 1}]\] 
that inverts the class $\mu$.
\end{theorem}

\Cref{thm:ellSegal} was first proved for primes $p \ge 5$ by Ausoni--Rognes \cite[Theorem 5.5]{AusoniRognes}.
An argument at $p=2$ is given in the thesis of Sverre {L}un{\o}e{-}{N}ielsen \cite{sverre-thesis}.
The result can be deduced at all primes from the discussion of the first and third authors in \cite[\textsection 4]{HahnWilson}.
To recall that argument, note that filtering $\ell$ by its $\mathbb{F}_p$-Adams filtration gives a map of spectral sequences, converging to the mod $(p,v_1)$ Frobenius, which on $\mathrm{E}_2$-pages is $\pi_{*}$ of
\[\THH(\mathbb{F}_p[v_0,v_1]) \otimes_{\mathbb{F}_p[v_0,v_1]} \mathbb{F}_p \xrightarrow{\varphi} \THH(\mathbb{F}_p[v_0,v_1])^{\tate \Cp} \otimes_{\mathbb{F}_p[v_0,v_1]} \mathbb{F}_p.\]
This map of $E_2$-pages can be identified with the ring map 
\[\mathbb{F}_p[x] \otimes \Lambda(\sigma v_0,\sigma v_1) \to \mathbb{F}_p[x^{\pm 1}] \otimes \Lambda(\sigma v_0,\sigma v_1)\]
that inverts the degree $2$ B\"okstedt generator $x \in \THH_{*}(\mathbb{F}_p)$.
In both spectral sequences, the class $\mu$ is detected by $x^{p^2}$.
Upon checking that $\varphi(\mu)$ is invertible, and not merely invertible on the $\mathrm{E}_2$-page, we can deduce \Cref{thm:ellSegal} by the fact that it is true on associated graded.

It follows from the Leibniz rule that $x^{-p^2}$ is a permanent cycle in the codomain spectral sequence, and therefore that $\varphi(\mu)$ has an inverse up to higher filtration.
By completeness of the filtration on homotopy groups,
$\varphi(\mu)$ is itself invertible.

\subsection{Naming conventions in the cobar complex}
In the next sections we begin our explicit computations of the prismatic and syntomic cohomologies of $\ell$. We make most of the former computation via a convenient, specific resolution, which was also used in \cite[Section 6]{HahnWilson}.

To begin, we note the following folklore proposition:\footnote{Since we are interested only in mod $p$ syntomic cohomology, it would also suffice to find an $\MU$-orientation of the $p$-completion of $\ell$.  This may be slightly simpler, since $\mathrm{gl}_1(\ell_p^{\wedge})$ is a retract of $\mathrm{gl}_1(\mathrm{ku}_p^{\wedge})$ via $\mathbb{E}_{\infty}$ Adams operations. There are several recorded proofs of $\mathbb{E}_{\infty}$-$\MU$-orientations of $\mathrm{ku}_{(p)}$ and $\mathrm{ku}_p^{\wedge}$, but we were unable to find a reference for an $\mathbb{E}_{\infty}$-orientation of the $p$-local ring $\ell$.  As a final remark, an alternative construction of an $\mathbb{E}_{\infty}$-orientation of $\ell$ can be extracted from \cite[Theorem 32]{HopkinsLawson}, using that $\ell$ is the connective cover of its $L_1$ localization, and checking the Ando criterion by means of the injection from $\ell$ to $\mathrm{ku}_{(p)}$.}

\begin{proposition}
There exists an $\mathbb{E}_{\infty}$-ring map
\[\MU \to \ell,\]
such that the induced map on $p$-localizations $\MU_{(p)} \to \ell$ is a quasiregular quotient.
\end{proposition}

\begin{proof}
First, let us prove that there exists some $\mathbb{E}_{\infty}$-ring map $\MU \to \ell$, without yet worrying about whether the map is a quasiregular quotient.

This is equivalent to asking that
\[\Sigma^2 \mathrm{ku} \xrightarrow{J} \Sigma \mathrm{gl}_1(\mathbb{S}) \to \Sigma \mathrm{gl}_1(\ell)\]
be nullhomotopic, and we will in fact observe that \emph{any} spectrum map $\Sigma^2 \mathrm{ku} \to \Sigma \mathrm{gl}_1(\ell)$ must be nullhomotopic.
The basic strategy will be that of \cite[Section 4]{HahnYuan}, building on \cite{AdamsPriddy} and \cite{LeeBPn}.

We first show that there is an equivalence of spectra
\begin{equation}
  \label{gl1ell}
  \tau_{\ge 2} \Sigma \mathrm{gl}_1(\ell) \simeq \Sigma^{2p-1}\ell,
\end{equation}
by an argument analogous to that in \cite[Lemma 4.5]{HahnYuan}. To begin, the equivalences of spaces
\[
  \Omega^{\infty} \tau_{\ge 2} \Sigma \mathrm{gl}_1(\ell) \simeq \Omega^{\infty} \tau_{\ge 2} \Sigma \ell \simeq \Omega^{\infty} \Sigma^{2p-1} \ell
\]
let us identify the homotopy groups of the two sides of \cref{gl1ell}. We next claim that, under this identification, the primary $k$-invariants of the two sides of \cref{gl1ell} induce the same maps in $\F_p$-cohomology. To check that two spectrum maps $\Z \to \Sigma^{2p-1}\Z$ induce the same map on $\F_p$-cohomology, it suffices to check that the same is true of their compositions with the reduction map $\Sigma^{2p-1}\Z \to \Sigma^{2p-1}\F_p$, as this latter map is surjective on $\F_p$-cohomology. Furthermore, the canonical map
\[
  \H^{2p-1}(\Z;\F_p) \to \H^{n+2p-1}(\mathrm{K}(\Z,n);\F_p)
\]
is injective for $n \ge 2p-1$. Thus, the claim about the $k$-invariants follows again from the above space-level equivalence.

Now, the spectral sequence
\[
  \H^*(\pi_*(\Sigma^{2p-1}\ell);\F_p) \implies \H^*(\Sigma^{2p-1}\ell;\F_p),
\]
determined by the Postnikov filtration of $\Sigma^{2p-1}\ell$, collapses after the first page onto the zero line (cf. the proof of \cite[Proposition 2.1]{AdamsPriddy}). The preceding paragraph implies that the same must be true of the analogous spectral sequence
\begin{align*}
  &\H^*(\pi_*(\tau_{\ge 2} \Sigma \mathrm{gl}_1(\ell));\F_p) \implies \H^*(\tau_{\ge 2} \Sigma \mathrm{gl}_1(\ell);\F_p)
\end{align*}
and that we have an isomorphism
\[
  \H^*(\tau_{\ge 2} \Sigma \mathrm{gl}_1(\ell);\F_p) \iso \H^*(\Sigma^{2p-1}\ell;\mathbb{F}_p)
\]
as modules over the mod $p$ Steenrod algebra. Since $\tau_{\ge 2} \Sigma \mathrm{gl}_1(\ell)$ is a finite type spectrum, it then follows from \cite[Theorems 1.1-1.2]{LeeBPn} that there is an equivalence \cref{gl1ell}.

Now, since $\ell$ is $p$-local, spectrum maps $\Sigma^2 \mathrm{ku} \to \Sigma^{2p-1}\ell$ factor through spectrum maps $\Sigma^2 \mathrm{ku}_{(p)} \to \Sigma^{2p-1} \ell$, and $\Sigma^2 \mathrm{ku}_{(p)} \simeq \bigoplus_{1 \le q \le p-1} \Sigma^{2q} \ell$.  By \cite[Proposition 1.6]{LeeBPn}, it follows that all spectrum maps $\Sigma^2 \mathrm{ku} \to \Sigma^{2p-1} \ell$ are nullhomotopic.

Finally, we claim that any $\mathbb{E}_{\infty}$-ring map $\MU_{(p)} \to \ell$ is a quasiregular quotient.  Indeed, we may choose polynomial generators $\pi_*\MU_{(p)} \cong \mathbb{Z}_{(p)}[x_1,x_2,\cdots]$ where $|x_i|=2p^i-2$, and such that the Quillen idempotent $\BP_* \to (\MU_{(p)})_*$ sends $v_1$ to $x_{p-1}$ modulo $p$.  For degree reasons, $x_1,\cdots,x_{p-2}$ map to zero in $\pi_*\ell$.  Comparing the first spectrum $k$-invariants of $\BP$ and $\ell$, we learn that $x_{p-1}$ maps to a polynomial generator of $\pi_*\ell$.  Possibly renaming $x_i$ for $i>p-1$, by subtracting multiples of powers of $x_{p-1}$, we may assume that all $x_i$ with $i \ne p-1$ map to zero in $\pi_*\ell$.
\end{proof}

In light of the above proposition, we make the following definition:

\begin{definition} \label{dfn:ellcobar}
The \emph{cobar complex} computing $\gr^*_{\mot}\THH(\ell)$ is the $E_1$-page of the descent spectral sequence for
\[\THH(\ell) \to \THH(\ell/\MU)=\THH(\ell/\MU_{(p)}).\]
This cobar complex has $s$th term given by $\pi_*\THH(\ell/\MU^{\otimes s+1})$.  Cocycles in the $s$th term represent elements in the $s$th Adams weight of $\gr^*_{\mot} \THH(\ell)$.

Similarly, we refer to $\pi_*(\THH(\ell/\MU^{\otimes \bullet+1})^{\h\cir})$ and $\pi_* (\THH(\ell/\MU^{\otimes \bullet+1})^{\tate\cir})$ as the cobar complexes computing $\gr^*_{\mot} \TC^{-}(\ell)$ and $\gr^*_{\mot} \TP(\ell)$, respectively.  There are canonical maps from the cobar complex computing $\gr^*_{\mot} \TC^{-}(\ell)$ to the cobar complexes computing $\gr^*_{\mot} \TP(\ell)$ and $\gr^*_{\mot}\THH(\ell)$, respectively.
\end{definition}

We emphasize again that the cobar complexes of \Cref{dfn:ellcobar} are merely our preferred presentations for the much more canonical $\gr^*_{\mot}\THH(\ell)$, $\gr^*_{\mot}\TC^{-}(\ell)$, and $\gr^*_{\mot}\TP(\ell)$.  In particular, these presentations were already studied in \cite[Section 6]{HahnWilson}.

Using the Quillen map $\mathrm{BP} \to \MU_{(p)}$, as well as various unit maps such as $\mathrm{MU}_{(p)} \to \TC^{-}(\ell/\MU_{(p)}) = \TC^{-}(\ell/\MU),$
we obtain a map from the Adams--Novikov $\mathrm{E}_1$-page to the cobar complex computing $\gr^*_{\mot} \TC^{-}(\ell)$.  
In particular, any Adams--Novikov cocycle names a cocycle in the cobar complex for $\gr^*_{\mot}\TC^{-}(\ell)$, and, via the canonical maps, also names cocycles in the cobar complexes for $\gr^*_{\mot} \TP(\ell)$ and $\gr^*_{\mot} \THH(\ell)$.  

When referring to elements of the Adams--Novikov $\mathrm{E}_1$-page, we will use the (standard) conventions of \cite[\textsection 3]{WilsonSampler}.  For example, $t_1^2+v_1t_1$, which is a cocycle at $p=2$ representing $\nu \in \pi_*(\mathbb{S})$, also represents a class of Adams weight $1$ in $\pi_* \gr^*_{\mot}\TC^{-}(\ell)$.

\begin{remark}
We will need the fact that $t_1$ is a cocycle on the Adams--Novikov $E_1$-page, or in other words that $d(t_1)=0$.  This furthermore implies that $d(t_1^p) \equiv 0$ modulo $p$.
\end{remark}

\begin{definition}
Using the unit map $\mathrm{BP}^{\h\cir} \to \MU_{(p)}^{\h\cir} \to \mathrm{TC}^{-}(\ell/\MU)$, we send the standard complex orientation of $\mathrm{BP}$ from \cite[Section 3]{WilsonSampler} to a class in the cobar complex for $\gr^*_{\mot} \TC^{-}(\ell)$ that we name $t$.
\end{definition}

\begin{remark}
We have the standard formula from \cite[Lemma 3.14]{WilsonSampler}:
$$\eta_R(t)=c(t+_{\mathbb{G}} t_1t^p+_{\mathbb{G}}t_2t^{p^2}+\cdots),$$
where $+_{\mathbb{G}}$ denotes addition using the $\mathrm{BP}_*$ formal group law and $c$ denotes the conjugation action on $\mathrm{BP}_*\mathrm{BP}$, extended by the rule that $c(t)=t$.  All we really need from this formula is the fact that $\eta_R(t) \equiv t-t_1t^p$ modulo $p,v_1$, and $t^{p+2}$.  This implies that $\eta_R(t^p) \equiv t^p-t_1^pt^{p^2}$ modulo $p,v_1,$ and $t^{p^2+2p}$.
\end{remark}

\begin{remark}
Each term $\THH_*(\ell/\MU^{\otimes s+1})$ in the cobar complex for $\gr^*_{\mot}\THH(\ell)$ is a free $\pi_* \ell$-module, and so has $(p,v_1)$ as a regular sequence.
By the mod $(p,v_1)$ cobar complex for $\gr^*_{\mot} \THH(\ell)$ we will mean the complex obtained by modding out both $p$ and $v_1$ levelwise.
This complex computes $\left(\gr^*_{\mot} \THH(\ell)\right) / (p,v_1)$, which is the $E_2$-term of the motivic spectral sequence for $\pi_*\left(\THH(\ell) \otimes_{\ell} \mathbb{F}_p\right)$.

We may also speak of the mod $(p,v_1)$ cobar complex for $\gr^*_{\mot} \TC^{-}(\ell)$, which has $s$th term isomorphic to $\THH_\star(\ell/\MU^{\otimes s+1})\llbracket t \rrbracket / (p,v_1)$, and analogously the mod $(p,v_1)$ cobar complex for $\gr^*_{\mot} \TP(\ell)$.  At the prime $2$, where $V(1)=\mathbb{S}/(2,v_1)$  does not exist, these cobar complexes do not necessarily compute the $\mathrm{E}_2$-page of any topologically relevant spectral sequence.
\end{remark}

\subsection{Hochschild homology}

The motivic spectral sequence for $\THH(\ell) \otimes_{\ell} \mathbb{F}_p$ was computed in \cite[\textsection 6.1]{HahnWilson}, and we recall the results below.

\begin{lemma}
In the mod $(p,v_1)$ cobar complex for $\gr^*_{\mot}\TC^{-}(\ell)$, the elements $v_2$ and $t_1$ are both divisible by $t$.
\end{lemma}

\begin{proof}
It suffices to check that $v_2$ is sent to zero under the composite non-equivariant ring map
\[\BP \to \MU_{(p)} \to \THH(\ell/\MU) / (p,v_1),\] 
and that $t_1$ is sent to zero under the composite
\[\BP \otimes \BP \to \MU_{(p)} \otimes \MU_{(p)} \to \THH(\ell/\MU \otimes \MU) / (p,v_1).\]
Non-equivariantly, the unit map $\MU_{(p)} \to \THH(\ell/\MU)$ factors through the unit map $\MU_{(p)} \to \ell$.  Thus, for the first statement it suffices to note that $v_2$ is sent to zero under the composite
\[\BP \to \MU_{(p)} \to \ell / (p,v_1) \iso \mathbb{F}_p,\]
and similarly $t_1$ is sent to zero under the composite
\[\BP \otimes \BP \to \MU_{(p)} \otimes \MU_{(p)} \to \ell / (p,v_1). \qedhere\]
\end{proof}

\begin{definition}
  Following \cite{HahnWilson}, we write $\sigma^2v_2$ and $\sigma^2t_1$ for the elements in the mod $(p,v_1)$ cobar complex for $\gr^*_{\mot}\TC^{-}(\ell)$ defined by the relations $t\sigma^2v_2=v_2$ and $t\sigma^2t_1=t_1,$ respectively.  Using the canonical map between the cobar complex for $\gr^*_{\mot}\TC^{-}(\ell)$ and the cobar complex for $\gr^*_{\mot}\THH(\ell)$, we may also speak of classes $\sigma^2v_2$ and $\sigma^2t_1$ in the mod $(p,v_1)$ cobar complex for $\gr^*_{\mot}\THH(\ell)$
\end{definition}

The following result is essentially the case $n=1$ of \cite[Proposition 6.1.6]{HahnWilson}:

\begin{theorem} \label{thm:THH-mot-collapse}
The motivic spectral sequence for $\pi_*\left(\THH(\ell) \otimes_{\ell} \mathbb{F}_p\right)$ collapses at the $\mathrm{E}_2$-page.  In the mod $(p,v_1)$ cobar complex for $\gr^*_{\mot}\THH(\ell)$, the elements $\sigma^2v_2,\sigma^2t_1,$ and $(\sigma^2t_1)^{p}$ are cocycles representing $\mu,\lambda_1,$ and $\lambda_2$, respectively. Here, $\mu$ has degree $2p^2$ and Adams weight $0$, $\lambda_1$ has degree $2p-1$ and Adams weight $1$, and $\lambda_2$ has degree $2p^2-1$ and Adams weight $1$.
\end{theorem}

\begin{proof}
The case $n=1$ of \cite[Proposition 6.1.6]{HahnWilson} exactly states that this motivic spectral sequence collapses, with $\mu$ detected in Adams weight $0$ and $\lambda_1,\lambda_2$ detected in Adams weight $1$. We also recall from \cite[Construction 6.1.5]{HahnWilson} that $\mu$ is detected by $\sigma^2v_2 \in \THH(\ell/\MU;\mathbb{F}_p)$ and that we have an identification
\[
  \THH(\ell) \otimes_{\ell} \mathbb{F}_p \iso \mathbb{F}_p \otimes_{\mathbb{F}_p \otimes \ell} \mathbb{F}_p
\] 
under which $\lambda_1 \doteq \sigma t_1$ and $\lambda_2 \doteq \sigma t_2$. Unravelling the identification
\[
\THH(\ell/\MU^{\otimes 2};\F_p) \iso \THH(\ell/\MU;\mathbb{F}_p) \otimes_{\THH(\ell;\mathbb{F}_p)} \THH(\ell/\MU;\mathbb{F}_p),
\]
it follows that $\lambda_1,\lambda_2$ are detected, respectively, by $\sigma^2 t_1,\sigma^2 t_2 \in \THH(\ell/\MU^{\otimes 2};\mathbb{F}_p)$. To finish, we need to show that the element $(\sigma^2t_1)^p \in \THH(\ell/\MU^{\otimes_2};\mathbb{F}_p)$ also detects $\lambda_2$, i.e. that $(\sigma^2t_1)^p \doteq \sigma^2 t_2$.



This last claim is equivalent to checking that the appropriate Dyer--Lashof operation in the $\E_\infty$-$\F_p$-algebra $\THH(\ell;\F_p)$ relates $\lambda_1$ and $\lambda_2$. This we may do by following the proof of \cite[Lemma 2.4.1]{HahnWilson}, i.e. by noting that $\sigma$ is compatible with Dyer--Lashof operations and that there is a Dyer--Lashof operation relating $t_1$ and $t_2$ in 
$(\mathbb{F}_p)_*(\ell)$ (as proved using the maps $(\mathbb{F}_p)_*\mathrm{BU} \cong (\mathbb{F}_p)_*\MU \to (\mathbb{F}_p)_*\ell)$. (We also note that this operation played an important role in the calculations of Ausoni--Rognes \cite[\textsection 2]{AusoniRognes}.)
\end{proof}

\includegraphics[scale=1,trim={4.0cm 20.0cm 2.5cm 2.5cm},clip]{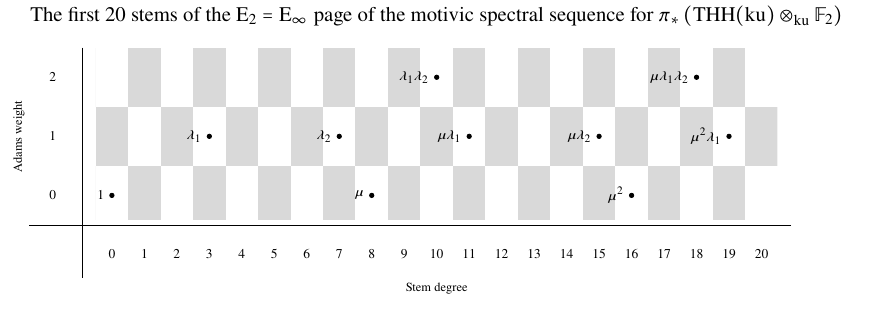}


\begin{remark}
Each term $\THH_*(\ell/\MU^{\otimes s+1})/(p,v_1)$ in the mod $(p,v_1)$ cobar complex for $\gr^*_{\mot} \THH(\ell)$ is a free $\mathbb{F}_p[\sigma^2v_2]$-module.  This means in particular that $(p,v_1,v_2=t \sigma^2v_2)$ is a regular sequence in $\TC^{-}(\ell/\MU^{\otimes s+1})$ and in $\TP(\ell/\MU^{\otimes s+1})$, so we may profitably speak of e.g. the mod $(p,v_1,v_2)$ cobar complex for $\gr^*_{\mot} \TP(\ell)$.
\end{remark}

\subsection{A useful isomorphism}

In order to calculate the mod $(p,v_1,v_2)$ syntomic cohomology of $\ell$ we will first need to compute the mod $(p,v_1,v_2)$ prismatic cohomology of $\ell$, or in other words $\pi_*\gr^*_{\mot} \TP(\ell)/(p,v_1,v_2)$.  
A useful and perhaps surprising fact is that $\gr^*_{\mot} \TP(\ell) / (p,v_1,v_2)$ has a second interpretation: as we explain in this brief section, it is canonically isomorphic to $\left(\gr^*_{\mot}\THH(\ell)^{\tate \Cp}\right)/(p,v_1)$ (cf. \cite[Construction 3.3]{BhattMathew}). For the purposes of this section we are defining
$\fil^{\star}_{\mot}\THH(\ell)^{\tate \Cp}$ via the pushout in 
filtered $\mathbb{E}_\infty$-rings:
	\[
	\begin{tikzcd}
	\fil^\star_\mot \TC^{-}(\ell)\arrow{r}{\varphi}\ar[d] & 
	\fil^\star_\mot \TP(\ell) \ar[d]\\
	\fil^\star_\mot \THH(\ell)\ar[r] &
	\fil^\star_\mot \THH(\ell)^{\tate \Cp}.
	\end{tikzcd}
	\]
To define the isomorphism, note that $v_2=0$ in $\pi_*\gr^*_{\ev} \ell/(p,v_1)$, so the sequence of algebra maps
\[\gr^*_{\ev} \ell \to \gr^*_{\mot} \THH(\ell) \xrightarrow{\varphi} \gr^*_{\mot} \THH(\ell)^{\tate \Cp}\]
imply that $v_2=0$ in $\pi_*\gr^*_{\mot} \THH(\ell)^{\tate \Cp}/(p,v_1)$.  Thus, the natural map
\[\gr^*_{\mot} \TP(\ell)/(p,v_1) \to \gr^*_{\mot} \THH(\ell)^{\tate \Cp} / (p,v_1) \]
factors over a map 
\[g:\gr^*_{\mot} \TP(\ell)/(p,v_1,v_2) \to \gr^*_{\mot} \THH(\ell)^{\tate \Cp} / (p,v_1) \]

\begin{theorem} \label{thm:HodgeTateIso}
The map $g$ above is an equivalence
\[\left(\gr^*_{\mot} \TP(\ell)\right)/(p,v_1,v_2) \isoto \left(\gr^*_{\mot} \THH(\ell)^{\tate \Cp}\right) / (p,v_1).\]
\end{theorem}

\begin{proof}
As usual, we may compute $\pi_*\gr^*_{\mot} \TP(\ell)/(p,v_1)$ via the mod $(p,v_1)$ cobar complex, which has $s$th term given by $\pi_*\left(\TP(\ell/\MU^{\otimes s+1})/(p,v_1) \right)$.  By the lemma below, it will then suffice to prove for each integer $s$ that the ideals generated by $[p](t)$ and $v_2$ are the same inside of the ring $\pi_*\TP(\ell/\MU^{\otimes s+1}) / (p,v_1)$.  In fact, since $\MU$ maps into $\MU^{\otimes s+1}$, it suffices to prove this when $s=0$.

In a moment, we shall prove that $v_2=0$ in $\pi_*(\THH(\ell/\MU)^{\tate\Cp} / (p,v_1))$.  Assuming this, it follows that we may write $v_2$ as a product
\[v_2=([p](t))(f)\]
for some $f \in \pi_*\TP(\ell/\MU) / (p,v_1)$.
Since $[p](t)=v_2t^{p^2}+\mathcal{O}(t^{p^2+2})$ in $\pi_*\BP^{\h\cir}/(p,v_1)$, and $v_2=t \sigma^2v_2$, we may rewrite this equation as
\[t\sigma^2v_2 = (t^{p^2+1} \sigma^2v_2 + \mathcal{O}(t^{p^2+2}) )(f).\]
Note that $\sigma^2v_2$ is a non-zerodivisor in $\pi_*\left(\THH(\ell/\MU) / (p,v_1) \right)$, by \cite[Theorem 2.5.4]{HahnWilson}.  Thus, $f$ must be represented in the $\cir$ Tate fixed point spectral sequence by $t^{-p^2}$, and so in particular $f$ is a unit.

It remains to prove that $v_2=0$ in $\pi_*(\THH(\ell/\MU)^{\tate\Cp} / (p,v_1))$.  To do so, observe that there is an $\mathbb{E}_{\infty}$-ring map $\MW \to \MU$, since $\MU$ is even.  This induces an $\cir$-equivariant ring map $\alpha:\THH(\ell/\MW) \to \THH(\ell/\MU)$.  Since $\MW$ is a cyclotomic base, we may form a sequence of maps
\[\ell \to \THH(\ell/\MW) \xrightarrow{\varphi} \THH(\ell/\MW)^{\tate\Cp} \xrightarrow{\alpha^{\tate\Cp}} \THH(\ell/\MU)^{\tate\Cp},\]
and so it suffices to note that $v_2=0$ in $\pi_* \ell / (p,v_1)$.  Note here that, even if a ring spectrum $X$ is given two different $\MU$-algebra structures, the image of $v_2$ in $\pi_*X$ is well-defined modulo $(p,v_1)$. \qedhere

\end{proof}

\begin{lemma} Let $M \in \mathrm{Mod}_{\mathrm{MU}}^{\Bcir}$
be an $\cir$-equivariant $\mathrm{MU}$-module. Then the map
	\[
	M^{\tate\cir}/[p](t) = M^{\tate\cir}\otimes_{\mathrm{MU}^{\tate\cir}}
	\mathrm{MU}^{\tate \Cp} \to M^{\tate \Cp}
	\]
is an equivalence.
\end{lemma}
\begin{proof} We argue exactly as in \cite[IV.4.12]{NikolausScholze}.
 Since $\mathrm{MU}^{\tate \Cp} = 
\mathrm{MU}^{\tate\cir}/[p](t)$ is a perfect $\mathrm{MU}^{\tate\cir}$-module,
the functor $(-)\otimes_{\mathrm{MU}^{\tate\cir}}\mathrm{MU}^{\tate \Cp}$
commutes with all limits and colimits. 
Using the equivalences $M^{\tate G}=\colim (\tau_{\ge n}M)^{\tate G}$ and $M^{\tate G}=\lim (\tau_{\le n}M)^{\tate G}$ we are reduced
to the case when $M$ is bounded, and further to the case where
$M$ is discrete. In this case the $\cir$-action (and hence
the $\Cp$-action) must be trivial, and then one concludes by
direct observation.
\end{proof}

\begin{remark} \label{rmk:isE2tate}
We will study and compute the homotopy groups of $\gr^*_{\mot} \TP(\ell)/(p,v_1,v_2)$ at any prime $p$.  However, if the Smith--Toda complex $V(2)=\mathbb{S}/(p,v_1,v_2)$ does not exist, we cannot say that $\pi_* \gr^*_{\mot} \TP(\ell)/(p,v_1,v_2)$ is the $\mathrm{E}_2$-page of a spectral sequence converging to $V(2)_* \TP(\ell)$.  In contrast, it is always the case that $\pi_*\gr^*_{\mot} \THH(\ell)^{\tate \Cp} / (p,v_1)$ is the $\mathrm{E}_2$-page of a spectral sequence converging to $\THH(\ell)^{\tate \Cp} \otimes_{\ell} \mathbb{F}_p$.  This spectral sequence is the one associated to the cosimplicial object $\THH(\ell/\MW^{\otimes \bullet+1})^{\tate \Cp} \otimes_{\ell} \mathbb{F}_p$, where the map from $\ell$ into $\THH(\ell/\MW^{\otimes s+1})^{\tate \Cp}$ is given by the composite of the map $\ell \to \THH(\ell/\MW^{\otimes s+1})$ with the relative cyclotomic Frobenius.
\end{remark}

\subsection{Prismatic cohomology}

In this section we will calculate $\left(\gr^*_{\mot} \TP(\ell) \right)/ (p,v_1,v_2)$ and $\left(\gr^*_{\mot} \TC^{-}(\ell)\right)/(p,v_1,v_2)$.  Our strategy will be to use the second filtration, $\fil^{\bullet}_+$, introduced in \cref{SecEvenFiltration}.
We will call the resulting spectral sequences the \emph{algebraic $t$-Bockstein
spectral sequences}, and they have signature:
	\[
	\pi_*(\gr^*_\mot \THH(\ell))[t] \Rightarrow 
	\pi_*(\gr^*_\mot \TC^{-}(\ell))
	\]
	\[
	\pi_*(\gr^*_\mot \THH(\ell))[t^{\pm 1}] \Rightarrow 
	\pi_*(\gr^*_\mot \TP(\ell))
	\]
The map between the first and second spectral sequences which inverts
$t$ converges to the canonical map
$\gr^*_\mot \TC^{-} \to \gr^*_\mot \TP$.  Explicitly, these spectral sequences may be obtained by filtering the cobar complexes by powers of the ideal generated by $t$.

The elements $p$ and $v_1$ are detected by the likewise named elements
in $\pi_*(\gr^*_\mot \THH(\ell))$, and after killing these elements $v_2$ is detected by $t\mu$. We may therefore choose
appropriately
filtered lifts of $p, v_1,$ and $v_2$ and take the cofiber by each
element in turn to obtain spectral sequences with signatures:
	\[
	\mathbb{F}_p[t,\mu]/(t\mu) \otimes \Lambda(\lambda_1,\lambda_2)
	\Rightarrow 
	\pi_*(\gr^*_\mot \TC^{-}(\ell)/(p,v_1,v_2))
	\]
	\[
	\mathbb{F}_p[t^{\pm 1}] \otimes \Lambda(\lambda_1,\lambda_2) \Rightarrow 
	\pi_*(\gr^*_\mot \TP(\ell)/(p,v_1,v_2))
	\]
	
We explain the behavior of the second spectral sequence in the
following theorem.  The main tool used to deduce differentials is naturality from the algebraic $t$-Bockstein spectral sequence computing $\pi_*(\gr^*_{\mathrm{mot}} \mathrm{TP}(\mathbb{S}) / (p,v_1,v_2))$.

\begin{theorem}
The algebraic $t$-Bockstein spectral sequence for $\left(\mathrm{gr}^*_{\mathrm{mot}}\mathrm{TP}(\ell) \right) / (p,v_1,v_2)$ has $\mathrm{E}_1$-page given by $\mathbb{F}_p[t^{\pm 1}] \otimes \Lambda(\lambda_1,\lambda_2)$. The spectral sequence is determined by multiplicative structure together with the following facts:
\begin{enumerate}
\item The classes $t^{p^2},\lambda_1,$ and $\lambda_2$ are permanent cycles.
\item There is a $d_p$ differential $d_{p}(t)\doteq t^{p+1}\lambda_1$.
\item There is a $d_{p^2}$ differential $d_{p^2}(t^p) \doteq t^{p^2+p} \lambda_2$.
\end{enumerate}
The $\mathrm{E}_{\infty}$-page is $\mathbb{F}_p[t^{\pm p^2}] \otimes \Lambda(\lambda_1,\lambda_2)$.
\end{theorem}

\begin{proof}
In the mod $(p,v_1,v_2)$ cobar complex, we compute
$$\eta_R(t) \equiv t-t^{p}t_1 \text{ modulo }t^{p+2},$$
and then note that $t^{p}t_1=t^{p+1}\sigma^2t_1$.  Since $\lambda_1$ is represented by $\sigma^2 t_1$ in the mod $(p,v_1)$ cobar complex for $\THH(\ell)$, we conclude the claimed $d_p$ differential. Taking $p$th powers, we compute
$$\eta_R(t^p) \equiv t^p-t^{p^2}t_1^{p}\text{ modulo }t^{p^2+2p},$$
and then note that $t^{p^2}t_1^p = t^{p^2+p}(\sigma^2t_1)^p$.  Since $\lambda_2$ is represented by $(\sigma^2 t_1)^p$ in the mod $(p,v_1)$ cobar complex for $\THH(\ell)$, we conclude the claimed $d_{p^2}$ differential.

It remains to see that $\lambda_1$, $\lambda_2$, and $t^{p^2}$ are permanent cycles. First, we note that
$$\eta_R(t^{p^2}) \equiv t^{p^2} -t^{p^3}t_1^{p^2}=t^{p^2}-t^{p^3+p^2} (\sigma^2t_1)^{p^2} \text{ modulo }t^{p^3+2p^2}.$$ 
This is the same as $t^{p^2}$ modulo $t^{p^3+p^2}$, and so $t^{p^2}$ must survive to the $\mathrm{E}_{p^3}$ page of the spectral sequence.  For sparsity reasons, it follows that $t^{p^2}$ is a permanent cycle.

The only way $\lambda_1$ could fail to be a permanent cycle is via a differential $d_{p^2}(\lambda_1) \doteq \lambda_1\lambda_2t^{p^2}$.  If such a differential occurred, we would learn that $\left(\mathrm{gr}^*_{\mathrm{mot}}\mathrm{TP}(\ell) \right) / (p,v_1,v_2)$ is trivial in degree $2p-1$.  However, it is also the $\mathrm{E}_2$-page of a spectral sequence converging to $\THH(\ell)^{\tate \Cp} \otimes_{\ell} \mathbb{F}_p$, by \Cref{rmk:isE2tate}, and the latter object is nontrivial in degree $2p-1$.

By the previous paragraph, we may choose some element $x \in \pi_*\left( \TP(\ell/\MU^{\otimes 2}) / (p,v_1,v_2) \right)$ that represents $\lambda_1$ in the cobar complex.
Then $x^p \in \pi_*\left(\TP(\ell/\MU^{\otimes 2}) / (p,v_1,v_2) \right)$ will represent $\lambda_2$ (since $(\sigma^2t_1)^p$ is the $p$th power of $\sigma^2t_1$), so $\lambda_2$ is a permanent cycle.
\end{proof}

\includegraphics[scale=1,trim={4cm 18.5cm 3.5cm 2.5cm},clip]{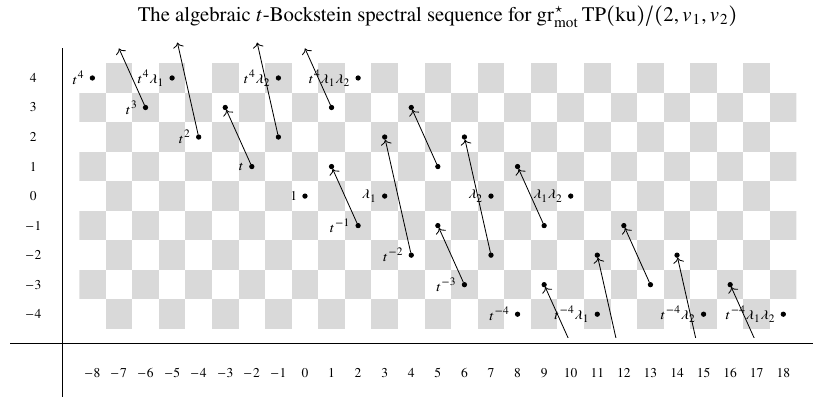}

\begin{corollary} \label{cor:TCminusBockstein}
The algebraic $t$-Bockstein spectral sequence for $\left(\mathrm{gr}^*_{\mathrm{mot}}\mathrm{TC}^{-}(\ell) \right) / (p,v_1,v_2)$ has $\mathrm{E}_1$-page given by $\mathbb{F}_p[t,\mu]/(t\mu) \otimes \Lambda(\lambda_1,\lambda_2)$.  The spectral sequence is determined by multiplicative structure together with the following facts:
\begin{enumerate}
\item The classes $t^{p^2},\lambda_1,$ $\lambda_2$, and $\mu$ are permanent cycles.
\item There is a $d_p$ differential $d_{p}(t) \doteq t^{p+1}\lambda_1$.
\item There is a $d_{p^2}$ differential $d_{p^2}(t^p) \doteq t^{p^2+p} \lambda_2$.
\end{enumerate}
The $\mathrm{E}_{\infty}$-page is $\mathbb{F}_p[t^{p^2}, \mu]/(t^{p^2} \mu) \otimes \Lambda(\lambda_1,\lambda_2) \oplus \mathbb{F}_p\{t^d \lambda_1, t^{pd} \lambda_2, t^d \lambda_1 \lambda_2, t^{pd} \lambda_1 \lambda_2 \text{ }|\text{ } 0 < d <p\}$.
\end{corollary}

\begin{proof}
The canonical map from the cobar complex for $\TC^{-}(\ell)$ to the cobar complex for $\TP(\ell)$ induces a map of $t$-Bockstein spectral sequences, from which we can read off the claimed differentials and the facts that $t^{p^2}$, $\lambda_1$, and $\lambda_2$ are permanent cycles.
\end{proof}

\includegraphics[scale=1,trim={4cm 18.5cm 3.5cm 2.5cm},clip]{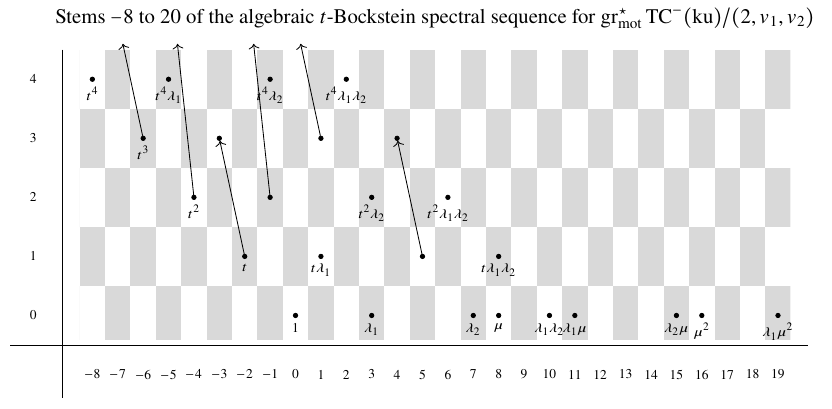}

\begin{corollary} \label{cor:ellCanMap}
There are isomorphisms of $\mathbb{F}_p$ vector spaces
\begin{align*}
\left(\mathrm{gr}^*_{\mathrm{mot}}\mathrm{TC}^{-}(\ell) \right) / (p,v_1,v_2) &\cong \mathbb{F}_p[t^{p^2}, \mu]/(t^{p^2}\mu) \otimes \Lambda(\lambda_1,\lambda_2)  \\ &\qquad\oplus  \mathbb{F}_p\{t^d \lambda_1, t^{pd} \lambda_2, t^d \lambda_1 \lambda_2, t^{pd} \lambda_1 \lambda_2 \text{ }|\text{ } 0 < d <p\},
\end{align*}
\[\left(\mathrm{gr}^*_{\mathrm{mot}}\mathrm{TP}(\ell) \right) / (p,v_1,v_2) \cong \mathbb{F}_p[t^{\pm p^2}] \otimes \Lambda(\lambda_1,\lambda_2)\]
The canonical map 
\[\left(\mathrm{gr}^*_{\mathrm{mot}}\mathrm{TC}^{-}(\ell) \right) / (p,v_1,v_2) \to \left(\mathrm{gr}^*_{\mathrm{mot}}\mathrm{TP}(\ell) \right) / (p,v_1,v_2)\]
sends each class of the form $\lambda_1^{\epsilon_1}\lambda_2^{\epsilon_2}t^{kp^2}$ to the correspondingly named class in the target, where $\epsilon_1,\epsilon_2\in\{0,1\}$ and $k \ge 0$.  It is zero on all other classes.
\end{corollary}

\begin{proof}
The only subtle point is to prove that classes not of the form $\lambda_1^{\epsilon_1} \lambda_2^{\epsilon_2} t^{kp^2}$ all map to zero.  The calculations above prove this to be the case after taking $t$-adic associated graded, and all non-zero classes in $\left(\mathrm{gr}^*_{\mathrm{mot}}\mathrm{TP}(\ell) \right) / (p,v_1,v_2)$ are in low enough $t$-adic filtrations that no filtration jumps are possible.
\end{proof}

\begin{remark} 
For $p \ge 5$, Ausoni and Rognes computed the homotopy fixed point spectral sequences for the $\mathbb{F}_p[v_2]$-module $V(1)_*\mathrm{TC}^{-}(\ell)$ \cite[\textsection 6]{AusoniRognes}.  As a corollary, one can straightforwardly deduce the $\mathbb{F}_p$-vector spaces $V(2)_*\mathrm{TC}^{-}(\ell)$ and $V(2)_*\mathrm{TP}(\ell)$, and observe a posteriori that the motivic spectral sequences for $V(2)_*\mathrm{TC}^{-}(\ell)$ and $V(2)_*\mathrm{TP}(\ell)$ degenerate at the $\mathrm{E}_2$-page.  The degeneration of these spectral sequences is in fact forced by bidegree reasons alone, since they are concentrated in a checkerboard pattern in Adams weights $0$, $1$, and $2$.
\end{remark}

\subsection{Syntomic cohomology} \label{subsec:finalsyntomic}

By \Cref{cor:ellCanMap}, we understand the canonical map 
\[\mathrm{can}:\left(\mathrm{gr}^*_{\mathrm{mot}}\mathrm{TC}^{-}(\ell) \right) / (p,v_1,v_2) \to \left(\mathrm{gr}^*_{\mathrm{mot}}\mathrm{TP}(\ell) \right) / (p,v_1,v_2).\]
To compute $\left(\mathrm{gr}^*_{\mathrm{mot}}\mathrm{TC}(\ell) \right) / (p,v_1,v_2)$, it remains to understand the Frobenius map 
\[\varphi:\left(\mathrm{gr}^*_{\mathrm{mot}}\mathrm{TC}^{-}(\ell) \right) / (p,v_1,v_2) \to \left(\mathrm{gr}^*_{\mathrm{mot}}\mathrm{TP}(\ell) \right) / (p,v_1,v_2).\]
For this, we contemplate the following diagram:
$$
\begin{tikzcd}
\left(\mathrm{gr}^*_{\mathrm{mot}}\mathrm{TC}^{-}(\ell) \right) / (p,v_1) \arrow{r}{\varphi} \arrow{d} & \left(\mathrm{gr}^*_{\mathrm{mot}}\mathrm{TP}(\ell) \right) / (p,v_1) \arrow{d}\\
\left(\mathrm{gr}^*_{\mathrm{mot}}\mathrm{THH}(\ell) \right) / (p,v_1) \arrow{r}{\varphi} & \left(\mathrm{gr}^*_{\mathrm{mot}}\mathrm{THH}(\ell)^{\tate \Cp} \right) / (p,v_1)
\end{tikzcd}
$$
Since $v_2=0$ in $\left(\mathrm{gr}^*_{\mathrm{mot}}\mathrm{THH}(\ell) \right) / (p,v_1)$, for example because $\mathrm{gr}^*_{\mathrm{mot}}\mathrm{THH}(\ell)$ is an algebra over $\gr^*_{\ev} \ell$, the diagram factors through a square of the form
$$
\begin{tikzcd}
\left(\gr^*_{\mathrm{mot}}\mathrm{TC}^{-}(\ell) \right) / (p,v_1,v_2) \arrow{r}{\varphi} \arrow{d}{f} & \left(\mathrm{gr}^*_{\mathrm{mot}}\mathrm{TP}(\ell) \right) / (p,v_1,v_2) \arrow{d}{g}\\
\left(\mathrm{gr}^*_{\mathrm{mot}}\mathrm{THH}(\ell) \right) / (p,v_1) \arrow{r}{\varphi} & \left(\mathrm{gr}^*_{\mathrm{mot}}\mathrm{THH}(\ell)^{\tate \Cp} \right) / (p,v_1)
\end{tikzcd}
$$
Here, the map $f$ is an isomorphism from the $0$-line of the spectral sequence of \Cref{cor:TCminusBockstein} onto $\left(\mathrm{gr}^*_{\mathrm{mot}}\mathrm{THH}(\ell) \right) / (p,v_1)$. It is trivial on classes above the $0$-line.  The map $g$ is the isomorphism of \Cref{thm:HodgeTateIso}.

\begin{corollary}  \label{cor:ellFrob}
In terms of the isomorphisms
\begin{align*}
\left(\mathrm{gr}^*_{\mathrm{mot}}\mathrm{TC}^{-}(\ell) \right) / (p,v_1,v_2) &\cong \mathbb{F}_p[t^{p^2}, \mu]/(t^{p^2}\mu) \otimes \Lambda(\lambda_1,\lambda_2)  \\ &\qquad\oplus  \mathbb{F}_p\{t^d \lambda_1, t^{pd} \lambda_2, t^d \lambda_1 \lambda_2, t^{pd} \lambda_1 \lambda_2 \text{ }|\text{ } 0 < d <p\},
\end{align*}
\[\left(\mathrm{gr}^*_{\mathrm{mot}}\mathrm{TP}(\ell) \right) / (p,v_1,v_2) \cong \mathbb{F}_p[t^{\pm p^2}] \otimes \Lambda(\lambda_1,\lambda_2)\]
of \Cref{cor:ellCanMap}, the Frobenius is trivial on classes not of the form $\lambda_1^{\epsilon_1} \lambda_2^{\epsilon_2} \mu^k$ where $k\ge 0$ and $\epsilon_1,\epsilon_2 \in \{0,1\}$.  On the other hand, the Frobenius sends each class of the form $\lambda_1^{\epsilon_1} \lambda_2^{\epsilon_2} \mu^k$ to an $\mathbb{F}_p^{\times}$ multiple of the class named $\lambda_1^{\epsilon_1}\lambda_2^{\epsilon_2} t^{-{p^2}k}$.
When $k=0$, this $\mathbb{F}_p^{\times}$ multiple can be determined, and in fact $\varphi(\lambda_1)=\lambda_1$ and $\varphi(\lambda_2)=\lambda_2$.
\end{corollary}

\begin{proof}
The map $f$ is already trivial on every class not of the form $\lambda_1^{\epsilon_1}\lambda_2^{\epsilon_2} \mu^k$.  \Cref{thm:ellSegal}, together with the fact that $g$ is an isomorphism, implies that each class of the form $\lambda_1^{\epsilon_1}\lambda_2^{\epsilon_2} \mu^k$ has non-trivial Frobenius image.  The only non-trivial classes in the codomain, in the same degree as $\lambda_1^{\epsilon_1}\lambda_2^{\epsilon_2} \mu^k$, are $\mathbb{F}_p^{\times}$ multiples of the class named $\lambda_1^{\epsilon_1}\lambda_2^{\epsilon_2} t^{-{p^2}k}$.

It remains to prove that $\varphi(\lambda_1)=\lambda_1$
and $\varphi(\lambda_2)=\lambda_2$. There is nothing
to prove when $p=2$. When $p\ge 3$, this follows from
\cite[Proposition 2.8]{AusoniRognes}, which proves that there
are elements in these degrees which survive to
$\mathrm{TC}$. 
\end{proof}

We can now deduce the main theorem of this section:

\begin{proof}[Proof of \Cref{thm:ellsyntomic}]
We deduce the first part of \Cref{thm:ellsyntomic} as an immediate consequence of the combination of \Cref{cor:ellCanMap} and \Cref{cor:ellFrob}, with the symbol $\partial$ decorating classes in $\left(\gr^*_{\mot}\TC(\ell)\right)/(p,v_1,v_2)$ that come from the cokernel of $\varphi-\mathrm{can}$. The second part of \Cref{thm:ellsyntomic}, about the $v_2$-Bockstein spectral sequence, follows by the argument given immediately after the theorem statement, which relies on the elementary lemma below applied
to $R=\gr^*_{\ev}(\mathbb{S})/(p,v_1)$
and $M = \gr^*_{\mot}(\TC(\ell))/
(p,v_1)$. We remind the reader that $v_2$ lives in $\pi_{2p^2-2}$ of the $(p^2-1)$'st graded piece of
$\gr^*_{\ev}(\mathbb{S})/(p,v_1)$, and that our
convention for displaying spectral sequences is to draw
a term from $\pi_nL^a$ of a graded object $L$ in column
$n$ and row $2a-n$. 
\end{proof}

\begin{lemma} Let $R$ be a graded ring, $M$ a graded
$R$-module. If $L^*$ is graded, write $\pi_{n,a}L$ for
$\pi_n(L^a)$. Then the Bockstein spectral sequence for
$\pi_{*,*}\left(\cpl{M}_x\right)$ associated
to an element $x \in \pi_{n,a}(R)$, with $\mathrm{E}_1$-page $\pi_{*,*}(M/x)[x]$, has $d_r$ differentials that send elements of bidegree $(m,b)$ to elements of bidegree $(m-rn-1,b-ra)$.
\end{lemma}

To finish the paper, we record proofs of the final two theorems mentioned in the introduction.

\begin{corollary} \label{cor:TCisfp}
For any prime number $p$ and $p$-local type $3$ complex $M$, $M_*\TC(\ell)$ is finite.
\end{corollary}

\begin{proof}
By thick subcategory considerations, it suffices to prove this for $M$ equal to a generalized Moore spectrum of the form $\mathbb{S}/(p^i,v_1^j,v_2^k)$, where $j\gg i$ and $ k \gg j$ so that killing $(p^i,v_1^j,v_2^k)$ is a well-defined operation in $\MU_*\MU$-comodules.  There is then a spectral sequence converging to $M_*\TC(\ell)$ beginning with $\gr^*_{\mot}(\TC(\ell)) / (p^i,v_1^j,v_2^k)$.  The latter object may be resolved by finitely many copies of $\gr^*_{\mot}(\TC(\ell)) / (p,v_1,v_2)$, and so is finite.
\end{proof}

As explained in \cite[\textsection 3]{HahnWilson}, \Cref{cor:TCisfp} implies that the map
\[\TC(\ell)_{(p)} \to L_2^{f}\TC(\ell)_{(p)}\]
is a $\pi_*$-iso in degrees $* \gg 0$, which can be seen as a telescopic analog of the Lichtenbaum--Quillen conjecture. In fact, one can localize at a wedge of Morava $K$-theories rather than a wedge of telescopes, which we record as our final result.

\begin{theorem}
The telescope conjecture is true of $\TC(\ell)$.  In other words, the natural map
$$L_2^{f} \TC(\ell) \to L_2 \TC(\ell)$$
is an equivalence.
\end{theorem}

\begin{proof}
We say that the height $2$ telescope conjecture holds for a spectrum $X$ if the natural map $L_2^{f}X \to L_2X$ is an equivalence.
First note that, since $L_2$ and $L_2^{f}$ are smashing localizations, if the height $2$ telescope conjecture holds for a ring $R$ then it also holds for every $R$-module.  We will prove the height $2$ telescope conjecture for $\TC^{-}(\ell)$.  Since $\TP(\ell)$ is a module over $\TC^{-}(\ell)$, we may conclude the height $2$ telescope conjecture for $\TP(\ell)$ and then, by the Nikolaus--Scholze fiber sequence, for $\TC(\ell)$.  

By the work of Mahowald and Miller \cite{MahowaldTelescope,MillerTelescope}, to prove the height $2$ telescope conjecture for $\TC^{-}(\ell)$ it suffices to prove it for $F \otimes \TC^{-}(\ell)$, where $F$ is a chosen $p$-local finite type $2$ complex.  To do this, consider the equivalence 
\[F \otimes\TC^{-}(\ell) \simeq F \otimes \left(\lim_{\Delta} \TC^{-}(\ell/\MU^{\otimes \bullet+1})\right) \simeq \lim_{\Delta} \left( F \otimes \TC^{-}(\ell/\MU^{\otimes \bullet+1})\right),\]
where we may pass $F$ inside of the totalization because it is finite. Now, the $L_2^{f}$ localization of $F \otimes\TC^{-}(\ell)$ is given by $v_2^{-1} F \otimes \TC^{-}(\ell)$.  We will prove in a moment that 
\[v_2^{-1}F \otimes\TC^{-}(\ell) \simeq \lim_{\Delta} \left(v_2^{-1}F \otimes \TC^{-}(\ell/\MU^{\otimes \bullet+1})\right).\]
Each term inside the totalization is an $L_2^{f}$-local $\MU$-module, and hence is $L_2$-local.  Thus, the totalization is also $L_2$-local.

It remains to check that 
\[v_2^{-1}F \otimes\TC^{-}(\ell) \simeq \lim_{\Delta} \left(v_2^{-1}F \otimes \TC^{-}(\ell/\MU^{\otimes \bullet+1})\right).\]
Combining e.g. \cite[Lemma 2.34]{clausen-mathew} and \cite[Proposition 3.12]{mathew-thick} (with $\mathcal{U}=\{\mathbb{S}\}$), this will hold if the descent spectral sequence computing $\pi_*(F \otimes \TC^{-}(\ell))$ from $\pi_*(F \otimes \TC^{-}(\ell/\MU^{\otimes \bullet+1}))$ admits a horizontal vanishing line at a finite page.

For simplicity, take $F$ to be a generalized Moore spectrum $\mathbb{S}/(p^{i},v_1^{j})$.  Then the motivic spectral sequence for $F_*\TC^{-}(\ell)$ is that associated to the filtered spectrum with \[\mathrm{fil}^n = \lim_{\Delta} (\tau_{\ge 2n} \left(\TC^{-}(\ell/\MU^{\otimes \bullet+1}) / (p^i,v_1^j) \right).\]  This is the d\'ecalage of the descent spectral sequence, and so agrees with the descent spectral sequence from the $\mathrm{E}_2$-page onward.  In other words, it suffices to prove that the motivic spectral sequence for $F_* \TC^{-}(\ell)$ admits a horizontal vanishing line. Now, $\pi_* \gr^*_{\mot} \TC^{-}(\ell) / (p^{i},v_1^{j})$ admits a finite filtration by copies of $\pi_*\gr^*_{\mot} \TC^{-}(\ell) / (p,v_1)$, and we have shown above that $\pi_* \gr^*_{\mot} \TC^{-}(\ell) / (p,v_1)$ has a horizontal vanishing line.
\end{proof}

%% file: ModuleAppendix.tex
In this appendix, we record an extension of the even filtration construction to the setting of modules over $\E_\infty$-rings, as opposed to the setting of $\E_\infty$-rings themselves that was treated in \cref{SecEvenFiltration}, and we establish descent results in this extended setting. Within the main body of the paper, this extension is only used in \cref{in--nu} and \cref{rmk:GammaStar}, in both cases regarding the functor denoted $\Gamma_\star$ in \cite{Cmot}.  We expect that the generality of this appendix will be useful in future applications (see, e.g., \cref{example:BPn}). 


\subsection{Defining the filtrations}
\label{mod--def}

\begin{notation}
  \label{mod--def--mod}
  Let $\Mod$ denote the category of pairs $(A,M)$ where $A$ is an $\E_\infty$-ring and $M$ is an $A$-module, and let $\Mod^{\ev}$ denote the full subcategory of $\Mod$ spanned by pairs $(A,M)$ where $A$ is even (but $M$ need not be). We denote by
  \begin{align*}
    U_{\Alg}&: \Mod \to \CAlg& (A,M)\mapsto A\\
    U_{\Mod}&: \Mod \to \Spt& (A,M) \mapsto M 
  \end{align*}
  the two forgetful functors.
\end{notation}

\begin{proposition}
  \label{mod--def--accessible}
  $\Mod^\ev$ is an accessible subcategory of $\Mod$.
\end{proposition}

\begin{proof}
  This follows from \cref{ev--def--accessible}, by virtue of \cite[Proposition 5.4.6.6]{htt} and the pullback square
  \[
    \begin{tikzcd}
      \Mod^\ev \ar[r, hook] \ar[d, "U_\Alg", swap] &
      \Mod \ar[d, "U_\Alg"] \\
      \CAlg^\ev \ar[r, hook] &
      \CAlg.
    \end{tikzcd}
    \qedhere
  \]
\end{proof}
 
\begin{construction}
  \label{mod--def--fil-ev}
  It follows from \cref{mod--def--accessible} that the composite functor
  \[
    \Mod^\ev \lblto{U_\Mod} \Spt \lblto{\tau_{\ge 2\star}} \FilSpt
  \]
  admits a right Kan extension along the inclusion $\Mod^\ev \subseteq \Mod$. We denote the resulting functor $\Mod \to \FilSpt$ by $(A,M) \mapsto \fil^{\star}_{\ev/A}(M)$, and refer to this construction as the \emph{even filtration}.
\end{construction}

\begin{remark}
  \label{mod--def--limit-formula}
  If $A \to B$ is a map of $\mathbb{E}_{\infty}$-rings and $M$ is an $A$-module, then $(B, M\otimes_AB)$ is initial among maps $(A,M) \to (B, N)$ in $\Mod$ lying over $A \to B$. It follows that the even filtration of \cref{mod--def--fil-ev} is given by the following limit expression:
  \[
    \fil^{\star}_{\ev/A}(M) \iso
    \lim_{A \to B, B \in \CAlg^{\ev}} \tau_{\ge 2\star}(M\otimes_AB);
  \]
  \cref{ev--def--accessible} implies that this limit is equivalent to a small limit and hence exists in the category $\FilSpt$.
\end{remark}

\begin{remark}
  \label{mod--def--complete-exhaustive}
  By the same reasoning as in \cref{ev--def--complete,ev--def--exhaustive,ev--exh--underlying}, for any $\E_\infty$-ring $A$ and $A$-module $M$:
  \begin{enumerate}
  \item the even filtration $\fil^\star_{\ev/A}(M)$ is complete;
  \item the fiber of the map $\fil^n_{\ev/A}(M) \to \colim(\fil^\star_{\ev/A}(M))$ is $(2n-3)$-truncated;
  \item there is a natural equivalence $\colim(\fil^\star_{\ev/A}(M)) \iso \lim_{A \to B,B\in\CAlg^\ev} M \otimes_A B$, in particular a natural map $M \to \colim(\fil^\star_{\ev/A}(M))$.
  \end{enumerate}
\end{remark}

\begin{remark}
  \label{mod--def--monoidal}
  The functor $\tau_{\ge 2\star} : \Spt \to \FilSpt$ has a canonical lax symmetric monoidal structure, from which the even filtration functor $\fil^\star_{\ev/(-)}$ obtains the same. It follows that, for $A$ an $\E_\infty$-ring and $A'$ an $\mathbb{E}_{\infty}$-$A$-algebra, $\fil^{\star}_{\ev/A}(A')$ is canonically a filtered $\mathbb{E}_{\infty}$-ring. For example, in the case $A'=A$, it follows from \cref{mod--def--limit-formula} that there is a canonical identification of filtered $\mathbb{E}_{\infty}$-rings
  \[
    \fil^\star_{\ev/A}(A) \iso \fil^\star_\ev(A).
  \]
  More generally, for $A$ an $\E_\infty$-ring, $\fil^{\star}_{\ev/A}$ lifts to a functor
  \[
    \fil^{\star}_{\ev/A}: \Mod_A \to \FilMod_{\fil^\star_\ev A}.
  \]
\end{remark}

\begin{variant}
  \label{mod--def--fil-ev-p}
  Let $\Mod_p$ denote the full subcategory of $\Mod$ spanned by the pairs $(A,M)$ where $A$ and $M$ are both $p$-complete, and let $\Mod_p^{\ev}$ denote the full subcategory of $\Mod^\ev$ spanned by the pairs $(A,M)$ where $A \in \CAlg_p^\ev$, i.e. $A$ is even and has bounded $p$-power torsion. We define  
  \begin{align*}
    \fil^\star_{\ev/(-),p} &: \Mod_p \to \FilSpt
  \end{align*}
  to be the right Kan extension of
  \begin{align*}
    \tau_{\ge 2\star}(U_{\Mod}) &: \Mod_p^\ev \to \FilSpt
  \end{align*}
  along the inclusion $\Mod_p^{\ev} \subseteq \Mod_p$ (its existence following from \cref{mod--def--accessible,ev--def--p-accessible}).
\end{variant}

\begin{variant}
  \label{mod--def--fil-ev-cir}
  We have a full subcategory $(\Mod^\ev)^{\Bcir}$ of $\Mod^{\Bcir}$ (resp. $(\Mod^\ev_p)^{\Bcir}$ of $\Mod_p^\Bcir = (\Mod_p)^{\Bcir}$) spanned by pairs $(A,M)$ where $A \in \CAlg^\ev$ (resp. $A \in \CAlg_p^\ev$). We define
  \begin{align*}
    &\fil^\star_{\ev/(-),\h\cir} : \Mod^\Bcir \to \FilSpt,
    &\fil^{\filledsquare}_{+}\fil^\star_{\ev/(-),\h\cir} : \Mod^\Bcir \to \BiFilSpt, \\
    &\fil^\star_{\ev/(-),\tate\cir} : \Mod^\Bcir \to \FilSpt, &\fil^{\filledsquare}_+\fil^\star_{\ev/(-),\tate\cir} : \Mod^\Bcir \to \BiFilSpt
  \end{align*}
  to be the right Kan extensions of
  \begin{align*}
    &\tau_{\ge 2\star}(U_{\Mod}^{\h\cir}) : (\Mod^\ev)^\Bcir \to \FilSpt,
    &\tau_{\ge 2\star}((\tau_{\ge \filledsquare}(U_{\Mod}))^{\h\cir}) : (\Mod^\ev)^\Bcir \to \BiFilSpt, \\
    &\tau_{\ge 2\star}(U_{\Mod}^{\tate\cir}): (\Mod^\ev)^\Bcir \to \FilSpt,
    &\tau_{\ge 2\star}((\tau_{\ge \filledsquare}(U_{\Mod}))^{\tate\cir}): (\Mod^\ev)^\Bcir \to \BiFilSpt
  \end{align*}
  along the inclusions $(\Mod^\ev)^\Bcir \subseteq \Mod^\Bcir$ (their existence following from \cref{mod--def--accessible}), and we define
  \begin{align*}
    &\fil^\star_{\ev/(-),p,\h\cir} : \Mod_p^\Bcir \to \FilSpt,
    &\fil^{\filledsquare}_+\fil^\star_{\ev/(-),p,\h\cir} : \Mod_p^\Bcir \to \BiFilSpt, \\ &\fil^\star_{\ev/(-),p,\tate\cir} : \Mod_p^\Bcir \to \FilSpt,
    &\fil^{\filledsquare}_+\fil^\star_{\ev/(-),p,\tate\cir} : \Mod_p^\Bcir \to \BiFilSpt,
  \end{align*}
  to be the right Kan extensions of
  \begin{align*}
    &\tau_{\ge 2\star}(U_{\Mod}^{\h\cir}) : (\Mod^\ev_p)^\Bcir \to \FilSpt,
    &\tau_{\ge 2\star}((\tau_{\ge \filledsquare}(U_{\Mod}))^{\h\cir}) : (\Mod^\ev_p)^\Bcir \to \BiFilSpt, \\
    &\tau_{\ge 2\star}(U_{\Mod}^{\tate\cir}) : (\Mod^\ev_p)^\Bcir \to \FilSpt,
    &\tau_{\ge 2\star}((\tau_{\ge \filledsquare}(U_{\Mod}))^{\tate\cir})) : (\Mod^\ev_p)^\Bcir \to \BiFilSpt
  \end{align*}
  along the inclusion $(\Mod^\ev_p)^\Bcir \subseteq \Mod_p^\Bcir$ (their existence following from \cref{mod--def--accessible,ev--def--p-accessible}).
\end{variant}

\begin{variant}
  \label{mod--def--fil-ev-cyc}
  Let $\CycMod_p$ denote the category of pairs $(A,M)$ where $A$ is a $p$-typical cyclotomic $\E_\infty$-ring and $M$ is an $A$-module in $\CycSpt_p$, and let $\CycMod^\ev_p$ denote the full subcategory of $\CycMod_p$ spanned by the pairs $(A,M)$ where $A \in \CAlg_p^\ev$. Recall that for $(A,M) \in \CycMod_p$, both $A$ and $M$ are bounded below and $p$-complete by convention. In particular, for $(A,M) \in \CycMod_p$, the Tate spectrum $\smash{M^{\tate\cir}}$ is $p$-complete (see \cite[\textsection 2.3]{BMS}), and hence, by \cite{NikolausScholze}, there is a natural Frobenius map $\smash{\varphi : M^{\h\cir} \to M^{\tate\cir}}$, together with a natural identification $\TC(M) \iso \fib(\varphi - \can : M^{\h\cir} \to M^{\tate\cir})$.

  We define
  \[
    \fil^\star_{\ev/(-),p,\TC}(-) : \CycMod_p \to \FilSpt,
  \]
  to be the right Kan extension of
  \[
    \fib(\varphi-\can : \tau_{\ge 2\star}(U_{\Mod}^{\h\cir}) \to \tau_{\ge 2\star}(U_{\Mod}^{\tate\cir})) : \CycMod_p^\ev \to \FilSpt,
  \]
  along the inclusion $\CycMod_p^\ev \subseteq \CycMod_p$ (its existence following from \cref{mod--def--accessible,ev--def--p-accessible}).
\end{variant}


\subsection{Descent properties of the filtrations}
\label{mod--desc}

\begin{definition}
  \label{mod--desc--free-topology-def}
  We say that a sieve on $(A,M) \in (\Mod^{\ev})^{\mathrm{op}}$ is a \emph{free covering sieve} if it contains a finite collection of maps $\{(A,M)\to (B_i,M_i)\}_{1\le i \le n}$ such that $\prod_i B_i$ is equivalent as an $A$-module to a direct sum of even shifts of $A$, which is nonzero if $A$ is nonzero, and each of the morphisms $(A,M) \to (B_i,M_i)$ induces an equivalence $M\otimes_AB_i \simeq M_i$.

  Similarly, we say that a sieve on $(A,M) \in (\Mod^{\ev}_p)^{\mathrm{op}}$ is a \emph{$p$-completely free covering sieve} if it contains a finite collection of maps $\{(A,M)\to (B_i,M_i)\}_{1\le i \le n}$ such that $\prod_i B_i$ is equivalent as an $A$-module to the $p$-completion of a direct sum of even shifts of $A$, which is nonzero if $A$ is nonzero, and each of the morphisms $(A,M) \to (B_i,M_i)$ induces an equivalence $\cpl{(M\otimes_A B_i)}_p \simeq M_i$.
\end{definition}

\begin{proposition}
  \label{mod--desc--free-topology-prop}
  The free covering sieves of \cref{mod--desc--free-topology-def} define a Grothendieck topology on $(\Mod^\ev)^\mathrm{op}$ and the $p$-completely free covering sieves define a Grothendieck topology on $(\Mod^\ev_p)^\mathrm{op}$.

  For a category $\mathcal{C}$ admitting small limits, a functor $F: \Mod^\ev \to \mathcal{C}$ is a sheaf for the former topology if and only if the following conditions are satisfied:
  \begin{enumerate}
  \item $F$ preserves finite products.
  \item For every nonzero, even $\E_\infty$-ring $A$ and every map of even $\E_\infty$-rings $A \to B$ that exhibits $B$ as a free $A$-module,
    the map
    \[
      F(A,M) \to \lim_{\Delta} F(B^{\otimes_A\bullet+1},M\otimes_AB^{\otimes_A\bullet+1})
    \]
    is an equivalence.
  \end{enumerate}
  The analogous claim holds for the latter topology as well.
\end{proposition}

\begin{proof}
  As with \cref{ev--desc--flat-topology-prop}, the proofs of \cite[A.3.2.1, A.3.3.1]{sag} go through.
\end{proof}

\begin{definition}
  \label{mod--desc--free-topology-name}
  We refer to the Grothendieck topologies of \cref{mod--desc--free-topology-prop} as the \emph{free topology} on $\Mod^\ev$ and the \emph{$p$-completely free topology} on $\Mod^\ev_p$.

  Since pushouts in $(\Mod^\ev)^\Bcir$ (resp. $(\Mod^\ev_p)^\Bcir$ and $\CycMod^\ev_p$) are
  computed in $\Mod^\ev$ (resp. $\Mod_p^\ev$), the above induce topologies on $\smash{(\Mod^\ev)^\Bcir}$, $\smash{(\Mod^\ev_p)^\Bcir}$, and $\CycMod^\ev_p$, which we call by the same names.
\end{definition}

\begin{lemma}
  \label{mod--desc--desc}
  \begin{enumerate}[leftmargin=*]
  \item \label{mod--desc--desc--plain}
    The functor $\pi_*(U_{\Mod}): \Mod^\ev \to \GrSpt$ is a sheaf for the free topology, and it restricts to a sheaf for the $p$-completely free topology on $\Mod^\ev_p$.
  \item \label{mod--desc--desc--ht}
    The functors $\pi_*(U_\Mod^{\h\cir}), \pi_*(U_\Mod^{\tate\cir}) : (\Mod^\ev)^{\Bcir} \to \GrSpt$ are sheaves for the free topology, and they restrict to sheaves for the $p$-completely free topology on $(\Mod^\ev_p)^{\Bcir}$.
  \item \label{mod--desc--desc--ht-stupid}
    The functors $\pi_{*}(\Sigma^\filledsquare(\pi_{\filledsquare}(U_\Mod))^{\h\cir}), \pi_{*}(\Sigma^\filledsquare(\pi_\filledsquare(U_\Mod))^{\tate\cir}) : (\Mod^\ev)^{\Bcir} \to \BiGrSpt$
    are sheaves for the free topology, and they restrict to sheaves for the $p$-completely free topology on $(\Mod^\ev_p)^{\Bcir}$.
  \end{enumerate}
\end{lemma}

\begin{proof}
  We will prove the $p$-complete statements; the integral statements can be addressed similarly. We begin by proving \cref{mod--desc--desc--plain}. Let $A \to B$ be a $p$-completely free cover in $\CAlg^\ev_p$ and let $M$ be a $p$-complete $A$-module. We need to prove that the canonical map
  \[
    \pi_*(M) \to 
    \lim_\Delta {\pi_*({\cpl{(M\otimes_A B^{\otimes_A\bullet+1})}_p})} \iso \lim_\Delta {\cpl{(\pi_*(M)\otimes^{\L}_{\pi_*(A)} (\pi_*(B))^{\otimes^{\L}_{\pi_*(A)}\bullet+1}))}_p}
  \]
  is an equivalence (where the limit is taken in $\GrSpt$ and the identification can be seen by arguing as in the proof of \cref{ev--desc--pcpl-flat-homotopy}). Both sides being $p$-complete, it suffices to show that it is a limit diagram after after derived base change along $\Z \to \Z/p$. The $p$-complete freeness hypothesis implies that
  \[
    \pi_{2*}(A) \otimes^\L_\Z \Z/p \to
    \pi_{2*}(B) \otimes^\L_\Z \Z/p
  \]
  is faithfully flat in the sense of \cite[Definition D.4.4.1]{sag} (after forgetting the gradings), so the claim follows from faithfully flat descent (\cite[Theorem D.6.3.5]{sag}).
  
  Let us now prove \cref{mod--desc--desc--ht}. Let $A \to B$ be a $p$-completely free cover in $(\CAlg_p^\ev)^{\Bcir}$ and let $M$ be a $p$-complete $\cir$-equivariant $A$-module. We need to prove that the canonical maps
  \[
    \pi_*(M^{\h\cir}) \to 
    \lim_{\Delta} {\pi_*((\cpl{(M\otimes_A B^{\otimes_A\bullet+1})}_p)^{\h\cir})},
    \ \ 
    \pi_*(M^{\tate\cir}) \to 
    \lim_{\Delta} {\pi_*((\cpl{(M\otimes_A B^{\otimes_A\bullet+1})}_p)^{\tate\cir})},
  \]
  are equivalences. We first treat the $(-)^{\h\cir}$ case. Choose $t \in \pi_{-2}(A^{\h\cir})$ as in \cref{ev--desc--orientation}\cref{ev--desc--orientation--class}. Using \cref{ev--desc--orientation}\cref{ev--desc--orientation--equivalences} and arguing as in \cref{ev--desc--pcpl-flat-homotopy}, we have that
  \begin{align*}
    \pi_*((\cpl{(M\otimes_A B^{\otimes_A\bullet+1})}_p)^{\h\cir})
    &\iso \pi_*(\cpl{(M^{\h\cir}\otimes_{A^{\h\cir}} (B^{\h\cir})^{\otimes_{(A^{\h\cir})}\bullet+1})}_{(p,t)}) \\
    &\iso \cpl{(\pi_*(M^{\h\cir})\otimes^{\L}_{\pi_*(A^{\h\cir})}
      \pi_*(B^{\h\cir})^{\otimes^{\L}_{\pi_*(A^{\h\cir})}\bullet+1})}_{(p,t)}.
  \end{align*}
  By $(p,t)$-completeness,
  it suffices to prove the claim after taking the derived base change along $\Z[t] \to \Z/p$, and then we may conclude by faithfully flat descent as in \cref{ev--desc--desc--plain} above. The $(-)^{\tate\cir}$ case then follows from the $(-)^{\h\cir}$ case by inverting $t$, as in the proof of \cref{ev--desc--desc}.

  Finally, \cref{ev--desc--desc--ht-stupid} can be proved by an argument very similar to the one just used to prove \cref{ev--desc--desc--ht}.
\end{proof}

\begin{theorem}
  \label{mod--desc--fil-desc}
  \begin{enumerate}[leftmargin=*]
  \item \label{mod--desc--fil-desc--plain}
    The functor $\tau_{\ge 2\star}(U_{\Mod}) : \Mod^\ev \to \FilSpt$ is a sheaf for the free topology, and it restricts to a sheaf for the $p$-completely free topology on $\Mod^\ev_p$.
  \item \label{mod--desc--fil-desc--ht}
    The functors $\tau_{\ge 2\star}(U_{\Mod}^{\h\cir}), \tau_{\ge 2\star}(U_{\Mod}^{\tate\cir}) : (\Mod^\ev)^{\Bcir} \to \FilSpt$ are sheaves for the free topology, and they restrict to sheaves for the $p$-completely free topology on $(\Mod^\ev_p)^{\Bcir}$.
  \item \label{mod--desc--fil-desc--nygaard}
    The functors $\tau_{\ge 2\star}((\tau_{\ge \filledsquare}(U_{\Mod}))^{\h\cir}), \tau_{\ge 2\star}((\tau_{\ge \filledsquare}(U_{\Mod}))^{\tate\cir}) : (\Mod^\ev)^{\Bcir} \to \BiFilSpt$ are sheaves for the free topology, and they restrict to sheaves for the $p$-completely free topology on $(\Mod^\ev_p)^{\Bcir}$.
  \item \label{mod--desc--fil-desc--tc}
    The functor
    \[
      \fib(\varphi-\can : \tau_{\ge 2\star}(U_{\Mod}^{\h\cir}) \to \tau_{\ge 2\star}(U_{\Mod}^{\tate\cir})) : \CycMod_p^\ev \to \FilSpt
    \]
    is a sheaf for the $p$-completely free topology.
  \end{enumerate}
\end{theorem}

\begin{proof}
  It suffices to prove these claims with $\tau_{\ge\star}$ in place of $\tau_{\ge 2\star}$, as doubling the speed of a filtration preserves limits. With this replacement made, we may proceed as in the proof of \cref{ev--desc--fil-desc}: we pass to associated graded objects, and then apply \cref{mod--desc--desc}.
\end{proof}

\begin{corollary}
  \label{mod--desc--eff-desc}
  \begin{enumerate}[leftmargin=*]
  \item \label{mod--desc--eff-desc--plain}
    For $A \to B$ an evenly free map of $\E_\infty$-rings and $M$ an $A$-module, the canonical map
    \[
      \fil^\star_{\ev/A}(M) \to \lim_{\Delta} {\fil^\star_{\ev/ B^{\otimes_A\bullet+1}}(M\otimes_AB^{\otimes_A\bullet+1})}
    \]
    is an equivalence. For $A \to B$ a $p$-completely evenly free map of $p$-complete $\E_\infty$-rings and $M$ a $p$-complete $A$-module, the canonical map
    \[
      \fil^\star_{\ev/A,p}(M) \to \lim_{\Delta} {\fil^\star_{\ev/
          \cpl{(B^{\otimes_A\bullet+1})}_p,p}(\cpl{(M\otimes_AB^{\otimes_A\bullet+1})}_p)}
    \]
    is an equivalence.
  \item \label{mod--desc--eff-desc--ht}
    For $A \to B$ an evenly free map of $\E_\infty$-rings with $\cir$-action and $M$ an $\cir$-equivariant module, the canonical maps
    \begin{align*}
     &\fil^\star_{\ev/A,\h\cir}(M) \to \lim_{\Delta} {\fil^\star_{\ev/
     B^{\otimes_A\bullet+1},\h\cir}(M\otimes_AB^{\otimes_A\bullet+1})}, \\
      &\fil^\star_{\ev/A,\tate\cir}(M) \to \lim_{\Delta} {\fil^\star_{\ev/
       B^{\otimes_A\bullet+1},\tate\cir}(M\otimes_AB^{\otimes_A\bullet+1})}
    \end{align*}
    are equivalences. For $A \to B$ a $p$-completely evenly free map of $p$-complete $\E_\infty$-rings with $\cir$-action and $M$ a $p$-complete $\cir$-equivariant module, the canonical maps
    \begin{align*}
     &\fil^\star_{\ev/A,p,\h\cir}(M) \to \lim_{\Delta} {\fil^\star_{\ev/
       \cpl{(B^{\otimes_A\bullet+1})}_p,p,\h\cir}(\cpl{(M\otimes_AB^{\otimes_A\bullet+1})}_p)},
       \\
      &\fil^\star_{\ev/A,p,\tate\cir}(M) \to \lim_{\Delta} {\fil^\star_{\ev/
        \cpl{(B^{\otimes_A\bullet+1})}_p,p,\tate\cir}(\cpl{(M\otimes_AB^{\otimes_A\bullet+1})}_p)},
      \end{align*}
    are equivalences.
  \item \label{mod--desc--eff-desc--tc}
    For $A \to B$ a $p$-completely evenly free map of $p$-typical cyclotomic $\E_\infty$-rings and $M$ a cyclotomic $A$-module, the canonical map
    \[
      \fil^\star_{\ev/A,p,\TC}(M) \to \lim_{\Delta} {\fil^\star_{\ev/
       \cpl{(B^{\otimes_A\bullet+1})}_p,p,\TC}(\cpl{(M\otimes_AB^{\otimes_A\bullet+1})}_p)}
    \]
    is an equivalence.
  \end{enumerate}
\end{corollary}

From \cref{mod--desc--eff-desc} we obtain \cref{mod--desc--exhaustivity,mod--desc--compute-tc-as-fiber,mod--desc--novikov} below, just as \cref{ev--exh--main,ev--desc--compute-tc-as-fiber,ev--desc--novikov} were obtained from \cref{ev--desc--eff-desc}.

\begin{definition}
  \label{mod--desc--descent}
  For $A \to B$ a map of $\E_\infty$-rings and $M$ an $A$-module, we say that $A \to B$ \emph{satisfies descent for $M$} (resp. \emph{satisfies $p$-complete descent for $M$}) if the canonical map $M \to \lim_\Delta M \otimes_A B^{\otimes_A \bullet+1}$ (resp. the canonical map $\smash{\cpl{M}_p \to \lim_\Delta \cpl{(M \otimes_A B^{\otimes_A \bullet+1})}_p}$) is an equivalence.

  For $A \to B$ a map of $\cir$-equivariant $\E_\infty$-rings and $M$ an $\cir$-equivariant $A$-module, we say that $A \to B$ \emph{satisfies descent for $M$} (resp. \emph{satisfies $p$-complete descent for $M$}) if the underlying map of $\E_\infty$-rings does, and we say that it \emph{satisfies Tate descent for $M$} (resp. \emph{satisfies $p$-complete Tate descent for $M$}) if the canonical map $M^{\tate\cir} \to \lim_\Delta (M \otimes_A B^{\otimes_A\bullet+1})^{\tate\cir}$ (resp. the canonical map $(\cpl{M}_p)^{\tate\cir} \to \lim_\Delta (\cpl{(M \otimes_A B^{\otimes_A\bullet+1})}_p)^{\tate\cir}$) is an equivalence.
\end{definition}

\begin{remark}
  \label{mod--desc--descent-examples}
  Let $A \to B$ be a map of connective $\E_\infty$-rings and let $M$ be a connective $A$-module. If $A \to B$ is $1$-connective, or faithfully flat in the sense of \cite[Definition D.4.4.1]{sag}, or a composition of such maps, then $A \to B$ satisfies descent for $M$ and $p$-complete descent for $M$; and the same can be said for Tate descent and $p$-complete Tate descent, when $A \to B$ and $M$ are $\cir$-equivariant. The arguments of \cref{ev--exh--connective-descent,ev--exh--ff-descent,ev--exh--universal} may be adapted straightforwardly to prove these statements.
\end{remark}

\begin{corollary}
  \label{mod--desc--exhaustivity}
  \begin{enumerate}[leftmargin=*]
  \item \label{mod--desc--exhaustivity--plain}
    Let $A$ be an $\E_\infty$-ring and let $M$ be an $A$-module. Suppose that there exists a map of $\E_\infty$-rings $A \to B$ that is eff and satisfies descent for $M$ (resp. $p$-completely eff and satisfies $p$-complete descent for $M$) and where $B$ is even (resp. $\cpl{B}_p$ is even and has bounded $p$-power torsion). Then the map $M \to \colim(\fil^\star_{\ev/A}(M))$ (resp. the map $\cpl{M}_p \to \colim(\fil^\star_{\ev/A,p}(\cpl{M}_p))$) is an equivalence.
  \item \label{mod--desc--exhaustivity--h}
    Let $A$ be an $\cir$-equivariant $\E_\infty$-ring and let $M$ be an $\cir$-equivariant $A$-module. Suppose that there exists a map of $\cir$-equivariant $\E_\infty$-rings $A \to B$ that is eff and satisfies descent for $M$ (resp. $p$-completely eff and satisfies $p$-complete descent for $M$) and where $B$ is even (resp. $\cpl{B}_p$ is even and has bounded $p$-power torsion). Then the map $\smash{M^{\h\cir} \to \colim(\fil^\star_{\ev/A,\h\cir}(M))}$ (resp. the map $\smash{(\cpl{M}_p)^{\h\cir} \to \colim(\fil^\star_{\ev/A,p,\h\cir}(\cpl{M}_p))}$) is an equivalence.
  \item \label{mod--desc--exhaustivity--t}
    Let $A$ be an $\cir$-equivariant $\E_\infty$-ring and let $M$ be an $\cir$-equivariant $A$-module. Suppose that there exists a map of $\cir$-equivariant $\E_\infty$-rings $A \to B$ that is eff and satisfies Tate descent for $M$ (resp. $p$-completely eff and satisfies $p$-complete Tate descent for $M$) and where $B$ is even (resp. $\cpl{B}_p$ is even and has bounded $p$-power torsion). Then the map $\smash{M^{\tate\cir} \to \colim(\fil^\star_{\ev/A,\tate\cir}(M))}$ (resp. the map $\smash{(\cpl{M}_p)^{\tate\cir} \to \colim(\fil^\star_{\ev/A,p,\tate\cir}(\cpl{M}_p))}$) is an equivalence.
  \end{enumerate}
\end{corollary}

\begin{corollary}
  \label{mod--desc--compute-tc-as-fiber}
  Let $A$ be a $p$-typical cyclotomic $\mathbb{E}_{\infty}$-ring, and suppose that there exists a $p$-completely evenly free map of such $A \to B$ where $B$ is even and has bounded $p$-power torsion. Let $M$ be an $A$-module in $\CycSpt_p$. Then the cyclotomic Frobenius and canonical maps $\varphi, \can : M^{\h\cir} \to M^{\tate\cir}$ refine to maps $\varphi,\can: \fil^{\star}_{\ev/A,p,\h\cir}(M) \to \fil^{\star}_{\ev/A,p,\tate\cir}(M)$, naturally in $(A,M)$, and there is a canonical equivalence
  \[
    \fil^\star_{\ev/A,p,\TC}(M) \simeq
    \fib(\varphi - \can: \fil^{\star}_{\ev/A,p,\h\cir}(M) \to 
    \fil^{\star}_{\ev/A,p,\tate\cir}(M)).
  \]
\end{corollary}

\begin{corollary}[Novikov descent]
  \label{mod--desc--novikov}
  The following statements hold:
  \begin{enumerate}[leftmargin=*]
    \item For $A$ an $\E_\infty$-ring and $M$ an $A$-module, the canonical map
    \[
      \fil^\star_{\ev/A}(M) \to \lim_{\Delta} {\fil^\star_{\ev/A \otimes \MU^{\bullet+1}}(M \otimes \MU^{\otimes \bullet+1})}
    \]
    is an equivalence. For $A$ a $p$-complete $\E_\infty$-ring and $M$ a $p$-complete $A$-module, the canonical map
    \[
      \fil^\star_{\ev/A,p}(M) \to \lim_{\Delta} {\fil^\star_{\ev/\cpl{(A \otimes \MU^{\bullet+1})}_p,p}(\cpl{(M \otimes \MU^{\otimes \bullet+1})}_p)}
    \]
    is an equivalence.
    \item For $A$ an $\E_\infty$-ring with $\cir$-action and $M$ an $\cir$-equivariant $A$-module, the canonical maps
    \begin{align*}
      &\fil^\star_{\ev/A,\h\cir}(M) \to \lim_{\Delta} {\fil^\star_{\ev/A \otimes \MU^{\otimes \bullet+1},\h\cir}(M \otimes \MU^{\otimes \bullet+1})}, \\
      &\fil^\star_{\ev/A,\tate\cir}(M) \to \lim_{\Delta} {\fil^\star_{\ev/A \otimes \MU^{\otimes \bullet+1},\tate\cir}(M \otimes \MU^{\otimes \bullet+1})}
    \end{align*}
    are equivalences. For $A$ a $p$-complete $\E_\infty$-ring with $\cir$-action and $M$ a $p$-complete $\cir$-equivariant $A$-module, the canonical maps
    \begin{align*}
      &\fil^\star_{\ev/A,p,\h\cir}(M) \to \lim_{\Delta} {\fil^\star_{\ev/\cpl{(A \otimes \MU^{\otimes \bullet+1})}_p,p,\h\cir}(\cpl{(M \otimes \MU^{\otimes \bullet+1})}_p)},\\
      &\fil^\star_{\ev/A,p,\tate\cir}(M) \to \lim_{\Delta} {\fil^\star_{\ev/\cpl{(A \otimes \MU^{\otimes \bullet+1})}_p,p,\tate\cir}(\cpl{(M \otimes \MU^{\otimes \bullet+1})}_p)}
    \end{align*}
    are equivalences, where $\MU$ is considered to have trivial $\cir$-action.
    \item For $A$ a $p$-typical cyclotomic $\mathbb{E}_{\infty}$-ring and $M$ a cyclotomic $A$-module, the canonical map
    \[
      \fil^\star_{\ev/A,p,\TC}(M) \to \lim_{\Delta} {\fil^\star_{\ev/\cpl{(A \otimes \MU^{\otimes \bullet+1})}_p,p,\TC}(\cpl{(M \otimes \MU^{\otimes \bullet+1})}_p)}
    \]
    is an equivalence, where $\MU$ is considered to have trivial cyclotomic structure.
  \end{enumerate}
\end{corollary}

 \begin{example} \label{example:BPn} Consider any $\mathbb{E}_1$-$\MU$-algebra $R$, and note that $\THH(R)$ is a $\THH(\MU)$-module.
In \cref{ell-example}, it is observed that $\THH(\MU) \to \MU$ is evenly free.  It follows that 
 \[\fil^{\star}_{\ev/\THH(\MU)}\THH(R) \iso \lim_{\Delta}\left( \tau_{\ge 2\star}\THH(R/\MU^{\otimes \bullet+1}) \right).\] 
When $R=\mathrm{BP}\langle n\rangle$, this filtration was studied by the first and third authors in \cite[\textsection 6.1]{HahnWilson}.
\end{example}